\documentclass{amsart}%
\usepackage{amsfonts}
\usepackage{amsmath}
\usepackage{amssymb}
\usepackage{graphicx}
\usepackage[matrix,arrow,curve]{xy}%
\providecommand{\U}[1]{\protect\rule{.1in}{.1in}}
\newtheorem{theorem}{Theorem}
\theoremstyle{plain}

\newtheorem{corollary}[theorem]{Corollary}

\newtheorem{example}[theorem]{Example}

\newtheorem{lemma}[theorem]{Lemma}

\newtheorem{proposition}[theorem]{Proposition}
\newtheorem{remark}[theorem]{Remark}

\numberwithin{equation}{section}
\numberwithin{theorem}{section}
\begin{document}

\title[Topological radicals, II.]{Topological radicals, II. Applications to spectral theory of multiplication operators}
\author{Victor S. Shulman}

\address{Department of Mathematics, Vologda State Technical University, 15 Lenina str.,
Vologda 160000, Russian Federation}
\email{shulman.victor80@gmail.com}

\author{Yurii V. Turovskii}

\address{Institute of Mathematics and Mechanics, National Academy of Sciences of
Azerbaijan, 9 F. Agayev Street, Baku AZ1141, Azerbaijan}
\email{yuri.turovskii@gmail.com}
\thanks{This paper is in final form and no version of it will be submitted for
publication elsewhere.}

\begin{abstract}
We develop the tensor spectral radius technique and the theory of the tensor
radical. Basing on them we obtain several results on spectra of multiplication
operators on Banach bimodules and indicate some applications to the spectral
theory of elementary and multiplication operators on Banach algebras and
modules with various compactness properties.

\end{abstract}

\maketitle
\tableofcontents

\section{ Introduction\label{s1}}

The localization of spectrum of an elementary operator in terms of spectra of
its coefficients is one of the most popular subjects in the theory of
elementary operators. The strongest results in this area were obtained for
operators with commutative coefficient families because this allows one to use
the theory of joint spectra (see \cite{Curto}).

Here we consider the less restrictive conditions than commutativity. For
instance it is not known for us whether an operator $Tx=\sum_{k=1}^{n}%
a_{k}xb_{k}$ on a Banach algebra $A$ is quasinilpotent if $a_{1},...,a_{n}$
belong to a radical closed subalgebra of $A$. However, if all $a_{i}$ are
compact operators, the answer is positive (see for example \cite[Lemma
5.10]{ST2005}) and may be obtained by using the joint spectral radius
technique. We consider multiplication operators of more general type than
elementary ones as well as more general classes of coefficient algebras than
algebras of compact operators. As a main technical tool we present the theory
of tensor spectral radius initiated in \cite{TR1} in the framework of the
general theory of topological radicals. Basing on it we obtain several results
on spectra of multiplication operators on Banach bimodules and indicate their
applications to spectral theory of elementary operators on Banach algebras
with various compactness properties.

Recall that an element $a$ of a normed algebra $A$ is called \textit{compact}
if the elementary operator $x\longmapsto axa$ on $A$ is compact. The reason
for such a definition is a well known theorem of Vala \cite{Vala} which states
that \textit{a bounded operator on a Banach space }$X$\textit{ is compact iff
it is a compact element of the algebra }$\mathcal{B}(X)$\textit{ of all
bounded operators on }$X$\textit{.}

If all elements $a\in A$ are compact then $A$ is called \textit{compact}. If,
more strongly, for all $a,b\in A$, the operator $x\longmapsto axb$ on $A$ is
compact then $A$ is called \textit{bicompact}. A less restrictive condition is
that $A$ is generated as a normed algebra by the semigroup of all its compact
elements. The most wide class of algebras of this kind is the class of
hypocompact algebras. A normed algebra $A$ is called \textit{hypocompact} if
each non-zero quotient of $A$ by a closed ideal has a non-zero compact
element. One may realize a hypocompact algebra as a result of a transfinite
sequence of extensions of bicompact algebras. This class has some resemblance
with the class of $\mathrm{GCR}$-algebras in the $C^{\ast}$-algebras. Note for
example that \textit{the image of each strictly irreducible representation of
a hypocompact Banach algebra contains a non-zero finite rank operator.}

We show that elementary operators on hypocompact Banach algebras commutative
modulo the Jacobson radical are \textit{spectrally computable}, that is
\[
\sigma(T+S)\subset\sigma(T)+\sigma(S)\text{ and }\sigma(TS)\subset
\sigma(T)\sigma(S)
\]
for all elementary operators $T,S$. Moreover, if all operators $L_{a}-R_{a}$
are quasinilpotent on $A$ (we call such algebras \textit{Engel}) then the
spectra of elementary operators satisfy the inclusion
\[
\sigma(\sum_{k}L_{a_{k}}R_{b_{k}})\subset\sigma(\sum_{k}a_{k}b_{k}).
\]
Among other applications we mention results on the structure of closed ideals
in a radical compact Banach algebra. We prove that if such an algebra is
infinite dimensional then it has infinite chains of ideals. As a consequence,
we get that there is an infinite chain of closed operator ideals in the sense
of Pietsch \cite{P78} intermediate between the ideals of approximable and
compact operators.

In the last section we consider the applications of the theory to spectral
subspaces of multiplication operators. In 1978 Wojty\'{n}ski, working on the
problem of the existence of a closed two-sided ideal in a radical Banach
algebra, proved the following result on linear operator equations with compact coefficients.

\begin{lemma}
\label{WW}\cite{W78} Let all coefficients $a,b,a_{i},b_{i}$ of the linear
operator equation
\begin{equation}
ax+xb+\sum_{i=1}^{n}a_{i}xb_{i}=\lambda x \label{wojt}%
\end{equation}
be compact operators on a Banach space $X$. If $\lambda\neq0$ then each
bounded solution $x$ of \emph{(\ref{wojt})} is a nuclear operator.
\end{lemma}

The presence of nuclear operators gives a possibility to use trace for proving
the quasinilpotence of some multiplication operators. Using this,
Wojty\'{n}ski proved in \cite{W78} that \textit{every radical Banach algebra
having non-zero compact elements is not topologically simple }(if dimension of
the algebra is larger than $1$). In \cite{W78-2} he applied the same argument
to Banach Lie algebras and proved that \textit{if all adjoint operators of a
Banach Lie algebra }$\mathfrak{L}$\textit{ are compact and quasinilpotent,
then} $\mathfrak{L}$ \textit{has a non-trivial closed Lie ideal}. Both results
are now obtained in a more general setting with using another technique
\cite{T1998, ST2005}, but Wojty\'{n}ski's approach itself is interesting and
still helpful.

Several years after \cite{W78} Fong and Radjavi \cite{FR} considered a more
general class of equations
\begin{equation}
\sum a_{i}xb_{i}=\lambda x, \label{FoRa}%
\end{equation}
where the sum is finite and for each $i$ at least one of operators
$a_{i},b_{i}$ is compact. They worked only in the case of Hilbert space
operators but proved much more, namely that all solutions of (\ref{FoRa})
belong to each Shatten class $C_{p}$, $p>0$. On the other hand, they showed
that solutions of (\ref{FoRa}) are not necessarily finite rank operators: each
operator $x$ whose singular numbers decrease more quickly than every geometric
progression is a solution of an equation of the form (\ref{FoRa}).

We will show here that the main results of Wojty\'{n}ski and of Fong and
Radjavi extend to multiplication operators with infinite number of summands.
Furthermore, we will see that not only eigenspaces with non-zero eigenvalues
consist of nuclear operators but that the same holds for spectral subspaces
corresponding to components of spectra non-containing $0$ (or stronger, for
all invariant subspaces on which the operator is surjective). Moreover, the
ideal of nuclear operator here can be changed by any quasi-Banach ideal in the
case of elementary operators. (i.e. when the number of summands is finite).

We prove also that the results extend to \textquotedblleft integral
multiplication operators\textquotedblright. In particular, if an operator $x$
satisfies the condition
\[
\int_{\alpha}^{\beta}(a(t)xu(t)+v(t)xb(t))d\mu=\lambda x\text{\quad with\quad
}\lambda\neq0,
\]
where $a(t)$ and $b(t)$ are continuous operator valued functions, $u(t)$ and
$v(t)$ are continuous compact operator valued functions, then $x$ is nuclear.

We also extend the results to systems of equations. A simple example is the
following: Let $x_{1},...,x_{n}$ satisfy a system of equations
\[
\sum_{k=1}^{n}a_{ik}x_{k}b_{ik}=\lambda x_{i},\text{ }i=1,...,n,\text{
}\lambda\neq0.
\]
If for each pair $(i,k)$ at least one of operators $a_{ik}$, $b_{ik}$ is
compact then all $x_{i}$ are nuclear (moreover belong to each quasi-Banach
operator ideal of $\mathcal{B}(X)$). Such systems of equations arise, for
example, in the study of subgraded Lie algebras \cite{KST}.

Apart of tensor radical technique our approach is based on a general result
(Theorem \ref{spectrclos}) which is not restricted by multiplication operators
but deals with bounded operators on an ordered pair of Banach spaces.

We would like to express our heartfelt gratitude to Niels Gr\o nb\ae k for a
very helpful discussion of the results of the paper \cite{Wil} and to the
referee for his patient and attentive reading of the manuscript and for
numerous useful suggestions.

\section{Preliminary results\label{s2}}

\subsection{Notation}

All spaces are assumed to be complex. If a normed algebra $A$ is not unital,
denote by $A^{1}$ the normed algebra obtained by adjoining the identity
elementt to $A$, and if $A$ is already unital, let $A^{1}=A$. We denote by
$\widehat{A}$ the completion of $A$. The term \textit{ideal} always means a
two-sided ideal. If $A$ is a normed algebra and $I$ is an ideal of $A$ then
the term $A/\overline{I}$ always denotes the quotient of $A$ by the closure of
$I$ \textit{in the norm of }$A$ (even if $I$ is supplied with its own norm).
It is convenient to write $a/\overline{I}$ for $a+\overline{I}\in
A/\overline{I}$. Also, by a \textit{quotient} of a normed algebra $A$ we
always mean any quotient of $A$ by a closed ideal.

Let $A$ be a normed algebra. A norm $\left\Vert \cdot\right\Vert _{A}$ on $A$
is called an \textit{algebra (}or\textit{ submultiplicative) norm }if
\[
\left\Vert ab\right\Vert _{A}\leq\left\Vert a\right\Vert _{A}\left\Vert
b\right\Vert _{A}%
\]
for all $a,b\in A$. If $A$ has another norm (or seminorm), say $\left\Vert
\cdot\right\Vert $, we write $\left(  A,\left\Vert \cdot\right\Vert \right)  $
to indicate that $A$ is considered with respect to $\left\Vert \cdot
\right\Vert $. The norm $\left\Vert \cdot\right\Vert $ is \textit{equivalent}
to $\left\Vert \cdot\right\Vert _{A}$ on $A$ if there are constants $s,t>0$
such that
\[
s\left\Vert \cdot\right\Vert \leq\left\Vert \cdot\right\Vert _{A}\leq
t\left\Vert \cdot\right\Vert
\]
on $A$. Assume now that $A$ is unital. Then $\left\Vert \cdot\right\Vert _{A}$
is called \textit{unital} if $\left\Vert 1\right\Vert _{A}=1$. It is well
known that every algebra norm on $A$ is equivalent to a unital one.

\subsection{Quasinilpotents and the radical modulo an ideal}

An element $a$ of a normed algebra $A$ is called \textit{quasinilpotent} if
\[
\inf_{n}\left\Vert a^{n}\right\Vert ^{1/n}=0.
\]
Let $Q\left(  A\right)  $ denote the set of all quasinilpotent elements of $A$.

Let $A$ be a normed algebra, and let $\mathrm{dist}_{A}\left(  a,E\right)  $,
or simply $\mathrm{dist}\left(  a,E\right)  $, denote the distance from $a\in
A$ to $E\subset A$, that is
\[
\mathrm{dist}\left(  a,E\right)  =\inf\left\{  \left\Vert a-b\right\Vert :b\in
E\right\}  .
\]
Let $Q_{E}\left(  A\right)  $ be the set of all elements $a\in A$ such that
\[
\inf_{n}\mathrm{dist}_{A}\left(  a^{n},E\right)  ^{1/n}=0.
\]
If $E=J$ is an ideal of $A$, then $\mathrm{dist}_{A}\left(  a,J\right)  $ is
simply a quotient norm of $q\left(  a\right)  $ in the quotient algebra
$A/\overline{J}$, where $\overline{J}$ denotes the closure of $J$ in $A$ and
$q:A\longrightarrow A/\overline{J}$ is the standard quotient map. In other
words, $Q_{J}\left(  A\right)  $ is the set of all $a\in A$ quasinilpotent
modulo $J$.

Let $\mathrm{rad}\left(  A\right)  $ denote the Jacobson radical of $A$, and
let $\mathrm{rad}_{J}\left(  A\right)  $ denote the Jacobson radical of $A$
modulo $J$, that is the preimage in $A$ of the radical of the quotient algebra
$A/\overline{J}$. If $A$ is complete, write $\mathrm{Rad}\left(  A\right)  $
instead of $\mathrm{rad}\left(  A\right)  $.

\begin{proposition}
\label{predv1} Let $A$ be a Banach algebra and $J$ be an ideal of $A$. Then

\begin{itemize}
\item[$\mathrm{(i)}$] $Q_{J}\left(  A\right)  $ contains the intersection of
all primitive ideals of $A$ containing $J$ ($=\mathrm{Rad}_{J}\left(
A\right)  $).

\item[$\mathrm{(ii)}$] If a closed subalgebra $B\subset A$ is such that
$J\subset B$ and $B/\overline{J}$ is radical then $B\subset Q_{J}\left(
A\right)  $.

\item[$\mathrm{(iii)}$] An element $a\in A$ belongs to $Q_{J}\left(  A\right)
$ if and only if for each $\varepsilon>0$ there is $n\in\mathbb{N}$ such that
$\mathrm{dist}(a^{m},J)<\varepsilon^{m}$ for all $m\geq n$.
\end{itemize}
\end{proposition}

\begin{proof}
(i) As is known the Jacobson radical $\mathrm{Rad}(A)$ of a Banach algebra $A$
is the largest ideal consisting of quasinilpotents. Hence $Q_{J}\left(
A\right)  $ contains the Jacobson radical of $A$ modulo $\overline{J}$. The
last, as well known, is the intersection of all primitive ideals containing
$J$.
(ii) Straightforward.
(iii) Follows immediately from the equality $\mathrm{dist}(a,J)=\Vert
a/\overline{J}\Vert$.
\end{proof}

An algebra is usually said to be \textit{radical} if it is Jacobson radical.
The following lemma slightly improves the respective classical result.

\begin{lemma}
\label{quotBan} Let $A$ be a Banach algebra, and let $I$ be an ideal of $A$.
If $I$ is radical and $A/\overline{I}$ is radical then $A$ is radical.
\end{lemma}

\begin{proof}
Let $\pi$ be a strictly irreducible representation of $A$ on $X$. We may
assume that $X$ is a Banach space and that $\pi$ is continuous. If the
restriction of $\pi$ to $I$ is non-zero then it is a strictly irreducible
representation of $I$, in contradiction with the radicality of $I$. Therefore
$\pi|_{I}=0$ and, by continuity, $\overline{I}\subset\ker\pi$. Hence $\pi$
defines a strictly irreducible representation of $A/\overline{I}$. Since
$A/\overline{I}$ is radical, this means that $\dim X=1$ and $\pi=0$.
\end{proof}


\subsection{Normed subalgebras and flexible ideals}

\subsubsection{Spectrum with respect to a Banach subalgebra}

Let $A,B$ be normed algebras with norms $\left\Vert \cdot\right\Vert _{A}$ and
$\left\Vert \cdot\right\Vert _{B}$ respectively, and let $B$ be a subalgebra
of $A$. We say that $B$ is a \textit{normed subalgebra} if $\left\Vert
\cdot\right\Vert _{A}\leq\left\Vert \cdot\right\Vert _{B}$ on $B$. Every
complete (with respect to $\left\Vert \cdot\right\Vert _{B}$) normed
subalgebra $B$ is called a \textit{Banach subalgebra}.

Let $\sigma_{A}\left(  a\right)  $, or simply $\sigma\left(  a\right)  $,
denote the spectrum of $a\in A$ with respect to $A^{1}$. Recall that this
definition of spectrum coincides with Definition 5.1 in \cite{BD} in virtue of
\cite[Lemma 5.2]{BD}. Let $\widehat{\sigma}_{A}\left(  a\right)  $ denote the
polynomially convex hull of $\sigma_{A}\left(  a\right)  $, and let $\rho
_{A}\left(  a\right)  $, or simply $\rho\left(  a\right)  $, denote the
spectral radius of $a$ defined as $\inf_{n}\left\Vert a^{n}\right\Vert
_{A}^{1/n}$.

If $A$ is a Banach algebra then $\rho_{A}\left(  a\right)  =\sup\left\{
\left\vert \lambda\right\vert :\lambda\in\sigma_{A}\left(  a\right)  \right\}
$ (Gelfand's formula), and $\widehat{\sigma}_{A}\left(  a\right)  $ is
received from $\sigma_{A}\left(  a\right)  $ by filling the holes of
$\sigma_{A}\left(  a\right)  $. The term \textquotedblleft
clopen\textquotedblright\ means \textquotedblleft closed\ and
open\ simultaneously\textquotedblright.

\begin{proposition}
\label{sp}Let $A$ be a unital Banach algebra, and let $B$ be a unital Banach
subalgebra of $A$ (i.e. the units for $A$ and $B$ coincide). Then

\begin{itemize}
\item[$\mathrm{(i)}$] $\sigma_{A}\left(  a\right)  \subset\sigma_{B}\left(
a\right)  $ for every $a\in B$, and each clopen subset of $\sigma_{B}(a)$ has
a non-void intersection with the polynomial hull $\widehat{\sigma}_{A}\left(
a\right)  $ of $\sigma_{A}(a)$.

\item[$\mathrm{(ii)}$] If $\sigma_{B}(A)$ is finite or countable then
$\sigma_{A}(a)=\sigma_{B}(a)$.
\end{itemize}
\end{proposition}

\begin{proof}
It is evident that $\sigma_{A}\left(  a\right)  \subset\sigma_{B}\left(
a\right)  $. To prove the second statement, suppose that $\sigma_{1}$ is a
clopen subset of $\sigma_{B}(a)$ which doesn't intersect $\widehat{\sigma}%
_{A}\left(  a\right)  $. Let $p$ be the corresponding Riesz projection in
$B$,
\[
p=\frac{1}{2\pi i}\int_{\Gamma}(\lambda-a)^{-1}d\lambda,
\]
where $\Gamma$ surrounds $\sigma_{1}$ and doesn't intersect $\widehat{\sigma
}_{A}\left(  a\right)  $. Then $p\neq0$.
On the other hand, $p=0$ because it can be regarded as a Riesz projection of
$a$ in $A$ and there are no points of $\sigma_{A}(a)$ inside $\Gamma$. The
obtained contradiction proves (i).
To show (ii), note that if $\sigma_{B}(a)$ is countable then $\sigma_{A}(a)$
is countable hence $\widehat{\sigma}_{A}\left(  a\right)  =\sigma_{A}(a)$.
Thus each clopen subset of $\sigma_{B}(a)$ intersects $\sigma_{A}(a)$. Any
point $\lambda\in\sigma_{B}(a)$ is clearly the intersection of a sequence of
clopen subsets of $\sigma_{B}(a)$. Since all of them intersects $\sigma
_{A}(a)$ we get that $\lambda\in\sigma_{A}(a)$. Thus $\sigma_{B}%
(a)\subset\sigma_{A}(a)$ and we are done.
\end{proof}

When the subalgebra $B$ is closed in $A$, the result is related to
\cite[Theorem 10.18]{Rud}. The situation is especially simple if $B$ is a
(non-necessarily closed) ideal of $A$.

\begin{remark}
\label{sp2}Let $I$ be an ideal of an algebra $A$ and $a\in I$. It is easy to
check that if $\left(  a-\lambda\right)  b=1$ or $b\left(  a-\lambda\right)
=1$ for $b\in A^{1}$ and $\lambda\neq0$, then $b+\lambda^{-1}\in I$. Hence, in
virtue of \cite[Lemma 5.2]{BD}, $\left\{  0\right\}  \cup\sigma_{I}\left(
a\right)  =\left\{  0\right\}  \cup\sigma_{A}\left(  a\right)  $.
\end{remark}

\subsubsection{Flexible ideals}

Let $A,I$ be normed algebras with norms $\left\Vert \cdot\right\Vert _{A}$ and
$\left\Vert \cdot\right\Vert _{I}$ respectively, and let $I$ be an ideal of
$A$ such that
\[
\left\Vert x\right\Vert _{A}\leq\left\Vert x\right\Vert _{I}\text{ and
}\left\Vert axb\right\Vert _{I}\leq\left\Vert a\right\Vert _{A}\left\Vert
x\right\Vert _{I}\left\Vert b\right\Vert _{A}%
\]
for all $x\in I$ and $a,b\in A^{1}$. Such an algebra norm $\left\Vert
\cdot\right\Vert _{I}$ on $I$ is called \textit{flexible }(with respect to
$\left\Vert \cdot\right\Vert _{A}$ or $\left(  A,\left\Vert \cdot\right\Vert
_{A}\right)  $, naturally). An ideal having a flexible norm is called a
\textit{flexible ideal}. Every ideal $I$ of a normed algebra $A$ with
$\left\Vert \cdot\right\Vert _{I}=\left\Vert \cdot\right\Vert _{A}$ on $I$ is
of course flexible.

By definition, a Banach ideal $I$ of a normed algebra $A$ is an ideal which is
a Banach subalgebra of $A$.

\begin{lemma}
Every Banach ideal of a Banach algebra is flexible with respect to an
equivalent algebra norm.
\end{lemma}

\begin{proof}
Let $I$ be a Banach ideal of a Banach algebra $A$. By\ \cite[Theorem
2.3]{Bar1}, there is $s>0$ such that $\left\Vert axb\right\Vert _{I}\leq
s\left\Vert a\right\Vert _{A}\left\Vert x\right\Vert _{I}\left\Vert
b\right\Vert _{A}$ for all $x\in I$ and $a,b\in A^{1}$. Define $\left\Vert
\cdot\right\Vert _{I}^{\prime}$ on $I$ by
\[
\left\Vert x\right\Vert _{I}^{\prime}=\sup\left\{  \left\Vert axb\right\Vert
_{I}:\left\Vert a\right\Vert _{A},\left\Vert b\right\Vert _{A}\leq1,a,b\in
A^{1}\right\}
\]
for every $x\in I$. It is easy to check that $\left\Vert \cdot\right\Vert
_{I}^{\prime}$ is an algebra norm on $I$,
\[
\left\Vert \cdot\right\Vert _{I}\leq\left\Vert \cdot\right\Vert _{I}^{\prime
}\leq s\left\Vert \cdot\right\Vert _{I}.
\]
and
\[
\left\Vert axb\right\Vert _{I}^{\prime}\leq\left\Vert a\right\Vert
_{A}\left\Vert x\right\Vert _{I}^{\prime}\left\Vert b\right\Vert _{A}%
\]
for every $a,b\in A^{1}$. As $\left\Vert \cdot\right\Vert _{A}\leq\left\Vert
\cdot\right\Vert _{I}^{\prime}$ on $I$, we obtain that $I$ is a flexible ideal
with respect to $\left\Vert \cdot\right\Vert _{I}^{\prime}$.
\end{proof}

\subsubsection{Completion of normed subalgebras and ideals}

If $B$ is a normed subalgebra in a normed algebra $A$ then the
\textquotedblleft identity\textquotedblright\ homomorphism $\mathfrak{i}%
:(B,\Vert\cdot\Vert_{B})\rightarrow A$ is continuous and therefore extends by
continuity to the homomorphism $\mathfrak{\hat{\imath}}:\widehat{B}%
\rightarrow\widehat{A}$. The proof of the following result is straightforward
and we omit it.

\begin{lemma}
\label{flex1}Let $B$ be a normed subalgebra of a normed algebra $A$. Then

\begin{itemize}
\item[$\mathrm{(i)}$] The image $\mathfrak{\hat{\imath}}(\widehat{B})$ of the
completion $\widehat{B}$ of $B$ in the completion $\widehat{A}$ of $A$ is a
Banach subalgebra of $\widehat{A}$ with respect to the norm $\left\Vert
\cdot\right\Vert _{\ast}$ of the quotient $\widehat{B}/\ker\mathfrak{\hat
{\imath}}$.

\item[$\mathrm{(ii)}$] If $B$ is a flexible ideal of $A$ then $\mathfrak{\hat
{\imath}}(\widehat{B})$ is a Banach ideal of $\widehat{A}$ and its norm
$\left\Vert \cdot\right\Vert _{\ast}$ is flexible.
\end{itemize}
\end{lemma}

When it cannot lead to a misunderstanding, we will write $\widehat{B}%
^{{}\left(  A\right)  }$ or, simply, $\widehat{B}^{{}\left(  \cdot\right)  }$
instead of $\mathfrak{\hat{\imath}}(\widehat{B})$.

\subsubsection{Sums and intersections of Banach ideals}

The following extends the class of examples of flexible ideals.

\begin{proposition}
\label{fl2}Let $I$ and $J$ be flexible ideals of a normed algebra $A$. Then

\begin{itemize}
\item[$\mathrm{(i)}$] $I\cap J$ is a flexible ideal of $A$ with respect to the
norm $\left\Vert \cdot\right\Vert _{I\cap J}=\max\left\{  \left\Vert
\cdot\right\Vert _{I},\left\Vert \cdot\right\Vert _{J}\right\}  $.

\item[$\mathrm{(ii)}$] $I+J$ is a flexible ideal of $A$ with respect to the
norm
\[
\left\Vert z\right\Vert _{I+J}=\inf\left\{  \left\Vert x\right\Vert
_{I}+\left\Vert y\right\Vert _{J}:z=x+y,x\in I,y\in J\right\}
\]
for every $z\in I+J$.

\item[$\mathrm{(iii)}$] $I\cap J$ is a flexible ideal of $I+J$.

\item[$\mathrm{(iv)}$] If $I$ and $J$ are Banach ideals then $I\cap J$ and
$I+J$ are Banach ideals with flexible norms $\left\Vert \cdot\right\Vert
_{I\cap J}$ and $\left\Vert \cdot\right\Vert _{I+J}$ respectively.
\end{itemize}
\end{proposition}

\begin{proof}
It follows from \cite[Lemma 2.3.1]{BL76} that $\left\Vert \cdot\right\Vert
_{I\cap J}$ and $\left\Vert \cdot\right\Vert _{I+J}$ are norms and (iv) holds
if (i) and (ii) hold. So it suffices to show that $\left\Vert \cdot\right\Vert
_{I\cap J}$ and $\left\Vert \cdot\right\Vert _{I+J}$ are flexible. It is easy
to see that $\left\Vert \cdot\right\Vert _{I\cap J}$ is flexible and
$\left\Vert \cdot\right\Vert _{A}\leq$ $\left\Vert \cdot\right\Vert _{I+J}$ on
$I+J$.
For $z=x+y$ and $z^{\prime}=x^{\prime}+y^{\prime}$ with $x,x^{\prime}\in I$
and $y,y^{\prime}\in J$, we obtain that%
\begin{align*}
\left\Vert zz^{\prime}\right\Vert _{I+J}  &  \leq\left\Vert xx^{\prime
}+xy^{\prime}\right\Vert _{I}+\left\Vert yx^{\prime}+yy^{\prime}\right\Vert
_{J}\leq\left\Vert xx^{\prime}\right\Vert _{I}+\left\Vert xy^{\prime
}\right\Vert _{I}+\left\Vert yx^{\prime}\right\Vert _{J}+\left\Vert
yy^{\prime}\right\Vert _{J}\\
&  \leq\left\Vert x\right\Vert _{I}\left\Vert x^{\prime}\right\Vert
_{I}+\left\Vert x\right\Vert _{I}\left\Vert y^{\prime}\right\Vert
_{A}+\left\Vert y\right\Vert _{J}\left\Vert x^{\prime}\right\Vert
_{A}+\left\Vert y\right\Vert _{J}\left\Vert y^{\prime}\right\Vert _{J}\\
&  \leq\left\Vert x\right\Vert _{I}\left\Vert x^{\prime}\right\Vert
_{I}+\left\Vert x\right\Vert _{I}\left\Vert y^{\prime}\right\Vert
_{J}+\left\Vert y\right\Vert _{J}\left\Vert x^{\prime}\right\Vert
_{I}+\left\Vert y\right\Vert _{J}\left\Vert y^{\prime}\right\Vert _{J}\\
&  \leq\left(  \left\Vert x\right\Vert _{I}+\left\Vert y\right\Vert
_{J}\right)  \left(  \left\Vert x^{\prime}\right\Vert _{I}+\left\Vert
y^{\prime}\right\Vert _{J}\right)
\end{align*}
and
\begin{align*}
\left\Vert azb\right\Vert _{I+J}  &  \leq\left\Vert axb\right\Vert
_{I}+\left\Vert ayb\right\Vert _{J}\leq\left\Vert a\right\Vert _{A}\left\Vert
x\right\Vert _{I}\left\Vert b\right\Vert _{A}+\left\Vert a\right\Vert
_{A}\left\Vert y\right\Vert _{J}\left\Vert b\right\Vert _{A}\\
&  =\left\Vert a\right\Vert _{A}\left(  \left\Vert x\right\Vert _{I}%
+\left\Vert y\right\Vert _{J}\right)  \left\Vert b\right\Vert _{A}%
\end{align*}
for every $a,b\in A^{1}$. Hence $\left\Vert \cdot\right\Vert _{I+J}$ is
clearly a flexible norm.
(iii) It is clear that $I\cap J$ is an ideal of $I+J$, $\left\Vert
\cdot\right\Vert _{I+J}\leq\left\Vert \cdot\right\Vert _{I\cap J}$ on $I\cap
J$, and flexibility of $\left\Vert \cdot\right\Vert _{I\cap J}$ with respect
to $I+J$ follows from one with respect to $A$.
\end{proof}

In conditions of Proposition \ref{fl2} it is convenient to call $I\cap J$ and
$I+J$ with their flexible norms a \textit{flexible intersection} and a
\textit{flexible sum} of ideals $I$ and $J$, respectively.

An important class of examples of flexible ideals may be obtained by using the
notion of normed operator ideals \cite{P78, DF93}. Note that normed operator
ideals in \cite{P78} are the same as Banach operator ideals in \cite{DF93},
and we prefer the terminology in \cite{DF93}.

\begin{example}
\label{fle}Let $\mathcal{I}$ be a Banach operator ideal. Then $\mathcal{I}%
\left(  X\right)  $ is a Banach ideal of $\mathcal{B}\left(  X\right)  $ for
every Banach space $X$ and its norm is flexible.
\end{example}

\subsection{Projective tensor products}

\subsubsection{ Tensor products of normed algebras\label{ss1}}

Let $A_{1}\otimes A_{2}$ denote the algebraic tensor product of normed
algebras $A_{1}$ and $A_{2}$, and let $A_{1}\otimes_{\gamma}A_{2}$ denote
$\left(  A_{1}\otimes A_{2},\gamma\right)  $, where $\gamma$ is the projective
crossnorm. Recall that $\gamma$ is defined by
\[
\gamma\left(  c\right)  =\inf\left\{  \sum\left\Vert a_{i}\right\Vert _{A_{1}%
}\left\Vert b_{i}\right\Vert _{A_{2}}:\sum a_{i}\otimes b_{i}=c\right\}
\]
for every $c\in A_{1}\otimes_{\gamma}A_{2}$. Then $A=A_{1}\otimes_{\gamma
}A_{2}$ is a normed algebra and $\gamma$ is its algebra norm. To underline
that the projective norm $\gamma$ is considered in $A_{1}\otimes A_{2}$, we
write $\gamma=\gamma_{A_{1},A_{2}}$ or $\gamma=\gamma_{A}$. We also write
$\gamma=\gamma_{\left\Vert \cdot\right\Vert _{A_{1}},\left\Vert \cdot
\right\Vert _{A_{2}}}$ to indicate which norm are considered in $A_{i}$.

Let $A_{1}$ and $A_{2}$ be normed algebras. By $A_{1}\widehat{\otimes}%
_{\gamma}A_{2}$, or simply $A_{1}\widehat{\otimes}A_{2}$, we denote the
projective tensor product of $A_{1}$ and $A_{2}$ that is the completion of
$A_{1}\otimes_{\gamma}A_{2}$. By definition, it is a Banach algebra, and
clearly it coincides with the projective tensor product of the completions of
$A_{i}$. The elements of $A_{1}\widehat{\otimes}A_{2}$ can be written in the
form
\begin{equation}
c=\sum_{k=1}^{\infty}a_{k}\otimes b_{k}\text{ with }\sum_{k}\Vert a_{k}%
\Vert\Vert b_{k}\Vert<\infty, \label{proje}%
\end{equation}
where $a_{k}\in A_{1},b_{k}\in A_{2}$. Moreover, the norm $\left\Vert
\cdot\right\Vert =\gamma\left(  \cdot\right)  $ in $A_{1}\widehat{\otimes
}A_{2}$ is given by%
\[
\Vert c\Vert=\inf\sum_{k}\Vert a_{k}\Vert\Vert b_{k}\Vert,
\]
where $\inf$ is taken over representations of $c$ in form (\ref{proje}).


\subsubsection{Tensor products of normed subalgebras and ideals}

If $B_{i}$ is a subalgebra of an algebra $A_{i}$ for $i=1,2$, then
$B_{1}\otimes_{\gamma}B_{2}$ is a subalgebra of $A_{1}\otimes_{\gamma}A_{2}$
(see \cite[Section 3.3.1]{B98}). If $I_{i}$ is an ideal of $A_{i}$ for
$i=1,2$, then $I_{1}\otimes_{\gamma}I_{2}$ is clearly an ideal of
$A_{1}\otimes_{\gamma}A_{2}$.

\begin{proposition}
\label{tenf}Let $A_{1}$ and $A_{2}$ be normed algebras, and $A=A_{1}%
\otimes_{\gamma}A_{2}$. Then

\begin{itemize}
\item[$\mathrm{(i)}$] If $B_{i}$ is a normed subalgebra of $A_{i}$ for
$i=1,2$, then $B:=B_{1}\otimes_{\gamma}B_{2}$ with $\gamma_{B}=\gamma
_{\left\Vert \cdot\right\Vert _{B_{1}},\left\Vert \cdot\right\Vert _{B_{2}}}$
is a normed subalgebra of $A$.

\item[$\mathrm{(ii)}$] If $I_{i}$ is a flexible ideal of $A_{i}$ for $i=1,2$,
then $I:=I_{1}\otimes_{\gamma}I_{2}$ with $\gamma_{I}=\gamma_{\left\Vert
\cdot\right\Vert _{I_{1}},\left\Vert \cdot\right\Vert _{I_{2}}}$ is a flexible
ideal of $A$ and $\gamma_{I}\left(  azb\right)  \leq\gamma_{A}\left(
a\right)  \gamma_{I}\left(  z\right)  \gamma_{A}\left(  b\right)  $ for every
$a,b\in A^{1}$ and $z\in I$.
\end{itemize}
\end{proposition}

\begin{proof}
(i) Indeed, the norm $\gamma_{B}$ on $B$ majorizes $\gamma_{A}$ (and the
equality does not hold in general even if $\Vert\cdot\Vert_{B_{i}}=\Vert
\cdot\Vert_{A_{i}}$ on $B_{i}$).
(ii) Straightforward.
\end{proof}

The natural embedding $\mathfrak{i}$ of $B_{1}{\otimes}_{\gamma}B_{2}$ into
$A_{1}\otimes_{\gamma}A_{2}$ extends by continuity to a continuous
homomorphism $\mathfrak{\hat{\imath}}$ of $B_{1}\widehat{\otimes}_{\gamma
}B_{2}$ into $A=A_{1}\widehat{\otimes}_{\gamma}A_{2}$. Let $\mathfrak{\hat
{\imath}}(B_{1}\widehat{\otimes}_{\gamma}B_{2})$ be supplied with the norm
inherited from the quotient $\left(  B_{1}\widehat{\otimes}_{\gamma}%
B_{2}\right)  /\ker\mathfrak{\hat{\imath}}$. We denote this subalgebra by
$B_{1}\widehat{\otimes}^{{}\left(  A\right)  }B_{2}$ or simply $B_{1}%
\widehat{\otimes}^{{}\left(  \cdot\right)  }B_{2}$.

Taking into account Lemma \ref{flex1} and Proposition \ref{tenf}, we obtain
the following result.

\begin{corollary}
Let $A_{1}$ and $A_{2}$ be normed algebras, and $A=A_{1}\widehat{\otimes
}_{\gamma}A_{2}$. Then

\begin{itemize}
\item[$\mathrm{(i)}$] If $B_{i}$ is a normed subalgebra of $A_{i}$ for
$i=1,2$, then $B:=B_{1}\widehat{\otimes}^{\left(  \cdot\right)  }B_{2}$ is a
Banach subalgebra of $A$.

\item[$\mathrm{(ii)}$] If $I_{i}$ is a flexible ideal of $A_{i}$ for $i=1,2$,
then $I:=I_{1}\widehat{\otimes}^{(\cdot)}I_{2}$ is a Banach ideal of $A$ and
its norm (inherited from the respective quotient) is flexible.
\end{itemize}
\end{corollary}

\subsubsection{Quotients of tensor products}

\begin{proposition}
\label{quo}Let $A_{1}$ and $A_{2}$ be normed algebras and $A=A_{1}%
\widehat{\otimes}A_{2}$. Let $J_{i}$ be ideals of $A_{i}$ for $i=1,2$, and let
$J=J_{1}\otimes A_{2}+A_{1}\otimes J_{2}$. Then

\begin{itemize}
\item[$\mathrm{(i)}$] The closure of $J$ in $A$ is an ideal of $A$, and
$A/\overline{J}$ is topologically isomorphic to $B=\left(  A_{1}%
/\overline{J_{1}}\right)  \widehat{\otimes}\left(  A_{2}/\overline{J_{2}%
}\right)  $.

\item[$\mathrm{(ii)}$] If $I_{i}$ are closed ideals of $A_{i}$ containing
$J_{i}$, and if $I:=I_{1}\widehat{\otimes}^{{}\left(  \cdot\right)  }%
A_{2}+A_{1}\widehat{\otimes}^{{}\left(  \cdot\right)  }I_{2}$ is a flexible
sum of Banach ideals in $A$, then the closure of $J$ in $I$ is an ideal of $I$
and $I/\overline{J}$ is topologically isomorphic to the algebra
\[
Q=\left(  I_{1}/\overline{J_{1}}\right)  \widehat{\otimes}^{{}\left(
\cdot\right)  }\left(  A_{2}/\overline{J_{2}}\right)  +\left(  A_{1}%
/\overline{J_{1}}\right)  \widehat{\otimes}^{{}\left(  \cdot\right)  }\left(
I_{2}/\overline{J_{2}}\right)
\]
taken with the norm of the flexible sum of Banach ideals in $\left(
A_{1}/\overline{J_{1}}\right)  \widehat{\otimes}\left(  A_{2}/\overline{J_{2}%
}\right)  $.
\end{itemize}
\end{proposition}

\begin{proof}
(i) The first statement follows from the fact that $J$ is an ideal of the
algebraic tensor product $A_{1}\otimes A_{2}$. To show the second one, assume
first that $A_{1}$ and $A_{2}$ are Banach algebras, and that $J_{i}$ is a
closed ideal of $A_{i}$ for $i=1,2$.
As usual, we denote by $q_{J_{i}}$ the standard epimorphisms from $A_{i}$ to
$A_{i}/J_{i}$ and by $q_{J}$ the standard epimorphism from $A=A_{1}%
\widehat{\otimes}A_{2}$ to $A/\overline{J}$. Setting
\[
\phi((a_{1}+J_{1}){\otimes}(a_{2}+J_{2}))=q_{J}(a_{1}{\otimes}a_{2}),
\]
we obtain a bounded homomorphism $\phi:(A_{1}/J_{1})\widehat{\otimes}%
(A_{2}/J_{2})\rightarrow A/\overline{J}$ such that the diagram
\[
\xymatrix{
&A/\overline{J}\\
A\ar[rr]^{q_{J_1}\otimes q_{J_2}}\ar[ur]^{q_J} &&(A_1/J_1)\widehat{\otimes}(A_2/J_2)\ar[ul] _\phi
}
\]
is commutative.
Since $q_{J}$ is surjective, $\phi$ is surjective. On the other hand, it is
easy to see that $q_{J_{1}}\otimes q_{J_{2}}$ is surjective. So, if
$\phi(z)=0$ for some $z\in(A_{1}/J_{1})\widehat{\otimes}(A_{2}/J_{2})$ then
$z=\left(  q_{J_{1}}\otimes q_{J_{2}}\right)  (a)$ for some $a\in A$. In fact,
we have that $a\in\overline{J}$ by the commutativity of the diagram. But
$J\subset\ker(q_{J_{1}}\otimes q_{J_{2}})$, whence $z=0$. This implies that
$\phi$ is injective. Thus $\phi$ establishes a bounded isomorphism of $B$ and
$A/\overline{J}$. By the Banach Theorem, this isomorphism is topological.
In the general case, passing to completions of $A_{i}$ and to closures of
$J_{i}$ and applying Proposition \ref{flex1} and simple identifications in
completions of quotients of normed algebras, we get the result.
(ii) Follows from a similar analysis of the commutative diagram
\[
\xymatrix{
&I/\overline{J}\\
I\ar[rr]^{q}\ar[ur]^{q_J} && Q\ar[ul] _\psi
}
\]
where $q$ is the map sending $p_{1}{\otimes}a_{2}+a_{1}{\otimes}p_{2}$ to
$(p_{1}/\overline{J_{1}}){\otimes}(a_{2}/\overline{J_{2}})+(a_{1}%
/\overline{J_{1}}){\otimes}(p_{2}/\overline{J_{2}})$ for all $p_{i}\in I_{i}$,
$a_{i}\in A_{i}$ ($i=1,2$). The existence of $\psi$ is evident, surjectivity
of $q$ can be verified in a standard way.
\end{proof}

\section{Tensor radical\label{ss}}


\subsection{Tensor spectral radius of a summable family\label{sss1}}

Let $A$ be a normed algebra. We will call by \textit{families} arbitrary
sequences of elements of $A$; two families are \textit{equivalent} (write
$\left\{  a_{n}\right\}  _{1}^{\infty}\simeq\left\{  b_{n}\right\}
_{1}^{\infty}$) if one of them can be obtained from the other by renumbering.
The equivalence classes can be considered as countable \textit{generalized
subsets} \cite{TR1}: to characterize the class determined by a sequence one
have only to indicate which elements of $A$ come into the sequence and how
many times.

By definition \cite[Section 3.4]{TR1}, a \textit{generalized subset} $S$ of
$A$ is a cardinal valued function $\varkappa_{S}$ defined on $A$. The set
$\left\{  a\in A:\varkappa_{S}\left(  a\right)  >0\right\}  $ is called a
\textit{support }of $S$. One can regard usual subsets $N\subset A$ as
generalized ones, identifying the indicator $\varkappa_{N}$ of $N$ with $N$. A
generalized subset $S$ of $A$ is \textit{countable }if its support and
$\varkappa_{S}\left(  a\right)  $ are (finite or) countable for every $a$ from
the support of $S$.

Let $S$ and $P$ be generalized subsets of $A$. The \textit{inclusion}
$S\subset P$ means
\[
\varkappa_{S}(a)\leq\varkappa_{P}(a)
\]
for every $a\in A$.

We define the \textit{disjoint union} $S\sqcup P$ of generalized subsets of
$A$ by
\[
\varkappa_{S\sqcup P}(a)=\varkappa_{S}(a)+\varkappa_{P}(a)
\]
for every $a\in A$. Disjoint union of a collection of generalized subsets is
defined similarly. In particular, for an integer $n>0$, the disjoint union of
$n$ copies of $S$ will be denoted by $n\bullet S$.

We define the product $SP$ of generalized subsets of $A$ by
\[
\varkappa_{SP}(a)=\sum_{(b,c)\in A\times A,\,bc=a}\varkappa_{S}(b)\varkappa
_{P}(c)
\]
for every $a\in A$.

Given a generalized subset $S$ of $A$, put
\[
\eta(S)=\sum_{a\in A}\varkappa_{S}(a)\Vert a\Vert
\]
and%
\[
\left\Vert S\right\Vert =\sup_{\varkappa_{S}(a)>0}\left\{  \left\Vert
a\right\Vert :a\in A,\text{ }\varkappa_{S}(a)>0\right\}  .
\]
If $\eta(S)<\infty$ then $S$ is called \textit{summable}, and if $\left\Vert
S\right\Vert <\infty$ then $S$ is called \textit{bounded.}

To each sequence $M=\{a_{n}\}_{n=1}^{\infty}$ in $A$ there corresponds a
countable generalized subset $S=S_{\left(  M\right)  }$ by the rule
\[
\varkappa_{S}(a)=\mathrm{card}\{n:a_{n}=a\}
\]
for every $a\in A$. We say that $M$ is a \textit{representative} of $S$. In
terms of representatives $M=\{a_{n}\}_{1}^{\infty}$ and $N=\{b_{n}%
\}_{1}^{\infty}$ the family $MN$ corresponds to the two-index sequence
$\{a_{n}b_{m}\}_{n,m=1}^{\infty}$ which can be renumbered in an arbitrary way,
while $M\sqcup N$ corresponds to the sequence $\{c_{n}\}_{1}^{\infty}$ with
$c_{2k-1}=a_{k}$, $c_{2k}=b_{k}$. It is obvious in this context that
\[
MN\simeq\sqcup_{i=1}^{\infty}a_{i}N\simeq\sqcup_{j=1}^{\infty}Mb_{j},
\]
where $aN=\left\{  ab_{n}\right\}  _{1}^{\infty}$ and $Mb=\left\{
a_{n}b\right\}  _{1}^{\infty}$ as usual. In particular,
\[
M\left(  N_{1}\sqcup N_{2}\right)  \simeq MN_{1}\sqcup MN_{2}\text{ and
}\left(  M_{1}\sqcup M_{2}\right)  N\simeq M_{1}N\sqcup M_{2}N
\]
for any families $M_{i}$ and $N_{i}$, $i=1,2$, and
\begin{equation}
\left(  MN\right)  K\simeq M(NK) \label{f39}%
\end{equation}
for any families in $A$. Set $M^{1}\simeq M$, $M^{n}\simeq M^{n-1}M$ for every
$n>0$. By (\ref{f39})
\begin{equation}
M^{n+m}\simeq M^{n}M^{m} \label{f41}%
\end{equation}
for every $n,m\in\mathbb{N}$.

For two families $M$ and $N$ in $A$, we say that $M$ is a \textit{subfamily}
of $N$ (write $M\sqsubset N$) if $S_{\left(  M\right)  }\subset S_{\left(
N\right)  }$ for corresponding generalized subsets of $A$.

Now let $S$ be a summable generalized subset of $A$. This is equivalent to the
condition that $S$ has a representative $M$ in $\ell_{1}(A)$, i.e.
$S=S_{\left(  M\right)  }$ for some $M\in\ell_{1}(A)$. Using this and setting
$\eta(M)=\eta\left(  S_{\left(  M\right)  }\right)  $, we simply write
\textquotedblleft a family $M=\{a_{n}\}_{1}^{\infty}$ in $A$ is
summable\textquotedblright. Moreover,
\[
\Vert M\Vert_{\ell_{1}(A)}=\eta(M)=\eta(N)
\]
for every $N\simeq M$.

Let $M$ and $N$ be summable families in $A$. It is evident that
\[
\eta(M\sqcup N)=\eta(M)+\eta(N)
\]
and
\begin{equation}
\eta(MN)\leqslant\eta(M)\eta(N). \label{f40}%
\end{equation}
We obtain by (\ref{f41}) and (\ref{f40}) that
\begin{equation}
\eta(M^{n+m})\leqslant\eta(M^{n})\eta(M^{m}) \label{f58}%
\end{equation}
for every $n,m\in\mathbb{N}$. It follows from (\ref{f58}) that, for every
summable family $M$, there exists a limit
\[
\rho_{t}(M)=\lim(\eta\left(  M^{n}\right)  )^{1/n}=\inf(\eta(M^{n}))^{1/n}.
\]
The number $\rho_{t}(M)$ is called a \textit{tensor spectral radius} of $M$.

As $(M^{m})^{n}\simeq M^{mn}$ for every $n,m\in\mathbb{N}$, then
\begin{equation}
\rho_{t}(M^{m})^{1/m}=(\lim_{n}(\eta((M^{m})^{n}))^{1/n})^{1/m}=\lim_{n}%
(\eta(M^{mn}))^{1/nm}=\rho_{t}(M). \label{f62}%
\end{equation}

Now let $B$ be a normed algebra, and let $S$ be a bounded countable
generalized subset of $B$. Then $S$ has a representative $L$ in $\ell_{\infty
}(B)$, i.e. $S=S_{\left(  L\right)  }$ for some $L\in\ell_{\infty}(B)$.
Setting $\left\Vert L\right\Vert =\left\Vert S_{\left(  L\right)  }\right\Vert
$, we write \textquotedblleft a family $L=\{b_{n}\}_{1}^{\infty}$ in $B$ is
bounded\textquotedblright. Moreover,
\[
\Vert L\Vert_{\ell_{\infty}(B)}=\left\Vert L\right\Vert =\left\Vert
K\right\Vert
\]
for every $K\simeq L$. A usual countable subset $N$ of $B$ is bounded if and
only if $\sup_{b\in N}\Vert b\Vert<\infty$.

Let $L$ and $K$ be bounded families in $K$. It is evident that
\[
\left\Vert L\sqcup K\right\Vert =\max\left\{  \left\Vert L\right\Vert
,\left\Vert K\right\Vert \right\}  \text{ and }\left\Vert LK\right\Vert
\leq\left\Vert L\right\Vert \left\Vert K\right\Vert .
\]
It follows as above that, for every bounded family $L$, there is a limit
\[
\rho(L)=\lim(\left\Vert L^{n}\right\Vert )^{1/n}=\inf(\left\Vert
L^{n}\right\Vert )^{1/n}.
\]
The number $\rho(L)$ is called a \textit{joint spectral radius} of $L$. It is
clear that $\rho(L^{m})^{1/m}=\rho(L)$ for every $m\in\mathbb{N}$.

Let $A$ and $B$ be normed algebras, $M=\left\{  a_{n}\right\}  _{1}^{\infty
}\in\ell_{1}(A)$ and $L=\left\{  b_{n}\right\}  _{1}^{\infty}\in\ell_{\infty
}(B)$. Let $M_{\otimes}L$ denote an element of $A\widehat{\otimes}B$ which is
equal to $\sum_{n=1}^{\infty}a_{n}\otimes b_{n}$. It is clear (see also
Section \ref{ss1}) that for every element $z\in A\widehat{\otimes}B$ there are
$M\in\ell_{1}(A)$ and $L\in\ell_{\infty}(B)$ such that $z=M_{\otimes}L$.

The following theorem justifies the term \textquotedblleft tensor spectral
radius\textquotedblright.

\begin{theorem}
\label{tsr}Let $A$ be a normed algebra. Then

\begin{itemize}
\item[$\mathrm{(i)}$] $\rho\left(  M_{\otimes}L\right)  \leq\rho_{t}%
(M)\rho(L)$ for every normed algebra $B$, $M\in\ell_{1}(A)$ and $L\in
\ell_{\infty}(B)$.

\item[$\mathrm{(ii)}$] There are a unital Banach algebra $B$, $L\in
\ell_{\infty}(B)$, and a bounded linear operator $\ T:M\longmapsto M_{\otimes
}L$ from $\ell_{1}(A)$ into $A\widehat{\otimes}B$ such that $\left\Vert
M_{\otimes}L\right\Vert =\eta(M)$ and $\rho(M_{\otimes}L)=\rho_{t}(M)$ for
every $M\in\ell_{1}(A)$.
\end{itemize}
\end{theorem}

\begin{proof}
(i) Let $L=\left\{  b_{n}\right\}  _{1}^{\infty}\in\ell_{\infty}(B)$. Then for
every $M=\left\{  a_{n}\right\}  _{1}^{\infty}\in\ell_{1}(A)$ we have that%
\[
\left\Vert \left(  M_{\otimes}L\right)  ^{k}\right\Vert =\gamma\left(  \left(
M_{\otimes}L\right)  ^{k}\right)  \leq\sum_{n_{1},\ldots,n_{k}}\left\Vert
a_{n_{1}}\cdots a_{n_{k}}\right\Vert \left\Vert b_{n_{1}}\cdots b_{n_{k}%
}\right\Vert \leq\eta(M^{k})\left\Vert L^{k}\right\Vert
\]
Taking $k$-roots and passing to limits, we obtain that $\rho\left(
M_{\otimes}L\right)  \leq\rho_{t}(M)\rho(L)$.
(ii) Let $G$ be the free unital semigroup with a countable set $W=\{w_{k}%
\}_{k\geqslant1}$ of generators. That is $G=\cup_{m\geqslant0}W_{m}$, where
$W_{0}={1}$, $W_{m}$ is the direct product of $m$ copies of $W$ realized as
the set of `words' $w_{k_{1}}w_{k_{2}}...w_{k_{m}}$ of the length $m$, and the
multiplication is lexical. Let $B=\ell_{1}(G)$ be the corresponding semigroup algebra.
Let $L=\left\{  w_{n}\right\}  _{1}^{\infty}$. For any $M=\left\{
a_{n}\right\}  _{1}^{\infty}\in\ell_{1}(A)$, we have that
\[
M_{\otimes}L=\sum_{k=1}^{\infty}a_{k}{\otimes}w_{k}.
\]
Then $T:M\longmapsto M_{\otimes}L$ is a bounded linear operator from $\ell
_{1}(A)$ into $A\widehat{\otimes}B$ and
\[
T\left(  M\right)  ^{n}=\sum_{k_{1},\ldots k_{n}}a_{k_{1}}...a_{k_{n}}%
{\otimes}w_{k_{1}}...w_{k_{n}}.
\]
Since $A{\widehat{\otimes}}\ell_{1}(G)$ is isometrically isomorphic via the
map defined by $(a{\otimes}f)(g)\longmapsto f(g)a$ to the Banach algebra
$\ell_{1}(G,A)$ of all summable $A$-valued functions on $G$, then
\[
\Vert T\left(  M\right)  ^{n}\Vert=\sum_{k_{1},\ldots k_{n}}\Vert a_{k_{1}%
}...a_{k_{n}}\Vert=\eta(M^{n}).
\]
It follows that
\[
\rho(M_{\otimes}L)=\rho_{t}(M).
\]
\end{proof}

We write $\eta_{\Vert\cdot\Vert}(M)$ instead of $\eta(M)$ if there is a
necessity to indicate which norm in $A$ is meant.

\begin{proposition}
Let $M$ be a summable family in a normed algebra $A$. Then $\rho_{t}\left(
M\right)  $ doesn't change if the norm on $A$ is changed by an equivalent norm.
\end{proposition}

\begin{proof}
If $\left\Vert \cdot\right\Vert \leq t\left\Vert \cdot\right\Vert ^{\prime}$
for some $t>0$, then $\lim\mathbf{\mathrm{\eta}}_{\left\Vert \cdot\right\Vert
}\left(  M^{m}\right)  ^{1/m}\leq\lim\mathbf{\mathrm{\eta}}_{\left\Vert
\cdot\right\Vert ^{\prime}}\left(  M^{m}\right)  ^{1/m}$, so that the opposite
inequality for norms implies the equality of limits.
\end{proof}

For summable families $M=\left\{  a_{n}\right\}  $ and $N=\left\{
b_{n}\right\}  _{1}^{\infty}$, let $M\ast N=\left\{  c_{n}\right\}  $ denote
the convolution of $M$ and $N$: $c_{n}=\sum_{i+j=n+1}a_{i}b_{j}$ for every
$n>0$.

\begin{proposition}
\label{change}If $M=\left\{  a_{n}\right\}  _{1}^{\infty}$ and $N=\left\{
b_{n}\right\}  _{1}^{\infty}$ are summable families in $A$ then $\rho
_{t}\left(  M\ast N\right)  \leq\rho_{t}\left(  MN\right)  =\rho_{t}\left(
NM\right)  $ and $\rho_{t}\left(  M+N\right)  \leq\rho_{t}\left(  M\sqcup
N\right)  $.
\end{proposition}

\begin{proof}
Note that $\mathbf{\mathrm{\eta}}\left(  \left(  MN\right)  ^{n+1}\right)
\leq\mathbf{\mathrm{\eta}}\left(  M\right)  \mathbf{\mathrm{\eta}}\left(
\left(  NM\right)  ^{n}\right)  \mathbf{\mathrm{\eta}}\left(  N\right)  $ for
every $n$. This implies that $\rho_{t}\left(  MN\right)  \leq\rho_{t}\left(
NM\right)  $. Changing $M$ and $N$ by places, we have the equality.
We have that
\begin{align*}
\mathbf{\mathrm{\eta}}\left(  \left(  M\ast N\right)  ^{k}\right)   &
=\sum_{n_{1},\ldots,n_{k}}\left\Vert \left(  \sum_{i_{i}+j_{i}=n_{1}%
+1}a_{i_{1}}b_{j_{1}}\right)  \cdots\left(  \sum_{i_{k}+j_{k}=n_{k}+1}%
a_{i_{k}}b_{j_{k}}\right)  \right\Vert \\
&  \leq\sum_{n_{1},\ldots,n_{k}}\sum_{i_{i}+j_{i}=n_{1}+1}\cdots\sum
_{i_{k}+j_{k}=n_{k}+1}\left\Vert a_{i_{1}}b_{j_{1}}\cdots a_{i_{k}}b_{j_{k}%
}\right\Vert \\
&  =\sum_{n_{1},\ldots,n_{2k}}\left\Vert a_{n_{1}}b_{n_{2}}\cdots a_{n_{2k-1}%
}b_{n_{2k}}\right\Vert =\mathbf{\mathrm{\eta}}\left(  \left(  MN\right)
^{k}\right)
\end{align*}
for every $k>0$, whence $\rho_{t}\left(  M\ast N\right)  \leq\rho_{t}\left(
MN\right)  $.
Further, $M+N=\left\{  a_{n}+b_{n}\right\}  _{1}^{\infty}$ and
\begin{align*}
\left(  M+N\right)  ^{k}  &  \simeq\left\{  \left(  a_{n_{1}}+b_{n_{1}%
}\right)  \cdots\left(  a_{n_{k}}+b_{n_{k}}\right)  \right\}  _{n_{1}%
,\ldots,n_{k}=1}^{\infty}\\
&  =\left\{  a_{n_{1}}\cdots a_{n_{k}}+b_{n_{1}}a_{n_{2}}\cdots a_{n_{k}%
}+\ldots+b_{n_{1}}\cdots b_{n_{k}}\right\}  _{n_{1},\ldots,n_{k}=1}^{\infty}.
\end{align*}
Then%
\begin{align*}
\mathbf{\mathrm{\eta}}\left(  \left(  M+N\right)  ^{k}\right)   &
=\sum_{n_{1},\ldots,n_{k}}\left\Vert a_{n_{1}}\cdots a_{n_{k}}+b_{n_{1}%
}a_{n_{2}}\cdots a_{n_{k}}+\ldots+b_{n_{1}}\cdots b_{n_{k}}\right\Vert \\
&  \leq\sum_{n_{1},\ldots,n_{k}}\left(  \left\Vert a_{n_{1}}\cdots a_{n_{k}%
}\right\Vert +\left\Vert b_{n_{1}}a_{n_{2}}\cdots a_{n_{k}}\right\Vert
+\ldots+\left\Vert b_{n_{1}}\cdots b_{n_{k}}\right\Vert \right) \\
&  =\sum_{n_{1},\ldots,n_{k}}\left\Vert a_{n_{1}}\cdots a_{n_{k}}\right\Vert
+\sum_{n_{1},\ldots,n_{k}}\left\Vert b_{n_{1}}a_{n_{2}}\cdots a_{n_{k}%
}\right\Vert +\\
&  \ldots+\sum_{n_{1},\ldots,n_{k}}\left\Vert b_{n_{1}}\cdots b_{n_{k}%
}\right\Vert \\
&  =\mathbf{\mathrm{\eta}}\left(  M^{k}\right)  +\mathbf{\mathrm{\eta}}\left(
NM^{k-1}\right)  +\ldots+\mathbf{\mathrm{\eta}}\left(  N^{k}\right) \\
&  =\mathbf{\mathrm{\eta}}\left(  M^{k}\sqcup NM^{k-1}\sqcup\cdots\sqcup
N^{k}\right)  =\mathbf{\mathrm{\eta}}\left(  \left(  M\sqcup N\right)
^{k}\right)
\end{align*}
for every $k>0$, whence $\rho_{t}\left(  M+N\right)  \leq\rho_{t}\left(
M\sqcup N\right)  $.
\end{proof}

As $A$ is embedded into $A^{1}$, let $M^{0}$ be $\left\{  x_{n}\right\}
_{1}^{\infty}$ with $x_{n}=0$ for every $n>1$ and $x_{1}=1$, the identity
element of $A^{1}$. Note that $M^{0}M^{0}\simeq M^{0}\simeq N^{0}$,
$M^{0}N\simeq NM^{0}\simeq N$ for any family $N$ in $A$, and
$\mathbf{\mathrm{\eta}}\left(  M^{0}\right)  =1$.

We say that families $M$ and $N$ in $A$ \textit{commute} if $MN\simeq NM$.
This of course doesn't mean that elements of $M$ commute with elements of $N$.
But the reverse statement clearly keeps: if each element of $M$ commutes with
each element of $N$ then $M$ and $N$ commute. In particular, if $N$ consists
of elements of the center of $A$ then $M$ and $N$ commute.

\begin{proposition}
\label{commute}Let $M$ and $N$ be summable families in $A$. If $M$ and $N$
commute then $\rho_{t}\left(  MN\right)  \leq\rho_{t}\left(  M\right)
\rho_{t}\left(  N\right)  $ and $\rho_{t}\left(  M\sqcup N\right)  \leq
\rho_{t}\left(  M\right)  +\rho_{t}\left(  N\right)  $.
\end{proposition}

\begin{proof}
Indeed, $\left(  NM\right)  ^{n}\simeq N^{n}M^{n}$ and $\mathbf{\mathrm{\eta}%
}\left(  \left(  NM\right)  ^{n}\right)  =\mathbf{\mathrm{\eta}}\left(
N^{n}M^{n}\right)  \leq\mathbf{\mathrm{\eta}}\left(  N^{n}\right)
\mathbf{\mathrm{\eta}}\left(  M^{n}\right)  $ for every $n.$ Taking $n$-roots
and passing to limits, we obtain that
\[
\rho_{t}\left(  MN\right)  \leq\rho_{t}\left(  M\right)  \rho_{t}\left(
N\right)  .
\]
It is easy to see that
\[
\left(  M\sqcup N\right)  ^{n}\simeq\sqcup_{i=0}^{n}\left(  C_{n}^{i}%
\bullet\left(  M^{i}N^{n-i}\right)  \right)
\]
for every $n>0$, whence
\[
\mathbf{\mathrm{\eta}}\left(  \left(  M\sqcup N\right)  ^{n}\right)
=\sum_{i=0}^{n}C_{n}^{i}\mathbf{\mathrm{\eta}}\left(  M^{i}N^{n-i}\right)
\leq\sum_{i=0}^{n}C_{n}^{i}\mathbf{\mathrm{\eta}}\left(  M^{i}\right)
\mathbf{\mathrm{\eta}}\left(  N^{n-i}\right)  .
\]
Let $\varepsilon>0$, and take $s\geq1$ such that $\mathbf{\mathrm{\eta}%
}\left(  M^{i}\right)  \leq s\left(  \rho_{t}\left(  M\right)  +\varepsilon
\right)  ^{i}$ and $\mathbf{\mathrm{\eta}}\left(  N^{i}\right)  \leq s\left(
\rho_{t}\left(  N\right)  +\varepsilon\right)  ^{i}$ for every $i\in
\mathbb{N}$. Then
\begin{align*}
\mathbf{\mathrm{\eta}}\left(  \left(  M\sqcup N\right)  ^{n}\right)   &  \leq
s^{2}\sum_{i=0}^{n}C_{n}^{i}\left(  \rho_{t}\left(  M\right)  +\varepsilon
\right)  ^{i}\left(  \rho_{t}\left(  N\right)  +\varepsilon\right)  ^{n-i}\\
&  =s^{2}\left(  \rho_{t}\left(  M\right)  +\rho_{t}\left(  N\right)
+2\varepsilon\right)  ^{n}%
\end{align*}
for every $n>0$. Taking $n$-roots and passing to limits, we obtain that%
\[
\rho_{t}\left(  M\sqcup N\right)  \leq\rho_{t}\left(  M\right)  +\rho
_{t}\left(  N\right)  +2\varepsilon
\]
As $\varepsilon$ is arbitrary, we have that $\rho_{t}\left(  M\sqcup N\right)
\leq\rho_{t}\left(  M\right)  +\rho_{t}\left(  N\right)  $.
\end{proof}

\subsection{Absolutely convex hulls and tensor quasinilpotent families}

Let $A$ be a normed algebra. A summable family $M$ of elements of $A$ is
called \textit{tensor quasinilpotent }if $\rho_{t}\left(  M\right)  =0$.

The following result is an immediate consequence of Theorem \ref{tsr}(i).

\begin{corollary}
\label{explains} If a family $M$ in $A$ is tensor quasinilpotent then for each
bounded family $L$ in a normed algebra $B$ the element $M_{\otimes}L$ is
quasinilpotent in $A\widehat{\otimes}B$.
\end{corollary}

For a summable family $M=\left\{  a_{n}\right\}  _{1}^{\infty}$, let
$\operatorname*{abs}_{t}\left(  M\right)  $ denote the set of all families
$N=\left\{  b_{n}\right\}  _{1}^{\infty}$ such that $b_{m}=\sum_{n=1}^{\infty
}t_{nm}a_{n}$, where the sequences $\left\{  t_{nm}\right\}  _{m=1}^{\infty}%
$of complex numbers satisfy the condition $\sum_{m=1}^{\infty}\left\vert
t_{nm}\right\vert \leq1$. We call $\operatorname*{abs}_{t}\left(  M\right)  $
the \textit{absolutely convex hull} of $M$. To justify the term, note that
$\operatorname*{abs}_{t}\left(  M\right)  $ is a closed absolutely convex
subset of $\ell_{1}\left(  A\right)  $.

\begin{proposition}
\label{abs}If $M=\left\{  a_{n}\right\}  _{1}^{\infty}$ is a summable family
of elements of a normed algebra $A$ then $\rho_{t}\left(  N\right)  \leq
\rho_{t}\left(  M\right)  $ for any $N=\left\{  b_{n}\right\}  _{1}^{\infty
}\in\operatorname*{abs}_{t}\left(  M\right)  $.
\end{proposition}

\begin{proof}
Indeed,%
\begin{align*}
\eta(N^{k})  &  =\sum_{m_{1},\ldots,m_{k}}\left\Vert b_{m_{1}}\cdots b_{m_{k}%
}\right\Vert =\sum_{m_{1},\ldots,m_{k}}\left\Vert \sum_{n_{1},\ldots,n_{k}%
}t_{n_{1}m_{1}}\cdots t_{n_{k}m_{k}}a_{n_{1}}\cdots a_{n_{k}}\right\Vert \\
&  \leq\sum_{m_{1},\ldots,m_{k}}\sum_{n_{1},\ldots,n_{k}}\left\vert
t_{n_{1}m_{1}}\cdots t_{n_{k}m_{k}}\right\vert \left\Vert a_{n_{1}}\cdots
a_{n_{k}}\right\Vert \\
&  =\sum_{n_{1},\ldots,n_{k}}\sum_{m_{1}}\left\vert t_{n_{1}m_{1}}\right\vert
\cdots\sum_{m_{k}}\left\vert t_{n_{k}m_{k}}\right\vert \left\Vert a_{n_{1}%
}\cdots a_{n_{k}}\right\Vert \\
&  \leq\sum_{n_{1},\ldots,n_{k}}\left\Vert a_{n_{1}}\cdots a_{n_{k}%
}\right\Vert =\eta(M^{k})
\end{align*}
for every $k$, whence $\rho_{t}\left(  N\right)  \leq\rho_{t}\left(  M\right)
$.
\end{proof}

Recall that a set $K$ in a normed space $X$ is called \textit{absolutely
convex }if $t_{1}x_{1}+t_{2}x_{2}\in K$ for every integer $n>0$ and for any
$x_{1},x_{2}\in K$ and $t_{1},t_{2}\in\mathbb{C}$ with $\left\vert
t_{1}\right\vert +\left\vert t_{2}\right\vert \leq1$. If $K$ is a compact set
in $X$, then the number $\max\left\{  \left\Vert x-y\right\Vert :x,y\in
K\right\}  $ is called the \textit{diameter} of $K$ and denoted by
$\operatorname*{diam}\left(  K\right)  $.

\begin{lemma}
\label{general} Let $X$ be a Banach space and let $\left\{  K_{n}\right\}  $
be a sequence of absolutely convex compact sets in $X$ such that
$\sum\operatorname*{diam}\left(  K_{n}\right)  <\infty$. Then $\sum K_{n}$ is
an absolutely convex compact set.
\end{lemma}

\begin{proof}
It is clear that $\sum K_{n}$ is absolutely convex. To see that it is compact,
note that the direct product $K=\Pi_{n=1}^{\infty}K_{n}$ is compact. Since
$0\in K_{n}$, we get that $\Vert a\Vert\leq\operatorname*{diam}(K_{n})/2$ for
$a\in K_{n}$. So there are numbers $\alpha_{n}>0$ with $\Vert a\Vert\leq
\alpha_{n}$ for $a\in K_{n}$, such that $\sum\alpha_{n}<\infty$. It follows
that the map $\varphi:K\longrightarrow X$ defined by the formula
\[
\varphi(\{a_{n}\}_{n=1}^{\infty})=\sum_{n=1}^{\infty}a_{n},
\]
is continuous. As $\sum K_{n}=\varphi(K)$, it is compact.
\end{proof}

Let $M=\left\{  a_{n}\right\}  _{1}^{\infty}$ be a summable family in a Banach
algebra $A$, and let $\Omega\left(  M\right)  $ denote the set of all elements
of the form $\sum t_{n}a_{n}$ for complex numbers $\left\vert t_{n}\right\vert
\leq1$.

\begin{corollary}
If $M$ is a summable family in a Banach algebra $A$ then $\Omega\left(
M\right)  $ is an absolutely convex compact set such that $\rho\left(
a\right)  \leq\rho_{t}\left(  M\right)  $ for any $a\in\Omega\left(  M\right)
$. Moreover, $\Omega\left(  M\right)  =\left\{  \sum b_{n}:\left\{
b_{n}\right\}  _{1}^{\infty}\in\operatorname*{abs}_{t}\left(  M\right)
\right\}  $.
\end{corollary}

\begin{proof}
Indeed, $\Omega\left(  M\right)  $ is the countable sum of absolutely convex
compact sets $\left\{  ta_{n}:\left\vert t\right\vert \leq1\right\}  $, the
sum of whose diameters is finite. So $\Omega\left(  M\right)  $ is a convex
compact set in $A$. Further, for any $a=\sum t_{n}a_{n}\in\Omega\left(
M\right)  $, we have that
\[
\left\Vert a^{k}\right\Vert \leq\sum_{n_{1},\ldots,n_{k}}\left\Vert t_{n_{1}%
}\cdots t_{n_{k}}a_{n_{1}}\cdots a_{n_{k}}\right\Vert \leq\sum_{n_{1}%
,\ldots,n_{k}}\left\Vert a_{n_{1}}\cdots a_{n_{k}}\right\Vert
=\mathbf{\mathrm{\eta}}\left(  M^{k}\right)
\]
for every $k\in\mathbb{N}$. So $\rho\left(  a\right)  \leq\rho_{t}\left(
M\right)  $.
The last assertion easily follows from well known properties of absolutely
summable series.
\end{proof}

\begin{corollary}
\label{gen}Let $M=\left\{  a_{n}\right\}  _{1}^{\infty}$ be a tensor
quasinilpotent family in a Banach algebra $A$. Then every element of
$\Omega\left(  M\right)  $ is quasinilpotent.
\end{corollary}

\begin{lemma}
\label{ppp}If $M=\left\{  a_{n}\right\}  _{1}^{\infty}$ is a summable family
in $A$ such that $\rho_{t}\left(  M\right)  <1$ then the family $\sqcup
_{m=1}^{\infty}M^{m}$ is summable and $\rho_{t}\left(  \sqcup_{m=1}^{\infty
}M^{m}\right)  =\rho_{t}\left(  M\right)  \left(  1-\rho_{t}\left(  M\right)
\right)  ^{-1}$.
\end{lemma}

\begin{proof}
Let $M_{\left(  k\right)  }\simeq\sqcup_{n=k}^{\infty}M^{n}$ for any $k$. Take
$t$ satisfying $\rho_{t}\left(  M\right)  <t<1$. Then there is an integer $p$
such that $\mathbf{\mathrm{\eta}}\left(  M^{n}\right)  \leq t^{n}$ for every
$n\geq p$. Then%
\[
\sum_{n=1}^{\infty}\mathbf{\mathrm{\eta}}\left(  M^{n}\right)  <\sum
_{n=1}^{p-1}\mathbf{\mathrm{\eta}}\left(  M^{n}\right)  +\sum_{n=p}^{\infty
}t^{n}=\sum_{n=1}^{p-1}\mathbf{\mathrm{\eta}}\left(  M^{n}\right)
+\frac{t^{p}}{1-t}<\infty.
\]
As $\mathbf{\mathrm{\eta}}\left(  \sqcup_{n=1}^{\infty}M^{n}\right)
=\sum_{n=1}^{\infty}\mathbf{\mathrm{\eta}}\left(  M^{n}\right)  <\infty$,
$M_{\left(  1\right)  }$ is summable.
Since $M_{\left(  2\right)  }\simeq MM_{\left(  1\right)  }\simeq M_{\left(
1\right)  }M$ and $M_{\left(  1\right)  }\simeq M\sqcup M_{\left(  2\right)
}\simeq M\left(  M^{0}\sqcup M_{\left(  1\right)  }\right)  $, we obtain that
\[
\rho_{t}\left(  M_{\left(  1\right)  }\right)  \leq\rho_{t}\left(  M\right)
\rho_{t}\left(  M^{0}\sqcup M_{\left(  1\right)  }\right)  \leq\rho_{t}\left(
M\right)  \left(  \rho_{t}\left(  M_{\left(  1\right)  }\right)  +1\right)
\]
by Proposition \ref{commute}, whence $\rho_{t}\left(  M_{\left(  1\right)
}\right)  \leq\rho_{t}\left(  M\right)  \left(  1-\rho_{t}\left(  M\right)
\right)  ^{-1}$.
On the other hand, we have, using (\ref{f62}), that
\begin{align*}
\rho_{t}\left(  M\right)  ^{n}\left(  1-\rho_{t}\left(  M\right)  \right)
^{-n}  &  =\left(  \rho_{t}\left(  M\right)  +\rho_{t}\left(  M\right)
^{2}+\rho_{t}\left(  M\right)  ^{3}+\ldots\right)  ^{n}\\
&  =\rho_{t}\left(  M\right)  ^{n}+n\rho_{t}\left(  M\right)  ^{n+1}%
+\frac{n(n+1)}{2}\rho_{t}\left(  M\right)  ^{n+2}+\ldots\\
&  \overset{(\ref{f62})}{=}\rho_{t}\left(  M^{n}\right)  +n\rho_{t}\left(
M^{n+1}\right)  +\frac{n(n+1)}{2}\rho_{t}\left(  M^{n+2}\right)  +\ldots\\
&  \leq\mathbf{\mathrm{\eta}}\left(  M^{n}\right)  +n\mathbf{\mathrm{\eta}%
}\left(  M^{n+1}\right)  +\frac{n(n+1)}{2}\mathbf{\mathrm{\eta}}\left(
M^{n+2}\right)  +\ldots\\
&  =\mathbf{\mathrm{\eta}}\left(  M^{n}\sqcup n\bullet M^{n+1}\sqcup
\frac{n(n+1)}{2}\bullet M^{n+2}\sqcup\cdots\right) \\
&  =\mathbf{\mathrm{\eta}}\left(  \left(  M\sqcup M^{2}\sqcup M^{3}%
\sqcup\cdots\right)  ^{n}\right)  =\mathbf{\mathrm{\eta}}\left(  M_{\left(
1\right)  }^{n}\right)
\end{align*}
for every $n>0$, whence $\rho_{t}\left(  M\right)  \left(  1-\rho_{t}\left(
M\right)  \right)  ^{-1}\leq\rho_{t}\left(  M_{\left(  1\right)  }\right)  $.
\end{proof}

\begin{corollary}
If $M=\left\{  a_{n}\right\}  _{1}^{\infty}$ is a tensor quasinilpotent family
in a normed algebra $A$ then the subalgebra generated by $M$ consists of quasinilpotents.
\end{corollary}

\begin{proof}
Indeed, $\rho_{t}\left(  \sqcup_{m=1}^{\infty}M^{m}\right)  =0$ by Lemma
\ref{ppp}, and every element of the subalgebra lies in $\cup_{t>0}%
\,t\,\Omega\left(  \sqcup_{m=1}^{\infty}M^{m}\right)  $ which consists of
quasinilpotents by Corollary \ref{gen}.
\end{proof}

\subsection{Upper semicontinuity and subharmonicity of the tensor spectral
radius}

Let $A$ be a normed algebra. Since sequences in $\ell_{1}\left(  {A}\right)  $
determine summable families, the tensor spectral radius can be considered as a
function on $\ell_{1}(A)$. We are going to show that this function is upper
semicontinuous and subharmonic.

If $G$ is a subset of $A$, let $\mathrm{F}_{1}\left(  G\right)  $ be the set
of all summable families $M=\left\{  a_{n}\right\}  _{1}^{\infty}$ with all
$a_{n}\in G$. In particular, $\mathrm{F}_{1}(A)=\ell_{1}(A)$. Clearly
$\mathrm{F}_{1}\left(  G\right)  $ is a metric space with respect to the
metric $\mathrm{d}(M,N)=\eta(M-N)$ induced by the norm on $\ell_{1}(A)$. If
$A$ is a Banach algebra and $G$ is a closed subset of $A$ then $\mathrm{F}%
_{1}\left(  G\right)  $ is a complete metric space.

Now we are able to establish the upper semicontinuity of the tensor spectral radius.

\begin{proposition}
Let $M\in\ell_{1}\left(  A\right)  $. For every $\varepsilon>0$, there is
$\delta>0$ such that $\rho_{t}\left(  N\right)  \leq\rho_{t}\left(  M\right)
+\varepsilon$ for every $N\in\ell_{1}\left(  A\right)  $ satisfying
$\mathrm{d}\left(  N,M\right)  <\delta$.
\end{proposition}

\begin{proof}
Let $T:M\longmapsto M_{\otimes}L$ be the map defined in Theorem \ref{tsr}(ii).
As the usual spectral radius is upper semicontinuous, to any $\varepsilon>0$,
there corresponds $\delta>0$ such that $\rho\left(  T\left(  N\right)
\right)  \leq\rho\left(  T\left(  M\right)  \right)  +\varepsilon$ for every
$N\in\ell_{1}(A)$ satisfying $\left\Vert T\left(  N\right)  -T\left(
M\right)  \right\Vert <\delta$. By Theorem \ref{tsr}(ii), we get that
$\rho\left(  T\left(  N\right)  \right)  =\rho_{t}\left(  N\right)  $,
$\rho\left(  T\left(  M\right)  \right)  =\rho_{t}\left(  M\right)  $ and
$\left\Vert T\left(  N\right)  -T\left(  M\right)  \right\Vert =\left\Vert
T\left(  N-M\right)  \right\Vert =\eta(M-N)=\mathrm{d}(M,N)$.
\end{proof}

Let $\digamma$ be an analytic function on an open set $\mathcal{D}%
\subset\mathbb{C}$ with values in $\ell_{1}\left(  A\right)  $. This means
that for every $\lambda\in\mathcal{D}$ there is $\digamma^{\prime}\left(
\lambda\right)  \in\ell_{1}\left(  A\right)  $ such that
\[
\digamma^{\prime}\left(  \lambda\right)  =\lim_{\mu\rightarrow\lambda}%
\frac{\digamma\left(  \mu\right)  -\digamma\left(  \lambda\right)  }%
{\mu-\lambda}%
\]
in the norm of $\ell_{1}\left(  A\right)  $. Clearly $\digamma$ as an
$\ell_{1}\left(  A\right)  $-valued function induces the family $\left\{
f_{n}\right\}  _{1}^{\infty}$ of $A$-valued functions on $\mathcal{D}$:
$\digamma=\left\{  f_{n}\right\}  _{1}^{\infty}$. These functions are analytic
since $\lambda\longmapsto\digamma\left(  \lambda\right)  =\left\{
f_{n}\left(  \lambda\right)  \right\}  _{1}^{\infty}$ is analytic on
$\mathcal{D}$. So one can write $\digamma^{\prime}=\left\{  f_{n}^{\prime
}\right\}  _{1}^{\infty}$, etc.

Let now $\digamma=\left\{  f_{n}\right\}  _{1}^{\infty}$ and $\Psi=\left\{
\psi_{n}\right\}  _{1}^{\infty}$ be $\ell_{1}\left(  A\right)  $-valued
functions on $\mathcal{D}$. Let $\digamma\Psi$ denote some sequence $\left\{
\varphi_{k}\right\}  _{1}^{\infty}$ of functions on $\mathcal{D}$ which is
obtained from $\left\{  f_{n}\psi_{m}\right\}  _{n,m=1}^{\infty}$ by
renumbering. We will assume that the renumbering for this operation is fixed,
so $\digamma\Psi$ is defined correctly. In such a case the functions
$\varphi_{k}$ are $A$-valued functions on $\mathcal{D}$, and we write that
$\digamma\Psi=\left\{  \varphi_{k}\right\}  $ is an $\ell_{1}\left(  A\right)
$-valued function on $\mathcal{D}$.

\begin{lemma}
\label{anal}If $\digamma=\left\{  f_{n}\right\}  _{1}^{\infty}$ and
$\Psi=\left\{  \psi_{n}\right\}  _{1}^{\infty}$ are analytic $\ell_{1}\left(
A\right)  $-valued function on $\mathcal{D}$ then $\digamma\Psi$ is an
analytic $\ell_{1}\left(  A\right)  $-valued function on $\mathcal{D}$.
\end{lemma}

\begin{proof}
Indeed, the derivation $\left(  \digamma\Psi\right)  ^{\prime}$ exists and is
clearly obtained from two-index sequence $\left\{  f_{n}^{\prime}\psi
_{m}+f_{n}\psi_{m}^{\prime}\right\}  _{n,m=1}^{\infty}$ by the same
renumbering as $\digamma\Psi$ from $\left\{  f_{n}\psi_{m}\right\}
_{n,m=1}^{\infty}$. Moreover,
\begin{align*}
\mathbf{\mathrm{\eta}}\left(  \left(  \digamma\Psi\right)  ^{\prime}\left(
\lambda\right)  \right)   &  =\sum_{n,m}\left\Vert f_{n}^{\prime}\left(
\lambda\right)  \psi_{m}\left(  \lambda\right)  +f_{n}\left(  \lambda\right)
\psi_{m}^{\prime}\left(  \lambda\right)  \right\Vert \\
&  \leq\sum_{n,m}\left(  \left\Vert f_{n}^{\prime}\left(  \lambda\right)
\right\Vert \left\Vert \psi_{m}\left(  \lambda\right)  \right\Vert +\left\Vert
f_{n}\left(  \lambda\right)  \right\Vert \left\Vert \psi_{m}^{\prime}\left(
\lambda\right)  \right\Vert \right) \\
&  =\sum_{n}\left\Vert f_{n}^{\prime}\left(  \lambda\right)  \right\Vert
\sum_{m}\left\Vert \psi_{m}\left(  \lambda\right)  \right\Vert +\sum
_{n}\left\Vert f_{n}\left(  \lambda\right)  \right\Vert \sum_{m}\left\Vert
\psi_{m}^{\prime}\left(  \lambda\right)  \right\Vert \\
&  =\mathbf{\mathrm{\eta}}\left(  \digamma^{\prime}\left(  \lambda\right)
\right)  \mathbf{\mathrm{\eta}}\left(  \Psi\left(  \lambda\right)  \right)
+\mathbf{\mathrm{\eta}}\left(  \digamma\left(  \lambda\right)  \right)
\mathbf{\mathrm{\eta}}\left(  \Psi^{\prime}\left(  \lambda\right)  \right)
<\infty.
\end{align*}
\end{proof}

For an $\ell_{1}\left(  A\right)  $-valued function $\digamma=\left\{
f_{n}\right\}  _{1}^{\infty}$, let $\digamma^{1}=\digamma$ and $\digamma
^{m}=\digamma^{m-1}\digamma$ for $m>1$, so that $\digamma^{m}=\left\{
\phi_{n}\right\}  _{1}^{\infty}$ is an $\ell_{1}\left(  A\right)  $-valued
function for some $A$-valued functions $\phi_{n}$ on $\mathcal{D}$. It follows
by induction from the definition of product of two functions that
\begin{equation}
\digamma^{m}\left(  \lambda\right)  \simeq\digamma\left(  \lambda\right)  ^{m}
\label{f65}%
\end{equation}
for every $\lambda\in\mathcal{D}$.

\begin{corollary}
\label{anals}If $\digamma$ is an analytic $\ell_{1}\left(  A\right)  $-valued
function on $\mathcal{D}$ then $\digamma^{m}$ is an analytic $\ell_{1}\left(
A\right)  $-valued function on $\mathcal{D}$ for every $m>0$.
\end{corollary}

\begin{proof}
Follows from Lemma \ref{anal} by induction.
\end{proof}

For an $\ell_{1}\left(  A\right)  $-valued function $\digamma$, the function
$\lambda\longmapsto\mathbf{\mathrm{\eta}}\left(  \digamma\left(
\lambda\right)  ^{m}\right)  $ doesn't depend on a renumbering. Moreover, it
follows from (\ref{f65}) that
\begin{equation}
\mathbf{\mathrm{\eta}}\left(  \digamma\left(  \lambda\right)  ^{m}\right)
=\mathbf{\mathrm{\eta}}\left(  \digamma^{m}\left(  \lambda\right)  \right)
\label{f66}%
\end{equation}
for every $\lambda\in\mathcal{D}$.

\begin{lemma}
If $\digamma$ is an analytic $\ell_{1}\left(  A\right)  $-valued function on
$\mathcal{D}$ then the functions $\lambda\longmapsto\mathbf{\mathrm{\eta}%
}\left(  \digamma\left(  \lambda\right)  ^{m}\right)  $ and $\lambda
\longmapsto\log\left(  \mathbf{\mathrm{\eta}}\left(  \digamma\left(
\lambda\right)  ^{m}\right)  \right)  $ are subharmonic on $\mathcal{D}$ for
all $m$.
\end{lemma}

\begin{proof}
As \textbf{$\mathrm{\eta}$}$\left(  \cdot\right)  $ is a norm on $\ell
_{1}\left(  A\right)  $, it is well known (see \cite{Ves}) that $\lambda
\longmapsto\mathbf{\mathrm{\eta}}\left(  \digamma\left(  \lambda\right)
\right)  $ is a subharmonic function. Then it follows from Corollary
\ref{anals} and (\ref{f66}) that $\lambda\longmapsto\mathbf{\mathrm{\eta}%
}\left(  \digamma\left(  \lambda\right)  ^{m}\right)  $ is subharmonic for
every $m$.
Let $\digamma^{m}=\left\{  \phi_{n}\right\}  _{1}^{\infty}$ and $\beta
\in\mathbb{C}$. As
\[
\left\vert \exp\left(  \beta\lambda\right)  \right\vert \mathbf{\mathrm{\eta}%
}\left(  \digamma\left(  \lambda\right)  ^{m}\right)  =\mathbf{\mathrm{\eta}%
}\left(  \left\{  \exp\left(  \beta\lambda\right)  \phi_{n}\left(
\lambda\right)  \right\}  _{1}^{\infty}\right)
\]
by (\ref{f66}) and $\lambda\longmapsto\left\{  \exp\left(  \beta
\lambda\right)  \phi_{n}\left(  \lambda\right)  \right\}  _{1}^{\infty}$
determines an analytic $\ell_{1}\left(  A\right)  $-valued function on
$\mathcal{D}$, then, by above, $\lambda\longmapsto\left\vert \exp\left(
\beta\lambda\right)  \right\vert \mathbf{\mathrm{\eta}}\left(  \digamma\left(
\lambda\right)  ^{m}\right)  $ is subharmonic for every $\beta\in\mathbb{C}$.
It follows from Rado's theorem \cite[Appendix 2, Theorem 9]{Aup} that the
function $\lambda\longmapsto\log\left(  \mathbf{\mathrm{\eta}}\left(
\digamma\left(  \lambda\right)  ^{m}\right)  \right)  $ is subharmonic on
$\mathcal{D}$.
\end{proof}

We use Vesentini's argument for subharmonicity of the usual spectral radius
\cite{Ves} in the following

\begin{theorem}
\label{subhar}If $\digamma$ is an analytic $\ell_{1}\left(  A\right)  $-valued
function on $\mathcal{D}$ then the functions $\log\left(  \rho_{t}\left(
\digamma\right)  \right)  :\lambda\longmapsto\log\left(  \rho_{t}\left(
\digamma\left(  \lambda\right)  \right)  \right)  $ and $\rho_{t}\left(
\digamma\right)  :\lambda\longmapsto\rho_{t}\left(  \digamma\left(
\lambda\right)  \right)  $ are subharmonic on $\mathcal{D}$.
\end{theorem}

\begin{proof}
As $\mathbf{\mathrm{\eta}}\left(  \digamma\left(  \lambda\right)  ^{2^{m+1}%
}\right)  \leq\mathbf{\mathrm{\eta}}\left(  \digamma\left(  \lambda\right)
^{2m}\right)  ^{2}$, the function $\lambda\longmapsto\log\left(  \rho
_{t}\left(  \digamma\left(  \lambda\right)  \right)  \right)  $ is a pointwise
limit of the decreasing sequence $\left\{  \lambda\longmapsto2^{-m}\log\left(
\mathbf{\mathrm{\eta}}\left(  \digamma\left(  \lambda\right)  ^{2m}\right)
\right)  \right\}  $ of subharmonic functions and is therefore a subharmonic
function by Theorem 1 of \cite[Appendix 2]{Aup}. Since the function
$t\longmapsto\exp\left(  t\right)  $ is convex and positive for $t\in
\mathbb{R}$, then the function $\lambda\longmapsto\exp\left(  \log\left(
\rho_{t}\left(  \digamma\left(  \lambda\right)  \right)  \right)  \right)
=\rho_{t}\left(  \digamma\left(  \lambda\right)  \right)  $ is also
subharmonic by the same theorem.
\end{proof}

\subsection{The ideal $\mathcal{R}_{t}\left(  A\right)  $ for a normed algebra
$A$\label{sid}}

Let $a$ be an element of a normed algebra $A$, and let $M=\left\{
a_{n}\right\}  _{n=1}^{\infty}$ be a summable family in $A$. Let $\left\{
a\right\}  \sqcup M$ denote the family $\left\{  x_{n}\right\}  _{1}^{\infty}$
with $x_{1}=a$ and $x_{n}=a_{n-1}$ for $n>1$. Note that the useful relation
\begin{equation}
\left(  \left\{  a\right\}  \sqcup M\right)  \left(  \left\{  b\right\}
\sqcup N\right)  \simeq\left\{  ab\right\}  \sqcup aN\sqcup Mb\sqcup MN
\label{fdeco}%
\end{equation}
is valid by our conventions, for any $a,b\in A$ and $M,N\in\ell_{1}\left(
A\right)  $.

Let $\mathcal{R}_{t}\left(  A\right)  $ be the set of all $a\in A$ such that
$\rho_{t}\left(  \left\{  a\right\}  \sqcup M\right)  =\rho_{t}\left(
M\right)  $ for every $M\in\ell_{1}\left(  A\right)  $. It is evident that
$\mathcal{R}_{t}\left(  A\right)  $ consists of quasinilpotent elements of $A$.

\begin{lemma}
\label{ox}Let $a\in A$. If there is $s>0$ such that $\rho_{t}\left(  \left\{
a\right\}  \sqcup M\right)  \leq s\rho_{t}\left(  M\right)  $ for every
$M\in\ell_{1}\left(  A\right)  $ then $a\in\mathcal{R}_{t}\left(  A\right)  $.
\end{lemma}

\begin{proof}
As the function $\mu\longmapsto\rho_{t}\left(  \left\{  \mu a\right\}  \sqcup
M\right)  $ is subharmonic by Theorem \ref{subhar} and bounded on $\mathbb{C}%
$, it is constant, whence $\rho_{t}\left(  \left\{  a\right\}  \sqcup
M\right)  =\rho_{t}\left(  M\right)  $.
\end{proof}

\begin{lemma}
\label{unital}$\mathcal{R}_{t}\left(  A\right)  =A\cap\mathcal{R}_{t}\left(
A^{1}\right)  $. If $A$ is complete then $\mathcal{R}_{t}\left(  A\right)
=\mathcal{R}_{t}\left(  A^{1}\right)  $.
\end{lemma}

\begin{proof}
Let $A$ be non-unital and $a\in\mathcal{R}_{t}\left(  A\right)  $. As
$A^{1}=A\oplus\mathbb{C}$, for every summable family $M$ in $A^{1}$ there are
$N\in\ell_{1}\left(  A\right)  $ and $K\in\ell_{1}\left(  \mathbb{C}\right)  $
such that $M=N+K$. Since $N$ and $K$ commute, then
\[
\rho_{t}\left(  M\right)  =\rho_{t}\left(  N+K\right)  \leq\rho_{t}\left(
N\sqcup K\right)  \leq\rho_{t}\left(  N\right)  +\rho_{t}\left(  K\right)
\]
by Propositions \ref{change} and \ref{commute}. Hence, as $\left\{  \mu
a\right\}  \sqcup N$ and $\left\{  0\right\}  \sqcup K$ commute,
\begin{align*}
\rho_{t}\left(  \left\{  \mu a\right\}  \sqcup M\right)   &  \leq\rho
_{t}\left(  \left\{  \mu a\right\}  \sqcup N+\left\{  0\right\}  \sqcup
K\right) \\
&  \leq\rho_{t}\left(  \left\{  \mu a\right\}  \sqcup N\right)  +\rho
_{t}\left(  \left\{  0\right\}  \sqcup K\right)  =\rho_{t}\left(  N\right)
+\rho_{t}\left(  K\right)  <\infty
\end{align*}
for every $\mu\in\mathbb{C}$. Therefore $\mu\longmapsto\rho_{t}\left(
\left\{  \mu a\right\}  \sqcup M\right)  $ is constant, and as a consequence,
\[
\rho_{t}\left(  \left\{  a\right\}  \sqcup M\right)  =\rho_{t}\left(
M\right)  .
\]
As $M$ is arbitrary, $a\in\mathcal{R}_{t}\left(  A^{1}\right)  $. So
$\mathcal{R}_{t}\left(  A\right)  \subset\mathcal{R}_{t}\left(  A^{1}\right)
$.
On the other hand, $A\cap\mathcal{R}_{t}\left(  A^{1}\right)  \subset
\mathcal{R}_{t}\left(  A\right)  $ by definition. So we obtain that
\[
\mathcal{R}_{t}\left(  A\right)  =A\cap\mathcal{R}_{t}\left(  A^{1}\right)  .
\]
Let now $A$ be complete. We show that $\mathcal{R}_{t}\left(  A^{1}\right)
\subset A$. Indeed, if $a-\lambda\in\mathcal{R}_{t}\left(  A^{1}\right)  $
with $a\in A$ and $\lambda\in\mathbb{C}$ then $a-\lambda$ is a quasinilpotent
element of $A^{1}$. This means that the spectrum $\sigma\left(  a\right)  $ of
$a$ is equal to $\left\{  \lambda\right\}  $. But, as $a\in A$ and $A$ is not
unital, $\sigma\left(  a\right)  $ contains zero. Therefore $\lambda=0$.
So, if $A$ is complete, $\mathcal{R}_{t}\left(  A^{1}\right)  \subset A$,
whence $\mathcal{R}_{t}\left(  A^{1}\right)  =\mathcal{R}_{t}\left(  A\right)
$ by above.
\end{proof}

\begin{theorem}
$\mathcal{R}_{t}\left(  A\right)  $ is a closed ideal of $A$.
\end{theorem}

\begin{proof}
Consider first the case when $A$ has the identity element $1$. Let
$a,b\in\mathcal{R}_{t}\left(  A\right)  $. As
\[
\rho_{t}\left(  \left\{  \mu a\right\}  \sqcup M\right)  =\left\vert
\mu\right\vert \rho_{t}\left(  \left\{  a\right\}  \sqcup\mu^{-1}M\right)
=\left\vert \mu\right\vert \rho_{t}\left(  \mu^{-1}M\right)  =\rho_{t}\left(
M\right)
\]
for every $M\in\ell_{1}\left(  A\right)  $ and non-zero $\mu\in\mathbb{C}$,
then $\mu a\in\mathcal{R}_{t}\left(  A\right)  $ for every $\mu\in\mathbb{C}$.
Since $\left\{  2^{-1}\left(  a+b\right)  \right\}  \sqcup M\in
\operatorname*{abs}_{t}\left(  \left\{  a\right\}  \sqcup\left\{  b\right\}
\sqcup M\right)  $ for every $M\in\ell_{1}\left(  A\right)  $, then, by
Proposition \ref{abs},
\[
\rho_{t}\left(  \left\{  2^{-1}\left(  a+b\right)  \right\}  \sqcup M\right)
\leq\rho_{t}\left(  \left\{  a\right\}  \sqcup\left\{  b\right\}  \sqcup
M\right)  =\rho_{t}\left(  M\right)  ,
\]
whence $a+b\in\mathcal{R}_{t}\left(  A\right)  $ by Lemma \ref{ox}. So
$\mathcal{R}_{t}\left(  A\right)  $ is a subspace of $A$.
Let $x\in A$. Then
\[
\rho_{t}\left(  \left\{  \mu a\right\}  \sqcup\left\{  1\right\}
\sqcup\left\{  x\right\}  \sqcup M\right)  \leq t_{0}:=\rho_{t}\left(
\left\{  1\right\}  \sqcup\left\{  x\right\}  \sqcup M\right)  <\infty
\]
for every $\mu\in\mathbb{C}$. Therefore
\[
\rho_{t}\left(  \left(  \left\{  \mu a\right\}  \sqcup\left\{  1\right\}
\sqcup\left\{  x\right\}  \sqcup M\right)  ^{2}\right)  \leq t_{0}^{2}%
\]
by (\ref{f62}). Since $\left\{  \mu ax\right\}  \sqcup M$ is a subfamily of
$\left(  \left\{  \mu a\right\}  \sqcup\left\{  1\right\}  \sqcup\left\{
x\right\}  \sqcup M\right)  ^{2}$ in virtue of (\ref{fdeco}), we obtain that
\[
\rho_{t}\left(  \left\{  \mu ax\right\}  \sqcup M\right)  \leq t_{0}^{2}%
\]
for every $\mu\in\mathbb{C}$ and for every $M\in\ell_{1}\left(  A\right)  $
with $t_{0}$ depending only on $x$ and $M$. Therefore $\mu\longmapsto\rho
_{t}\left(  \left\{  \mu ax\right\}  \sqcup M\right)  $ is bounded on
$\mathbb{C}$. As this function is subharmonic, it is constant, whence
$ax\in\mathcal{R}_{t}\left(  A\right)  $, and, similarly, $xa\in
\mathcal{R}_{t}\left(  A\right)  $. Thus $\mathcal{R}_{t}\left(  A\right)  $
is an ideal of $A$.
Note that
\begin{align*}
\rho_{t}\left(  \left\{  a+x\right\}  \sqcup M\right)   &  \leq\rho_{t}\left(
\left\{  2a\right\}  \sqcup\left\{  2x\right\}  \sqcup M\right)  =\rho
_{t}\left(  \left\{  2x\right\}  \sqcup M\right) \\
&  \leq\mathbf{\mathrm{\eta}}\left(  \left\{  2x\right\}  \sqcup M\right)
=\left\Vert 2x\right\Vert +\mathbf{\mathrm{\eta}}\left(  M\right)
\end{align*}
for every $a\in\mathcal{R}_{t}\left(  A\right)  $ and $x\in A$. Now if $c$ is
in the closure of $\mathcal{R}_{t}\left(  A\right)  $, then for every $\mu
\in\mathbb{C}$ there are $a\in\mathcal{R}_{t}\left(  A\right)  $ and $x\in A$
with $\left\Vert x\right\Vert \leq1$ such that $\mu c=a+x$. Hence
\[
\rho_{t}\left(  \left\{  \mu c\right\}  \sqcup M\right)  =\rho_{t}\left(
\left\{  a+x\right\}  \sqcup M\right)  \leq2+\mathbf{\mathrm{\eta}}\left(
M\right)
\]
for every $M\in\ell_{1}\left(  A\right)  $. So $\mu\longmapsto\rho_{t}\left(
\left\{  \mu c\right\}  \sqcup M\right)  $ is bounded and therefore constant,
whence $c\in\mathcal{R}_{t}\left(  A\right)  $. Thus $\mathcal{R}_{t}\left(
A\right)  $ is a closed ideal of $A$.
Now assume that $A$ is not unital. We already have proved that $\mathcal{R}%
_{t}\left(  A^{1}\right)  $ is a closed ideal of $A^{1}$. Then Lemma
\ref{unital} shows that $\mathcal{R}_{t}\left(  A\right)  $ is a closed ideal
of $A$.
\end{proof}

\begin{corollary}
\label{rad}If $A$ is a Banach algebra then $\mathcal{R}_{t}\left(  A\right)
\subset\mathrm{Rad}\left(  A\right)  $.
\end{corollary}

\begin{proof}
Indeed, $\mathcal{R}_{t}\left(  A\right)  $ is an ideal of $A$ consisting of
quasinilpotents. So we have that $\mathcal{R}_{t}\left(  A\right)
\subset\mathrm{Rad}\left(  A\right)  $.
\end{proof}

\begin{theorem}
\label{multi}If $\rho_{t}\left(  aM\right)  =0$ for every $M\in\ell_{1}\left(
A\right)  $ then $a\in\mathcal{R}_{t}\left(  A\right)  $.
\end{theorem}

\begin{proof}
Let first $A$ have the identity element $1$. Let $M\in\ell_{1}\left(
A\right)  $. Multiplying $M$ by a scalar, one can assume that%
\begin{equation}
\mathbf{\mathrm{\eta}}\left(  M\right)  <1. \label{f70}%
\end{equation}
Then $N:=\sqcup_{i=0}^{\infty}M^{i}$ is a summable family in $A$ by Lemma
\ref{ppp}, where $M^{0}=\left\{  x_{n}\right\}  _{n=1}^{\infty}$ with
$x_{1}=1$ and $x_{i}=0$ for every $i>1$ as usual. By condition, we have that
$\rho_{t}\left(  aN\right)  =0$. Let $\mu\in\mathbb{C}$ be non-zero and take
$\varepsilon>0$ such that $\varepsilon\left\vert \mu\right\vert <2^{-1}.$ Then
there is $t>0$ such that
\begin{equation}
\mathbf{\mathrm{\eta}}\left(  \left(  aN\right)  ^{n}\right)  \leq
t\varepsilon^{n} \label{f71}%
\end{equation}
for every $n>0$. We have that
\begin{align*}
\mathbf{\mathrm{\eta}}\left(  \left(  \left\{  \mu a\right\}  \sqcup M\right)
^{n}\right)   &  =\sum_{i=0}^{n}\left\vert \mu\right\vert ^{i}\sum_{%
{\textstyle\sum_{k=0}^{i}}
m_{k}=n-i}\mathbf{\mathrm{\eta}}\left(  M^{m_{0}}aM^{m_{1}}\cdots aM^{m_{i}%
}\right) \\
&  \overset{(\ref{f70})}{\leq}\mathbf{\mathrm{\eta}}\left(  M^{n}\right)
+\sum_{i=1}^{n}\left\vert \mu\right\vert ^{i}\sum_{%
{\textstyle\sum_{k=0}^{i}}
m_{k}=n-i}\mathbf{\mathrm{\eta}}\left(  aM^{m_{1}}\cdots aM^{m_{i}}\right)
\end{align*}
for every $n>0$. As, for every $i>0$, the number of summands
$\mathbf{\mathrm{\eta}}\left(  aM^{m_{1}}\cdots aM^{m_{i}}\right)  $ is less
than or equal to $2^{n}$ and every such summand is less than or equal to
$\mathbf{\mathrm{\eta}}\left(  \left(  aN\right)  ^{i}\right)  $, then we
obtain that%
\begin{align*}
\mathbf{\mathrm{\eta}}\left(  \left(  \left\{  \mu a\right\}  \sqcup M\right)
^{n}\right)   &  \leq\mathbf{\mathrm{\eta}}\left(  M^{n}\right)  +2^{n}%
\sum_{i=1}^{n}\left\vert \mu\right\vert ^{i}\mathbf{\mathrm{\eta}}\left(
\left(  aN\right)  ^{i}\right)  \overset{(\ref{f71})}{\leq}%
\mathbf{\mathrm{\eta}}\left(  M^{n}\right)  +2^{n}t_{0}\sum_{i=1}%
^{n}\left\vert \mu\right\vert ^{i}\varepsilon^{i}\\
&  \leq\mathbf{\mathrm{\eta}}\left(  M^{n}\right)  +2^{n}t\leq2\max\left\{
\mathbf{\mathrm{\eta}}\left(  M^{n}\right)  ,2^{n}t\right\}
\end{align*}
Taking $n$-roots and passing to limits, we get $\rho_{t}\left(  \left\{  \mu
a\right\}  \sqcup M\right)  \leq\max\left\{  \rho_{t}\left(  M\right)
,2\right\}  $ for every $\mu\in\mathbb{C}$. As the function $\mu
\longmapsto\rho_{t}\left(  \left\{  \mu a\right\}  \sqcup M\right)  $ is
bounded and subharmonic, it is constant. Therefore
\[
\rho_{t}\left(  \left\{  a\right\}  \sqcup M\right)  =\rho_{t}\left(
M\right)  .
\]
As $M$ is arbitrary, $a\in\mathcal{R}_{t}\left(  A\right)  $.
Now assume that $A$ is not unital. Then, for each $M\in\ell_{1}\left(
A^{1}\right)  $, the family $K=MaM$ belongs to $\ell_{1}\left(  A\right)  $,
and $\rho_{t}((aM)^{2})=\rho_{t}(aK)=0$ by condition. By (\ref{f62}), we
obtain that $\rho_{t}(aM)=0$ for every $M\in\ell_{1}\left(  A^{1}\right)  $,
whence $a\in R_{t}\left(  A^{1}\right)  $ by the proof above. Now the result
follows from Lemma \ref{unital}.\textbf{ }
\end{proof}

\subsection{Tensor quasinilpotent algebras and ideals\label{stqi}}

Let $A$ be a normed algebra. A subset $G$ of $A$ is called a \textit{tensor
quasinilpotent set }if all summable families with elements in $G$ are tensor
quasinilpotent. A \textit{tensor quasinilpotent ideal} in $A$ is an ideal
which is a tensor quasinilpotent subset of $A$.

\begin{theorem}
\label{inv}$\rho_{t}\left(  M\sqcup N\right)  =\rho_{t}\left(  N\right)  $ for
every $M\in\ell_{1}\left(  \mathcal{R}_{t}\left(  A\right)  \right)  $ and
$N\in\ell_{1}\left(  A\right)  $.
\end{theorem}

\begin{proof}
Let $N\in\ell_{1}\left(  A\right)  $, $M=\left\{  a_{n}\right\}
_{n=1}^{\infty}\in\ell_{1}\left(  \mathcal{R}_{t}\left(  A\right)  \right)  $
and $M_{k}=\left\{  a_{n}\right\}  _{n=k}^{\infty}$ for every integer $k>0$.
For every $\varepsilon>0$ and $\mu\in\mathbb{C}$, there is $k>0$ such that
$\mathbf{\mathrm{\eta}}\left(  \mu M_{k}\right)  <\varepsilon$. Then
$\mu\longmapsto\mu M\sqcup N$ is an analytic function and%
\begin{align*}
\rho_{t}\left(  \mu M\sqcup N\right)   &  =\rho_{t}\left(  \mu M_{2}\sqcup
N\right)  =\cdots=\rho_{t}\left(  \mu M_{k}\sqcup N\right)  \leq
\mathbf{\mathrm{\eta}}\left(  \mu M_{k}\sqcup N\right) \\
&  =\mathbf{\mathrm{\eta}}\left(  \mu M_{k}\right)  +\mathbf{\mathrm{\eta}%
}\left(  N\right)  <\mathbf{\mathrm{\eta}}\left(  N\right)  +\varepsilon.
\end{align*}
So $\mu\longmapsto\rho_{t}\left(  \mu M\sqcup N\right)  $ is bounded and
therefore constant. Hence we obtain that $\rho_{t}\left(  M\sqcup N\right)
=\rho_{t}\left(  N\right)  $.
\end{proof}

As a consequence, we obtain the following

\begin{corollary}
\label{tq}$\mathcal{R}_{t}\left(  A\right)  $ is a tensor quasinilpotent ideal.
\end{corollary}

\begin{corollary}
\label{tq1}$\rho_{t}\left(  M+N\right)  =\rho_{t}\left(  N\right)  $ and
$\rho_{t}\left(  M\ast N\right)  =\rho_{t}\left(  MN\right)  =0$ for every
$M\in\ell_{1}\left(  \mathcal{R}_{t}\left(  A\right)  \right)  $ and $N\in
\ell_{1}\left(  A\right)  $.
\end{corollary}

\begin{proof}
Let $\mu\in\mathbb{C}$. Then $\rho_{t}\left(  \mu M+N\right)  \leq\rho
_{t}\left(  \mu M\sqcup N\right)  =\rho_{t}\left(  N\right)  $ by Proposition
\ref{change} and Theorem \ref{inv}. As $\mu\longmapsto\rho_{t}\left(  \mu
M+N\right)  $ is subharmonic and bounded on $\mathbb{C}$, it is constant,
whence $\rho_{t}\left(  M+N\right)  =\rho_{t}\left(  N\right)  $.
Since $MN\in\ell_{1}\left(  \mathcal{R}_{t}\left(  A\right)  \right)  $, then
$\rho_{t}\left(  MN\right)  =0$ by Corollary \ref{tq}. Then we obtain that
$\rho_{t}\left(  M\ast N\right)  =0$ by Proposition \ref{change}.
\end{proof}

\begin{corollary}
\label{largest}Let $A\mathcal{\ }$be a normed algebra and $a\in A$. The
following conditions are equivalent:

\begin{itemize}
\item[$\mathrm{(i)}$] $a\in\mathcal{R}_{t}\left(  A\right)  $.

\item[$\mathrm{(ii)}$] $\rho_{t}\left(  aM\right)  =0$ for every $M\in\ell
_{1}\left(  A\right)  $.
\end{itemize}
\end{corollary}

\begin{proof}
Indeed, $(\mathrm{i})\Longrightarrow\left(  \mathrm{ii}\right)  $ follows from
Corollary \ref{tq1}, and $(\mathrm{ii})\Longrightarrow\left(  \mathrm{i}%
\right)  $ was proved in Theorem \ref{multi}.
\end{proof}

We will prove now that $\mathcal{R}_{t}(A)$ is the \textit{largest} tensor
quasinilpotent ideal.

\begin{theorem}
\label{tq2}If $I$ is a tensor quasinilpotent (possible, one-sided) ideal of
$A$ then $I\subset\mathcal{R}_{t}\left(  A\right)  $.
\end{theorem}

\begin{proof}
Let $I$ be a right ideal of $A$, and let $a\in I$. Then $aM\in\ell_{1}\left(
I\right)  $ for every $M\in\ell_{1}\left(  A\right)  $. As $I$ is tensor
quasinilpotent then $\rho_{t}\left(  aM\right)  =0$ for every $M\in\ell
_{1}\left(  A\right)  $. By Theorem \ref{multi}, $a\in\mathcal{R}_{t}\left(
A\right)  $. So $I\subset\mathcal{R}_{t}\left(  A\right)  $.
If $I$ is a left ideal of $A$ and $a\in I$, then $\rho_{t}\left(  aM\right)
=\rho_{t}\left(  Ma\right)  $ by Proposition \ref{change}. So $\rho_{t}\left(
aM\right)  =0$ for every $M\in\ell_{1}\left(  A\right)  $. We have again that
$a\in\mathcal{R}_{t}\left(  A\right)  $.
\end{proof}

\begin{lemma}
\label{above} Let $M=\left\{  a_{n}\right\}  _{1}^{\infty}$ be a summable
family in $A$ and $g:A\longrightarrow B$ be a bounded homomorphism of normed
algebras. Then $g\left(  M\right)  :=\left\{  g\left(  a_{n}\right)  \right\}
_{1}^{\infty}$ is a summable family of $B$ and $\rho_{t}\left(  g\left(
M\right)  \right)  \leq\rho_{t}\left(  M\right)  $.
\end{lemma}

\begin{proof}
Indeed, it suffices to note that
\[
\mathbf{\mathrm{\eta}}\left(  g\left(  M\right)  ^{n}\right)
=\mathbf{\mathrm{\eta}}\left(  g\left(  M^{n}\right)  \right)  \leq\left\Vert
g\right\Vert \mathbf{\mathrm{\eta}}\left(  M^{n}\right)
\]
for every $n$.
\end{proof}

\begin{theorem}
\label{hom}Let $A$ and $B$ be a normed algebras, and let $g:A\longrightarrow
B$ be an open bounded epimorphism. Then $g\left(  \mathcal{R}_{t}\left(
A\right)  \right)  \subset\mathcal{R}_{t}\left(  B\right)  $.
\end{theorem}

\begin{proof}
Let $N=\left\{  b_{n}\right\}  _{1}^{\infty}$ be a summable family of $B$. As
$g$ is open, there is a summable family $M=\left\{  a_{n}\right\}  $ in $A$
such that $g\left(  M\right)  =N$. It follows from Lemma \ref{above} that
$\rho_{t}\left(  g\left(  K\right)  \right)  \leq\rho_{t}\left(  K\right)  $
for every $K\in\ell_{1}\left(  A\right)  $. So if $a\in\mathcal{R}_{t}\left(
A\right)  $ then, by Corollary \ref{largest}, we obtain that $\rho_{t}\left(
g\left(  a\right)  N\right)  \leq\rho_{t}\left(  aM\right)  =0$ for an
arbitrary $N\in\ell_{1}\left(  B\right)  $. Hence $g\left(  a\right)
\in\mathcal{R}_{t}\left(  B\right)  $ by Corollary \ref{largest}.
\end{proof}

Let $A$ be a normed algebra. Recall that if $I$ is a closed ideal of $A$, then
by $a/I$ (and also by $q_{I}(a)$) we denote the element $a+I$ of the algebra
$A/I$. By $M/I$ we denote the family $\left\{  a_{n}/I\right\}  $, for every
$M=\left\{  a_{n}\right\}  \in\ell_{1}\left(  A\right)  $.

\begin{theorem}
\label{tenquot} Let $M=\left\{  a_{n}\right\}  _{1}^{\infty}$ be a summable
family in a normed algebra $A$. Then $\rho_{t}\left(  M\right)  =\rho
_{t}(M/I)$ for each closed tensor quasinilpotent ideal $I$. In particular,
$\rho_{t}\left(  M\right)  =\rho_{t}\left(  M/\mathcal{R}_{t}\left(  A\right)
\right)  $.
\end{theorem}

\begin{proof}
As clearly $\rho_{t}\left(  M/\mathcal{R}_{t}\left(  A\right)  \right)
\leq\rho_{t}\left(  M\right)  $, it suffices to show the reverse inequality.
Let $\delta>0$, $n\in\mathbb{N}$ and $M^{n}=\left\{  b_{m}\right\}
_{1}^{\infty}$. Then for every $m$ there are $c_{m}\in A$ and $d_{m}\in I$
such that $b_{m}=c_{m}+d_{m}$ and
\[
\left\Vert c_{m}\right\Vert \leq\left\Vert b_{m}/I\right\Vert +2^{-m}\delta.
\]
Let $N=\left\{  c_{m}\right\}  _{1}^{\infty}$ and $S=\left\{  d_{m}\right\}
_{1}^{\infty}$. Since $\mathbf{\mathrm{\eta}}\left(  N\right)  \leq
\mathbf{\mathrm{\eta}}\left(  \ M^{n}/I\right)  +\delta$, then $N\in\ell
_{1}\left(  A\right)  .$ As $S=M^{n}-N$, we have that $S\in\ell_{1}\left(
I\right)  $ and that
\[
\rho_{t}\left(  N+S\right)  \leq\rho_{t}\left(  N\sqcup S\right)
\]
by Proposition \ref{change}. As $I\subset\mathcal{R}_{t}\left(  A\right)  $,
we have that
\[
\rho_{t}\left(  N+S\right)  \leq\rho_{t}\left(  N\sqcup S\right)  =\rho
_{t}\left(  N\right)  .
\]
by Theorem \ref{inv}. Therefore we obtain that
\[
\rho_{t}\left(  M\right)  ^{n}\overset{(\ref{f62})}{=}\rho_{t}\left(
M^{n}\right)  =\rho_{t}\left(  N+S\right)  \leq\rho_{t}\left(  N\right)
\leq\mathbf{\mathrm{\eta}}\left(  N\right)  \leq\mathbf{\mathrm{\eta}}\left(
M^{n}/I\right)  +\delta.
\]
Since $\delta$ is arbitrary, then $\rho_{t}\left(  M\right)  ^{n}%
\leq\mathbf{\mathrm{\eta}}\left(  M^{n}/I\right)  $ for every $n>0$. Taking
$n$-roots and passing to limits, we obtain that $\rho_{t}\left(  M\right)
\leq\rho_{t}\left(  M/I\right)  $.
\end{proof}

\begin{corollary}
\label{tenquo}Let $A$ be a normed algebra. Then $\mathcal{R}_{t}\left(
A/\mathcal{R}_{t}\left(  A\right)  \right)  =0$.
\end{corollary}

\begin{proof}
Let $a/\mathcal{R}_{t}\left(  A\right)  \in\mathcal{R}_{t}\left(
A/\mathcal{R}_{t}\left(  A\right)  \right)  $. Then it follows from Theorem
\ref{tenquot} that $\rho_{t}\left(  \left\{  a\right\}  \sqcup M\right)
=\rho_{t}\left(  M\right)  $ for every $M\in\ell_{1}\left(  A\right)  $. Hence
$a\in\mathcal{R}_{t}\left(  A\right)  $, and therefore $\mathcal{R}_{t}\left(
A/\mathcal{R}_{t}\left(  A\right)  \right)  =0$.
\end{proof}

\begin{theorem}
Let $A$ be a normed algebra. If $I$ is an ideal of $A$ then $\mathcal{R}%
_{t}\left(  I\right)  =\mathcal{R}_{t}\left(  A\right)  \cap I$.
\end{theorem}

\begin{proof}
It is clear that $\mathcal{R}_{t}\left(  A\right)  \cap I\subset
\mathcal{R}_{t}\left(  I\right)  $.
Let $a\in\mathcal{R}_{t}\left(  I\right)  $. For every $M\in\ell_{1}\left(
A\right)  $, we have that $MaM\in\ell_{1}\left(  I\right)  $ and then
\[
\rho_{t}\left(  aM\right)  ^{2}=\rho_{t}\left(  aMaM\right)  =0
\]
by (\ref{f62}) and Corollary \ref{largest}. Therefore $a\in\mathcal{R}%
_{t}\left(  A\right)  $ and $\mathcal{R}_{t}\left(  I\right)  \subset
\mathcal{R}_{t}\left(  A\right)  \cap I$.
\end{proof}

Note that this result contains Lemma \ref{unital} and implies that
\[
\mathcal{R}_{t}\left(  \mathcal{R}_{t}\left(  A\right)  \right)
=\mathcal{R}_{t}\left(  A\right)
\]
for every normed algebra $A$.

\subsection{Tensor radical algebras and ideals}

A normed algebra $A$ is called \textit{tensor radical} if the projective
tensor product $A\widehat{\otimes}_{\gamma}B$ is radical for every normed
algebra $B$. It is evident that $A$ is tensor radical if and only if its
completion $\widehat{A}$ is tensor radical. If $A$ is tensor radical then its
opposite algebra $A^{\mathrm{op}}$ is also tensor radical. An ideal of a
normed algebra is called \textit{tensor radical} if it is a tensor radical algebra.

The following result is an easy consequence of associativity and
distributivity of tensor product.

\begin{proposition}
\label{tensum}Let $A$ and $B$ be normed algebras.

\begin{itemize}
\item[$\mathrm{(i)}$] If $A$ is tensor radical then $A\widehat{\otimes}B$ is
tensor radical.

\item[$\mathrm{(ii)}$] If $A$ and $B$ are tensor radical then $A\oplus B$ is
tensor radical.
\end{itemize}
\end{proposition}

The study of deeper properties is based on the following theorem.

\begin{theorem}
\label{perrad} For a normed algebra $A$ the following conditions are equivalent.

\begin{itemize}
\item[$\mathrm{(i)}$] $A$ is tensor radical.

\item[$\mathrm{(ii)}$] $A$ is tensor quasinilpotent.
\end{itemize}
\end{theorem}

\begin{proof}
$(\mathrm{ii})\Rightarrow(\mathrm{i})$ follows from Corollary \ref{explains},
taking into account that every element of $A\widehat{\otimes}B$ can be
represented as $M_{\otimes}L$ for some $M\in\ell_{1}\left(  A\right)  $ and
$L\in\ell_{\infty}\left(  B\right)  $.
$(\mathrm{i})\Rightarrow(\mathrm{ii})$ follows from Theorem \ref{tsr}. Indeed,
by this theorem, there are a Banach algebra $B$ and $L\in\ell_{\infty}\left(
B\right)  $ such that $\rho\left(  M_{\otimes}L\right)  =\rho_{t}\left(
M\right)  $ for every $M\in\ell_{1}\left(  A\right)  $. If $A$ is tensor
radical then $A\widehat{\otimes}B$ is radical, whence $\rho\left(  M_{\otimes
}L\right)  =0$ for every $M\in\ell_{1}\left(  A\right)  $. Then $\rho
_{t}\left(  M\right)  =0$ for every $M\in\ell_{1}\left(  A\right)  $, i.e. $A$
is tensor quasinilpotent.
\end{proof}

\begin{corollary}
\label{den}Every subalgebra of a tensor radical normed algebra is tensor radical.
\end{corollary}

\begin{proof}
Follows from Theorem \ref{perrad}, since subalgebras of a tensor radical
algebra are obviously tensor quasinilpotent.
\end{proof}

\begin{corollary}
\label{dense}Let $A$ be a normed algebra. If there is a tensor quasinilpotent
dense subalgebra $B$ of $A$ then $A$ is tensor quasinilpotent.
\end{corollary}

\begin{proof}
Indeed, as $B$ is tensor radical by Theorem \ref{perrad}, the completion
$\widehat{B}$ is also tensor radical. As $\widehat{B}$ and $\widehat{A}$ are
identified, the algebra $\widehat{A}$ is tensor quasinilpotent by Theorem
\ref{perrad}. Therefore $A$ is tensor quasinilpotent.
\end{proof}

As a consequence of Corollary \ref{tq} and Theorem \ref{perrad}, for every
normed algebra $A$, $\mathcal{R}_{t}\left(  A\right)  $ is the largest tensor
radical ideal of $A$.

\begin{theorem}
\label{har}Let $A$ be a normed algebra and $a\in A$. Then $a\in\mathcal{R}%
_{t}\left(  A\right)  $ if and only if $a\otimes b\in\mathrm{Rad}\left(
A\widehat{\otimes}B\right)  $, for every normed algebra $B$ and $b\in B$.
\end{theorem}

\begin{proof}
Let $a\in\mathcal{R}_{t}\left(  A\right)  $. Then $a\otimes b\in
\mathcal{R}_{t}\left(  A\right)  \widehat{\otimes}^{{}\left(  \cdot\right)
}B$ for an arbitrary normed algebra $B$ and for every $b\in B$. In the same
time, $\mathcal{R}_{t}\left(  A\right)  \widehat{\otimes}B$ is a radical
Banach algebra. Being the image of a bounded homomorphism from $\mathcal{R}%
_{t}\left(  A\right)  \widehat{\otimes}B$, $\mathcal{R}_{t}\left(  A\right)
\widehat{\otimes}^{{}\left(  \cdot\right)  }B$ consists of quasinilpotent
elements of $A\widehat{\otimes}B$. But it is also an ideal of $A\widehat
{\otimes}B$. So $\mathcal{R}_{t}\left(  A\right)  \widehat{\otimes}^{{}\left(
\cdot\right)  }B\subset\mathrm{Rad}\left(  A\widehat{\otimes}B\right)  $ and
therefore $a\otimes b\in\mathrm{Rad}\left(  A\widehat{\otimes}B\right)  $.
Suppose that $a\otimes b\in\mathrm{Rad}\left(  A\widehat{\otimes}B\right)  $
for every normed algebra $B$ and $b\in B$. Take $B$ as in Theorem
\ref{tsr}(ii). Then $B$ has the identity element $1$ and there is a family
$L\in\ell_{\infty}\left(  B\right)  $ such that $\rho\left(  M_{\otimes
}L\right)  =\rho_{t}\left(  M\right)  $ for every $M\in\ell_{1}\left(
A\right)  $. Since $a\otimes1\in\mathrm{Rad}\left(  A\widehat{\otimes
}B\right)  $ then $\rho\left(  \left(  a\otimes1\right)  \left(  M_{\otimes
}L\right)  \right)  =0$. As $\left(  a\otimes1\right)  \left(  M_{\otimes
}L\right)  =aM_{\otimes}L$, we have that $\rho_{t}\left(  aM\right)  =0$ for
every $M\in\ell_{1}\left(  A\right)  $. By Theorem \ref{multi}, $a\in
\mathcal{R}_{t}\left(  A\right)  $.
\end{proof}

\begin{proposition}
\label{tr}Let $A$ and $B$ be normed algebras. Then
\[
\mathcal{R}_{t}\left(  A\right)  \widehat{\otimes}^{{}\left(  \cdot\right)
}B\subset\mathcal{R}_{t}\left(  A\widehat{\otimes}B\right)  \subset
\mathrm{Rad}\left(  A\widehat{\otimes}B\right)  .
\]

\end{proposition}

\begin{proof}
If $a\in\mathcal{R}_{t}\left(  A\right)  $ then it follows from Theorem
\ref{har} that
\[
a\otimes b\otimes c\in\mathrm{Rad}\left(  A\widehat{\otimes}B\widehat{\otimes
}C\right)  ,
\]
for every $b\in B$ and for every normed algebra $C$ and $c\in C$. By the same
theorem, $a\otimes b\in\mathcal{R}_{t}\left(  A\widehat{\otimes}B\right)  $
for every $b\in B$. So the closure of $\mathcal{R}_{t}\left(  A\right)
\otimes B$ in $A\widehat{\otimes}B$ is contained into $\mathcal{R}_{t}\left(
A\widehat{\otimes}B\right)  $ since $\mathcal{R}_{t}\left(  A\widehat{\otimes
}B\right)  $ is closed. But $\mathcal{R}_{t}\left(  A\right)  \widehat
{\otimes}^{{}\left(  \cdot\right)  }B$ is generated as a normed algebra by
elements of $\mathcal{R}_{t}\left(  A\right)  \otimes B$. Hence $\mathcal{R}%
_{t}\left(  A\right)  \widehat{\otimes}^{{}\left(  \cdot\right)  }%
B\subset\mathcal{R}_{t}\left(  A\widehat{\otimes}B\right)  $.
As $A\widehat{\otimes}B$ is a Banach algebra, $\mathcal{R}_{t}\left(
A\widehat{\otimes}B\right)  \subset\mathrm{Rad}\left(  A\widehat{\otimes
}B\right)  $ by Corollary \ref{rad}.
\end{proof}

\begin{proposition}
\label{tr7}Let $A$ be a normed algebra and $I$ be an ideal of $A$. If $I$ and
$A/\overline{I}$ are tensor radical then $A$ is tensor radical.
\end{proposition}

\begin{proof}
Let $M\in\ell_{1}\left(  A\right)  $. As $I$ is tensor radical then, as we
mentioned above, the closure $\overline{I}$ of $I$ in $A$ is also tensor
radical. Then $\rho_{t}\left(  M\right)  =\rho_{t}\left(  M/\overline
{I}\right)  $ by Theorem \ref{tenquot}, but $\rho_{t}\left(  M/\overline
{I}\right)  =0$ by Theorem \ref{perrad}. So $A$ is tensor radical.
\end{proof}

\begin{proposition}
\label{sirano} Let $A$ be a normed algebra and $I$ be a flexible ideal of $A$.
If $A$ is tensor radical then $\left(  I,\left\Vert \cdot\right\Vert
_{I}\right)  $ is tensor radical.
\end{proposition}

\begin{proof}
Let $M=\left(  a_{n}\right)  _{n=1}^{\infty}$ be a summable family in $I$.
Then $M$ is summable in $A$ and
\begin{align*}
\mathbf{\mathrm{\eta}}_{\left\Vert \cdot\right\Vert _{A}}\left(  M^{n}\right)
&  \leq\mathbf{\mathrm{\eta}}_{\left\Vert \cdot\right\Vert _{I}}\left(
M^{n}\right)  =\sum_{i_{1},\ldots,i_{n}}\left\Vert a_{i_{1}}\cdots a_{i_{n-1}%
}a_{i_{n}}\right\Vert _{I}\\
&  \leq\sum_{i_{1},\ldots,i_{n-1}}\left\Vert a_{i_{1}}\cdots a_{i_{n-1}%
}\right\Vert _{A}\sum_{i_{n}}\left\Vert a_{i_{n}}\right\Vert _{I}%
\leq\mathbf{\mathrm{\eta}}_{\left\Vert \cdot\right\Vert _{A}}\left(
M^{n-1}\right)  \mathbf{\mathrm{\eta}}_{\left\Vert \cdot\right\Vert _{I}%
}\left(  M\right)  .
\end{align*}
Hence $\rho_{t|I}\left(  M\right)  =\rho_{t}\left(  M\right)  =0$, where
$\rho_{t|I}$ is the tensor spectral radius in $\left(  I,\left\Vert
\cdot\right\Vert _{I}\right)  $.
\end{proof}

\begin{corollary}
\label{tr9}Let $I$ and $J$ be flexible ideals of a normed algebra $A$. If $I$
and $J$ are tensor radical then $I\cap J$ and $I+J$ are tensor radical with
respect to their flexible norms (see Proposition \emph{\ref{fl2}}).
\end{corollary}

\begin{proof}
By Corollary \ref{largest}, $I$ and $J$ are contained in $\mathcal{R}_{t}(A)$.
Hence the same is true for the ideals $I\cap J$ and $I+J$. It follows that
they are tensor radical with respect to the norm $\left\Vert \cdot\right\Vert
_{A}$. By Proposition \ref{sirano}, they are tensor radical with respect to
their flexible norms.
\end{proof}

The following result will be often applied in the further sections.

\begin{lemma}
\label{project} Let $A_{1},A_{2}$ be normed algebras, $A=A_{1}\widehat{\otimes
}A_{2}$, and let $J_{i}\subset I_{i}$ be ideals of $A_{i}$, for $i=1,2$.
Denote by $J$ the ideal of $A$ generated by $J_{1}\otimes A_{2}+A_{1}\otimes
J_{2}$, and let $I$ be the ideal of $A$ generated by $I_{1}\otimes A_{2}%
+A_{1}\otimes I_{2}$. If $\overline{I_{i}}/\overline{J_{i}}$ are tensor
radical then $\overline{I}\subset Q_{J}\left(  A\right)  $.
\end{lemma}

\begin{proof}
Let $\pi$ be a strictly irreducible representation of $A$ such that
$\pi\left(  J\right)  =0$. First we show that $\pi(I_{1}\otimes A_{2})=0$ and
$\pi(A_{1}\otimes I_{2})=0$.
Assume, to the contrary, that $\pi\left(  I_{1}\otimes A_{2}\right)  \neq0$.
Hence the restriction $\tau$ of $\pi$ to $\overline{I_{1}}\widehat{\otimes
}^{{}\left(  \cdot\right)  }A_{2}$ is strictly irreducible. As $\tau\left(
J_{1}\otimes A_{2}\right)  =0$ then $\tau\left(  \overline{J_{1}}\otimes
A_{2}\right)  =0$ because one may assume that $\tau$ is continuous.
Moreover, $\tau$ induces a strictly irreducible representation of the algebra
$C:=(\overline{I_{1}}/\overline{J_{1}})\widehat{\otimes}A_{2}$, because the
composition of natural maps
\[
(\overline{I_{1}}/\overline{J_{1}})\widehat{\otimes}A_{2}\longrightarrow
\left(  \overline{I_{1}}\widehat{\otimes}A_{2}\right)  /\overline{J_{1}\otimes
A_{2}}\longrightarrow\left(  \overline{I_{1}}\widehat{\otimes}^{{}\left(
\cdot\right)  }A_{2}\right)  /\overline{J_{1}\otimes A_{2}}^{\prime}%
\]
is a contractive epimorphism, where $\overline{J_{1}\otimes A_{2}}$ and
$\overline{J_{1}\otimes A_{2}}^{\prime}$ are the closures of $J_{1}\otimes
A_{2}$ in $\overline{I_{1}}\widehat{\otimes}A_{2}$ and $\overline{I_{1}%
}\widehat{\otimes}^{{}\left(  \cdot\right)  }A_{2}$, respectively, and clearly
$\overline{J_{1}\otimes A_{2}}\subset\overline{J_{1}\otimes A_{2}}^{\prime}$.
As there exists a non-zero strictly irreducible representation of $C$ then $C$
is not radical, but $C$ is radical in virtue of tensor radicality of
$\overline{I_{1}}/\overline{J_{1}}$, a contradiction. Hence $\pi\left(
I_{1}\otimes A_{2}\right)  =0$ and, similarly, $\pi(A_{1}\otimes I_{2})=0$. As
$I$ lies in the intersection of of all primitive ideals of $A$ containing $J$,
so does $\overline{I}$. By Proposition \ref{predv1}(i), $\overline{I}\subset
Q_{J}\left(  A\right)  $.
\end{proof}

\subsection{Algebras commutative modulo the tensor radical}

We say that a normed algebra $A$ is \textit{commutative modulo the tensor
radical} if the algebra $A/\mathcal{R}_{t}(A)$ is commutative. An equivalent
condition is $[a,b]\in\mathcal{R}_{t}(A)$ for all $a,b\in A$.

\begin{theorem}
\label{cmtr} If normed algebras $A_{1}$ and $A_{2}$ are commutative modulo the
tensor radical then the same is true for $A:=A_{1}\widehat{\otimes}A_{2}$.
\end{theorem}

\begin{proof}
By Proposition \ref{tr}, $\mathcal{R}_{t}(A_{1})\widehat{\otimes}A_{2}%
+A_{1}\widehat{\otimes}^{{}\left(  \cdot\right)  }\mathcal{R}_{t}%
(A_{2})\subset\mathcal{R}_{t}(A)$. Then
\[
\lbrack a_{1}{\otimes}b_{1},a_{1}{\otimes}b_{2}]=[a_{1},a_{2}]{\otimes}%
b_{1}b_{2}+a_{2}a_{1}{\otimes}[b_{1},b_{2}]\in\mathcal{R}_{t}(A)
\]
for all $a_{1},a_{2}\in A_{1}$ and $b_{1},b_{2}\in A_{2}$. Hence $[c_{1}%
,c_{2}]\in\mathcal{R}_{t}(A)$ for all $c_{1},c_{2}\in A$.
\end{proof}

\begin{theorem}
Let $A_{1}$ be a normed algebra, and let $A_{2}$ be a Banach algebra. If
$A_{1}$ is commutative modulo the tensor radical and $A_{2}$ is radical then
$A_{1}\widehat{\otimes}A_{2}$ is radical.
\end{theorem}

\begin{proof}
Let $A=A_{1}\widehat{\otimes}A_{2}$ and $I=\mathcal{R}_{t}(A_{1}%
)\widehat{\otimes}^{{}\left(  \cdot\right)  }A_{2}$. The Banach ideal $I$ of
$A$ is radical, being isometric to the quotient of the radical algebra
$\mathcal{R}_{t}(A_{1})\widehat{\otimes}A_{2}$. On the other hand, the
quotient $A/\overline{I}$ is topologically isomorphic to $(A/\mathcal{R}%
_{t}(A_{1}))\widehat{\otimes}A_{2}$ by Proposition \ref{quo}. Since
$A/\mathcal{R}_{t}(A_{1})$ is commutative and $A_{2}$ is radical, the algebra
$(A/\mathcal{R}_{t}(A_{1}))\widehat{\otimes}A_{2}$ is radical by \cite[Theorem
4.4.2]{Aup}. Thus $A/\overline{I}$ is radical, and $A$ is radical by Lemma
\ref{quotBan}.
\end{proof}

\subsection{Relation with joint spectral radius}

In 1960 Rota and Strang \cite{RS} defined a notion of spectral radius for
bounded subsets of a Banach algebra. This definition holds for normed algebras
(and we already introduced it for countable bounded subsets in Section
\ref{sss1}). Namely, if $K$ is a bounded subset of a normed algebra $A$ then
its\textit{ joint spectral radius} $\rho(K)$ is defined by
\[
\rho(K)=\inf_{n\in\mathbb{N}}\Vert K^{n}\Vert^{1/n},
\]
where the norm of a set is defined as the supremum of the norms of its
elements, and the products of sets are defined by $KN=\{ab:a\in K,b\in N\}$.
Since $\Vert K^{n+k}\Vert\leq\Vert K^{n}\Vert\Vert K^{k}\Vert$ for every
$n,k>0$, one has that
\begin{equation}
\rho(K)=\underset{n\rightarrow\infty}{\lim}\Vert K^{n}\Vert^{1/n}.
\label{nach2}%
\end{equation}
Taking $n=mk$ in (\ref{nach2}) for $m=1,2,\ldots$, we observe that
\[
\rho(K^{k})=\rho(K)^{k}%
\]
for every bounded $K\subset A$ and integer $k>0$.

It was proved in \cite[Theorem 3.5]{ST2000} that the joint spectral radius is
a subharmonic function. This means that if $\lambda\rightarrow K(\lambda)$ is
an analytic map in a natural sense from a domain $\mathcal{D}\subset
\mathbb{C}$ into the set of bounded subsets of $A$ then the function
$\lambda\rightarrow\rho(K(\lambda))$ is subharmonic.

Let $K$ be a subset of $A$, and let $\mathrm{F}_{\infty}\left(  K\right)  $ be
the set of all bounded families $N=\left\{  a_{n}\right\}  _{1}^{\infty}$ with
$a_{n}\in K$ for every $n>0$. Clearly $\mathrm{F}_{\infty}\left(  K\right)  $
is a metric space with respect to the metric induced by the norm of
$\ell_{\infty}\left(  A\right)  $. In particular, we have that $\mathrm{F}%
_{\infty}\left(  A\right)  =\ell_{\infty}\left(  A\right)  $. If $K$ is
bounded, it follows from \cite[Proposition 2.2]{ST2000} that there is a family
$L=\left\{  b_{n}\right\}  _{1}^{\infty}\in\mathrm{F}_{\infty}\left(
K\right)  $ such that $\rho(K)=\rho(L)$. So
\[
\rho(K)=\max_{N\in\mathrm{F}_{\infty}\left(  K\right)  }\rho(N)
\]
and this allows to obtain some results on joint spectral radius of bounded
subsets considering bounded families.

The following property is important for our applications: \textit{If}
$\rho(K)=0$\textit{ then the linear span of }$K$ \textit{consists of
quasinilpotent elements.} This result from \cite{Sh84} can be easily proved by
the direct evaluation of the norms of powers for $\sum_{i=1}^{n}\lambda
_{i}a_{i}$, where $a_{i}\in K.$ The following result is similar.

\begin{lemma}
\label{comb} Let $K$ be a bounded subset of a normed algebra $A$. If
$\rho(K)=0$ then $\rho(\sum_{n=1}^{\infty}\lambda_{n}a_{n})=0$ for each
sequence $a_{n}\in K$ and each summable sequence $\lambda_{n}$ of complex
numbers, where $\sum_{n=1}^{\infty}\lambda_{n}a_{n}$ are elements of the
completion $\widehat{A}$ as usual.
\end{lemma}

\begin{proof}
Let $M=\left\{  \lambda_{n}\right\}  _{1}^{\infty}$ and $L=\left\{
a_{n}\right\}  _{1}^{\infty}$. As $\mathbb{C}\widehat{\mathbb{\otimes}}A$ is
identified with $\widehat{A}$ under the identification $\lambda\widehat
{\mathbb{\otimes}}x$ with $\lambda x$, we obtain that
\[
\rho(\sum_{n=1}^{\infty}\lambda_{n}a_{n})\leq\rho_{t}\left(  M\right)
\rho(L)\leq\rho_{t}\left(  M\right)  \rho(K)=0
\]
by Theorem \ref{tsr}(i).
\end{proof}

We say that a normed algebra $A$ is \textit{compactly quasinilpotent}
\cite{TR1} if $\rho(K)=0$ for each precompact subset $K$ of $A$. The following
statement improves \cite[Theorem 4.29]{TR1} which was proved for Banach algebras.

\begin{theorem}
\label{comptens} Every compactly quasinilpotent normed algebra is tensor radical.
\end{theorem}

\begin{proof}
Let $A$ be a compactly quasinilpotent normed algebra, $B$ a Banach algebra,
and let $x=\sum_{n=1}^{\infty}a_{n}\otimes b_{n}\in A\widehat{\otimes}B$. One
can assume that $\left\{  a_{n}\right\}  _{1}^{\infty}$ consists of elements
of $A$ (see Section \ref{ss1}), the sequence $\alpha_{n}=\Vert a_{n}\Vert$ is
summable, while $\Vert b_{n}\Vert\leq1$ for every $n$. It is obvious that
there exists a sequence $\varepsilon_{n}\rightarrow0$ such that $\lambda
_{n}:=\alpha_{n}/\varepsilon_{n}$ is summable. Set $c_{n}=\lambda_{n}%
^{-1}a_{n}$ for $n>0$. Then $\Vert c_{n}\Vert\rightarrow0$ as $n\rightarrow
\infty$, so the set $N:=\{c_{n}:n=1,2,\ldots\}$ is precompact and $\rho(N)=0$
by our assumption. For $K=\{c_{n}\otimes b_{n}:n=1,2,\ldots\}$, it is easy to
check that
\[
\Vert K^{k}\Vert\leq\Vert N^{k}\Vert,
\]
whence $\rho(K)\leq\rho(N)=0$. Applying Lemma \ref{comb}, we obtain that
\[
\rho(x)=\rho(\sum_{n=1}^{\infty}\lambda_{n}c_{n}\otimes b_{n})=0.
\]
Therefore $A\widehat{\otimes}B$ consists of quasinilpotent elements.
\end{proof}

It was proved in \cite{TR1} that each normed algebra has the largest compactly
quasinilpotent ideal $\mathcal{R}_{c}(A)$.

\begin{corollary}
\label{inclC-T} Let $A$ be a normed algebra. Then $\mathcal{R}_{c}%
(A)\subset\mathcal{R}_{t}(A)\subset A\cap\mathrm{Rad}(\widehat{A})$.
\end{corollary}

\begin{proof}
The first inclusion follows by Theorem \ref{comptens}.
By Corollary \ref{dense}, the closure $\overline{\mathcal{R}_{t}(A)}$ in the
completion $\widehat{A}$ of $A$ is tensor quasinilpotent. Then $\overline
{\mathcal{R}_{t}(A)}$ consists of quasinilpotent elements of $\widehat{A}$. As
$\overline{\mathcal{R}_{t}(A)}$ is an ideal of $\widehat{A}$ consisting of
quasinilpotents, $\overline{\mathcal{R}_{t}(A)}\subset\mathrm{Rad}(\widehat
{A})$. Therefore we obtain that $\mathcal{R}_{t}(A)\subset A\cap
\mathrm{Rad}(\widehat{A})$.
\end{proof}

\subsection{Compactness conditions}

It is still an open problem if any radical Banach algebra is tensor radical.
We will show here that the answer is positive if $A$ has some compactness properties.

In the well known paper of Vala \cite{Vala} it was shown that

\begin{itemize}
\item[$\mathrm{(i)}$] For compact operators $a,b$ on a Banach space $X$ the
multiplication operator $x\longmapsto axb$ is compact on $\mathcal{B}(X)$.

\item[$\mathrm{(ii)}$] If the operator $x\longmapsto axa$ is compact then the
operator $a$ is compact.
\end{itemize}

This gave a possibility to introduce a notion of a compact element of a normed
algebra: an element $a$ of $A$ is \textit{compact} if
\[
W_{a}:=L_{a}R_{a}%
\]
is a compact operator, where $L_{a}$ and $R_{a}$ are defined by $L_{a}x=ax$
and $R_{a}x=xa$ for every $x\in A$. Similarly, one says that $a$ is a
\textit{finite rank element} of $A$ if $W_{a}$ has finite rank. Basing on
these definitions there were introduced various Banach-algebraic analogues of
the class of algebras of compact operators. The most popular one is the class
of compact algebras: $A$ is \textit{compact} if all its elements are compact.
Slightly more narrow but much more convenient is the class of bicompact
algebras: $A$ is \textit{bicompact} if $L_{a}R_{b}$ are compact for all
$a,b\in A$. Furthermore, $A$ is called an \textit{approximable algebra} if the
set of finite rank elements is dense in $A$. These classes are closed under
passing to ideals and quotients, but in general not stable under extensions.

To overcome this obstacle and considerably extend the class of algebras in
consideration, let us say that a normed algebra $A$ is \textit{hypocompact}
(respectively, \textit{hypofinite}) if each non-zero quotient of $A$ has a
non-zero compact (respectively, finite rank) element. It is not difficult to
check that the class of all hypocompact algebras is closed under extensions,
as well as under passing to ideals and quotients. It is not known if it is
closed under passing to subalgebras. One can realize a hypocompact algebra as
a result of a transfinite sequence of extensions of bicompact algebras, but we
will need a close result, see Proposition \ref{trans} below.

Since a quotient of a quotient of $A$ is isomorphic to a quotient of $A$, the
following result is an immediate consequence of the definition of hypocompact algebras.

\begin{corollary}
\label{hypquot} A quotient of a hypocompact normed algebra (by a closed ideal)
is hypocompact.
\end{corollary}

A similar result is valid for hypofinite normed algebras.

We also need the following result.

\begin{proposition}
\label{trans}Let $A$ be a normed algebra. Then $A$ is hypocompact
(respectively, hypofinite) if and only if there is an increasing transfinite
chain $\left(  J_{\alpha}\right)  _{\alpha\leq\beta}$ of closed ideals of $A$
such that $J_{0}=0$, $J_{\beta}=A$, $J_{\alpha}=\overline{\cup_{\alpha
^{\prime}<\alpha}J_{\alpha^{\prime}}}$ for every limit ordinal $\alpha
\leq\beta$ and $J_{\alpha+1}/J_{\alpha}$ is a non-zero ideal of $A/J_{\alpha}$
having a dense set of compact (respectively, finite rank) elements of
$A/J_{\alpha}$ for every ordinal $\alpha$ between $0$ and $\beta$.
\end{proposition}

\begin{proof}
$\Rightarrow$ Let us use the transfinite induction. Let $J_{0}=0$. If we
constructed $J_{\alpha}$ and $J_{\alpha}\neq A$ then take a non-zero compact
(finite rank) element $b$ in $A/J_{\alpha}$ and denote by $K$ the closed ideal
of $A/J_{\alpha}$ generated by $b$. Let us define $J_{\alpha+1}$ as the
preimage of $K$ in $A$:
\[
J_{\alpha+1}=\left\{  c\in A:c/J_{\alpha}\in K\right\}  .
\]
Then $J_{\alpha+1}/J_{\alpha}$ is topologically isomorphic to $K$. It remains
to note that the chain is stabilized at some step $\beta$ because $A$ has a
definite cardinality. So $J_{\beta}=A$.
$\Leftarrow$ Let $I$ be a closed ideal of $A$ and $I\neq A$. Then there is the
first ordinal $\alpha^{\prime}<\beta$ such that $I$ is not contained into
$J_{\alpha^{\prime}}$. Then $J_{\alpha}\subset I$ for every $\alpha
<\alpha^{\prime}$, whence $\overline{\cup_{\alpha<\alpha^{\prime}}J_{\alpha}%
}\subset I$. So $\alpha^{\prime}$ has a precessor $\alpha^{\prime\prime}$:
$\alpha^{\prime}=\alpha^{\prime\prime}+1$. Take $a\in J_{\alpha^{\prime}%
}\backslash I$, and let $G=\left\{  x\in A:\left\Vert x-a\right\Vert
<\mathrm{dist}\left(  a,I\right)  \right\}  $. Then $G/J_{\alpha^{\prime
\prime}}:=\left\{  x/J_{\alpha^{\prime\prime}}\in A/J_{\alpha^{\prime\prime}%
}:x\in G\right\}  $ is an open neighbourhood of $a/J_{\alpha^{\prime\prime}%
}\in J_{\alpha^{\prime}}/J_{\alpha^{\prime\prime}}$ and therefore has a
compact (finite rank) element $b/J_{\alpha^{\prime\prime}}$ of $A/J_{\alpha
^{\prime\prime}}$. It is clear that $b/I$ is a non-zero compact (finite rank)
element of $A/I$. So $A$ is hypocompact (respectively, hypofinite).
\end{proof}

\begin{corollary}
\label{tran}Let $B$ be a normed algebra, and let $A$ be a hypocompact
(respectively, hypofinite) dense subalgebra of $B$. Then $B$ is hypocompact
(respectively, hypofinite).
\end{corollary}

\begin{proof}
let $\left(  J_{\alpha}\right)  _{\alpha\leq\beta}$ be a transfinite chain of
ideals of $A$ described in Proposition \ref{trans}. Let $I_{\alpha}%
=\overline{J_{\alpha}}$, the closure of $J_{\alpha}$ in $B$, for every ordinal
$\alpha\leq\beta$. Then the chain $\left(  I_{\alpha}\right)  _{\alpha
\leq\beta}$ satisfies the conditions of Proposition \ref{trans}. So $B$ is
hypocompact (respectively, hypofinite) by Proposition \ref{trans}.
\end{proof}

The following result of \cite{STFaa} will be very useful in Section \ref{s5}.

\begin{theorem}
\label{scat} If a Banach algebra is hypocompact then spectra of its elements
are (finite or) countable.
\end{theorem}

Our main aim here is to show that for hypocompact Banach algebras the ideal
$\mathcal{R}_{t}(A)$ coincides with the Jacobson radical $\mathrm{Rad}(A)$.

For a bounded subset $M$ of a Banach algebra, set
\[
r(M)=\underset{n\rightarrow\infty}{\lim\sup}\left(  \sup\left\{  \rho(a):a\in
M^{n}\right\}  \right)  ^{1/n}.
\]
Clearly $r(M)\leq\rho(M)$.

This spectral characteristic, introduced (for sets of matrices) in 1992 by M.
A. Berger and Y. Wang \cite{BW}, turned out to be very useful in operator
theory. It was proved in \cite{BW} that $r(M)=\rho(M)$ for any bounded set $M$
of matrices. In \cite{ST2000} the authors showed that the same is true if $M$
is a precompact set of compact operators on a Banach space. In the further
works \cite{STFaa, ST2001-S, GBWF} there were obtained several extensions of
this result. Here we need the following consequence of \cite[Theorem
4.11]{GBWF} (where only Banach algebras were considered).

\begin{corollary}
\label{hypoBW} Let $A$ be a hypocompact normed algebra. Then $r(M)=\rho(M)$
for each precompact subset $M$ of $A$.
\end{corollary}

\begin{proof}
It is clear that $r(M)$ and $\rho(M)$ don't change if pass to the completion
$\widehat{A}$. Moreover, $\widehat{A}$ is hypocompact by Corollary \ref{tran}.
Now, appying \cite[Theorem 4.11]{GBWF}, we get the result.
\end{proof}

\begin{theorem}
\label{hyphypura} If $A$ is a hypocompact normed algebra then
\[
\mathcal{R}_{c}(A)=\mathcal{R}_{t}(A)=A\cap\mathrm{Rad}(\widehat{A}).
\]

\end{theorem}

\begin{proof}
Taking into account Corollary \ref{inclC-T}, we have to prove only the
inclusion $A\cap\mathrm{Rad}(\widehat{A})\subset\mathcal{R}_{c}(A)$. Since
$A\cap\mathrm{Rad}(\widehat{A})$ is an ideal of a hypocompact algebra $A$, it
is hypocompact (see Corollary \ref{ihf}). If $M$ is a precompact subset of
$A\cap\mathrm{Rad}(\widehat{A})$ then $r(M)=0$ because all elements in
$\cup_{n=1}^{\infty}M^{n}$ are quasinilpotent. By Corollary \ref{hypoBW},
$\rho(M)=0$. So $A\cap\mathrm{Rad}(\widehat{A})$ is a compactly quasinilpotent
ideal of $A$ which implies that $A\cap\mathrm{Rad}(\widehat{A})\subset
\mathcal{R}_{c}(A)$.
\end{proof}

\begin{corollary}
\label{tensrad} Each radical hypocompact normed algebra is tensor radical.
\end{corollary}

The following result \cite[Corollary 6.2]{A68} supplies us with an important
class of examples of bicompact radical Banach algebras. Our proof of
radicality of $\mathcal{K}(X)/\mathcal{A}(X)$ differs from the proof in
\cite{A68}.

\begin{lemma}
\label{quot} Let $\mathcal{K}(X)$ be the algebra of all compact operators on a
Banach space $X$, $\mathcal{A}(X)$ be the closure in $\mathcal{K}(X)$ of the
ideal $\mathcal{F}\left(  X\right)  $ of finite rank operators. Then
$\mathcal{K}(X)/\mathcal{A}(X)$ is a radical bicompact Banach algebra.
\end{lemma}

\begin{proof}
The fact that $\mathcal{K}(X)$ is bicompact follows from the mentioned result
of Vala \cite{Vala}. Since the quotient of a bicompact algebra is obviously
bicompact, $\mathcal{K}(X)/\mathcal{A}(X)$ is bicompact.
To see that $\mathcal{K}(X)/\mathcal{A}(X)$ is radical, note that all
projections in $\mathcal{K}(X)$ are of finite rank and therefore belong to
$\mathcal{A}(X)$. It follows that for each $a\in\mathcal{K}(X)$ and each
spectral (= Riesz) projection $p$ of $a$ corresponding to a subset
$\alpha\subset\mathbb{C}$ non-containing $0$, we have that
\[
q((1-p)a)=q(a),
\]
where $q$ is the quotient map from $\mathcal{K}(X)$ to $\mathcal{K}%
(X)/\mathcal{A}(X)$. Since
\[
\rho(q((1-p)a))\leq\rho((1-p)a)
\]
and $\rho((1-p)a)$ can be made arbitrary small by an appropriate choice of
$p$, we conclude that $\rho(q(a))=0$. So $\mathcal{K}(X)/\mathcal{A}(X)$
consists of quasinilpotent elements.
\end{proof}

Recall that \textit{Riesz operators} are defined \cite{Ruston} as operators
that are quasinilpotent modulo the compact operators.

\begin{corollary}
For every Riesz operator $a\in\mathcal{B}\left(  X\right)  $ and for every
$\varepsilon>0$, there is $m\in\mathbb{N}$ such that $\mathrm{dist}%
_{\left\Vert \cdot\right\Vert _{\mathcal{B}}}\left(  a^{n},\mathcal{F}\left(
X\right)  \right)  <\varepsilon^{n}$ for each $n>m$.
\end{corollary}

\begin{proof}
Indeed, it follows from Lemma \ref{quot} that every Riesz operator is
quasinilpotent modulo the approximable operators.
\end{proof}

Our aim now is to show that the class of hypocompact algebras is stable under
tensor products. Let $\mathrm{ball}(A)$ denote the closed unit ball of $A$.
Recall that $W_{a}=L_{a}R_{a}$ for every $a\in A$, and if $M\subset A$ and
$N\subset B$ are not subspaces then $M{\otimes N}$ means only the set
$\left\{  a{\otimes b:a\in M,b\in N}\right\}  $, not its linear span.
Moreover, if $I$ is a closed ideal of $A$, then $M/I$ means the set $\left\{
a/I:a\in M\right\}  \subset A/I$.

\begin{lemma}
\label{tenzcomp} Let $A,B$ be unital Banach algebras, $J$ a closed ideal in
$A{\hat{\otimes}}B$, $I_{1}$ and $I_{2}$ closed ideals in $A$ and $B$
respectively. Let elements $a\in A$, $b\in B$ satisfy the conditions
$a{\otimes}I_{2}\subset J$ and $I_{1}{\otimes}b\subset J$. If $a/I_{1}$ and
$b/I_{2}$ are compact elements of $A/I_{1}$ and $B/I_{2}$ respectively then
$(a{\otimes}b)/J$ is a compact element of $(A{\widehat{\otimes}}B)/J$.
\end{lemma}

\begin{proof}
Note that $aAa{\otimes}I_{2}=(a{\otimes}I_{2})(Aa{\otimes}1)\subset J$ and,
similarly, $I_{1}{\otimes}bBb\subset J$. Choose $\varepsilon>0$. Let
$x_{1},...,x_{n}\in\mathrm{ball}(A)$ be such that $\{W_{a}x_{i}/I_{1}:1\leq
i\leq n\}$ is an {$\varepsilon$}-net in $W_{a}(\mathrm{ball}(A/I_{1}))$. This
means that for any $x\in\mathrm{ball}(A)$ there are $i\leq n$ and $x^{\prime
}\in I_{1}$ with
\[
\Vert W_{a}x-W_{a}x_{i}-x^{\prime}\Vert<\varepsilon.
\]
In the same way one finds $y_{1},...,y_{m}\in\mathrm{ball}(B)$ such that for
each $y\in\mathrm{ball}(B)$ there are $k\leq m$ and $y^{\prime}\in I_{2}$
with
\[
\Vert W_{b}y-W_{b}y_{k}-y^{\prime}\Vert<\varepsilon.
\]
Let us check that the set $\{(W_{a}x_{i}{\otimes}W_{b}y_{k})/J:i\leq
n,\,\,k\leq m\}$ is a $\delta$-net for the set $(W_{a}(\mathrm{ball}%
(A)){\otimes}W_{b}(\mathrm{ball}(B)))/J$, where $\delta=(\Vert a\Vert
^{2}+\Vert b\Vert^{2})\varepsilon$. Indeed, for $x\in\mathrm{ball}(A)$,
$y\in\mathrm{ball}(B)$, choose $i,k$ as above. Then we obtain that
\begin{align*}
z  &  :=W_{a}x{\otimes}W_{b}y-W_{a}x_{i}{\otimes}W_{b}y_{k}\\
&  =(W_{a}x-W_{a}x_{i}){\otimes}W_{b}y+W_{a}x_{i}{\otimes}(W_{b}y-W_{b}%
y_{k})\\
&  =x^{\prime}{\otimes}W_{b}y+W_{a}x_{i}{\otimes}y^{\prime}+{u\otimes}%
W_{b}y+W_{a}x_{i}{\otimes v}%
\end{align*}
where $\Vert u\Vert<\varepsilon$, $\Vert v\Vert<\varepsilon$. Since the first
two summands belong to $J$, we conclude that the norm of $z/J$ in
$A\widehat{{\otimes}}B/J$ is less than $\delta$.
We proved that $(W_{a}(\mathrm{ball}(A)){\otimes}W_{b}(\mathrm{ball}(B)))/J$
is precompact in $\left(  A\widehat{{\otimes}}B\right)  /J$. Since
$\mathrm{ball}(A\widehat{{\otimes}}B)$ is the closed convex hull of the set
$\mathrm{ball}(A){\otimes}\mathrm{ball}(B)$, then the set $(W_{a}{\otimes
}W_{b})(\mathrm{ball}(A\widehat{{\otimes}}B))/J$ is precompact, whence
$W_{(a{\otimes}b)/J}$ is compact.
\end{proof}

Let us say, for brevity, that an element $a$ is \textit{compact modulo a
closed ideal} $J$ if $a/J$ is a compact element of $A/J$.

\begin{theorem}
\label{tenz} If normed algebras $A$ and $B$ are hypocompact then
$A\widehat{{\otimes}}B$ is hypocompact.
\end{theorem}

\begin{proof}
By Corollary \ref{tran}, one can assume that $A$ and $B$ are complete.
Suppose first that $A$ and $B$ are unital. Let $J$ be a proper closed ideal of
$A\widehat{{\otimes}}B$. We have to prove that $\left(  A\widehat{{\otimes}%
}B\right)  /J$ has non-zero compact elements.
Set $I_{1}=\{x\in A:x{\otimes}B\subset J\}$. By our assumption, $I_{1}\neq A$
(indeed, otherwise $J=A\widehat{{\otimes}}B$), so there exists an element
$a\in A\backslash I_{1}$ which is compact modulo $I_{1}$. Set $I_{2}=\{y\in
B:a{\otimes}y\in J\}$. Since $a\notin I_{1}$ then $I_{2}\neq B$. Let $b\in
B\backslash I_{2}$ be an element of $B$ compact modulo $I_{2}$.
By the definition of $I_{2}$, $a{\otimes}I_{2}\subset J$. Furthermore,
$I_{1}{\otimes}b\subset J$ because $I_{1}{\otimes}B\subset J$. Hence the
assumptions of Lemma \ref{tenzcomp} are satisfied, therefore $a{\otimes}b$ is
an element of $A\widehat{{\otimes}}B$ compact modulo $J$. It is clear that
$a{\otimes}b\notin J$ by the choice of $b$.
In general it suffices to note that the unitalization $A^{1}$ of a hypocompact
algebra $A$ is hypocompact and $A\widehat{{\otimes}}B$, being a closed ideal
of $A^{1}\widehat{{\otimes}}B^{1}$, is hypocompact by Corollary \ref{ihf}.
\end{proof}


\subsection{Topological radicals\label{toprad}}

A map $P$ which associates with every normed algebra $A$ a closed ideal
$P\left(  A\right)  $ of $A$ is called a \textit{topological radical } if
$P\left(  P\left(  A\right)  \right)  =P\left(  A\right)  $, $P\left(
A/P\left(  A\right)  \right)  =0$, $P\left(  I\right)  $ is an ideal of $A$
and $P\left(  I\right)  \subset P\left(  A\right)  $ for every ideal $I$ of
$A$, and $f\left(  P\left(  A\right)  \right)  \subset P\left(  B\right)  $
for every morphism $f:A\longrightarrow B$. The meaning of the later
requirement depends on the specification of morphisms in the different
categories whose objects are normed algebras. In the applications below,
\textit{open bounded epimorphisms} are included in the class of morphisms of
any such category. The study of topological radicals was initiated by
\cite{Dix}.

A topological radical $P$ is called \textit{hereditary} if $P\left(  I\right)
=I\cap P\left(  A\right)  $ for any ideal $I$ of each normed algebra $A$. A
normed algebra $A$ is called $P$\textit{-radical} if $A=P\left(  A\right)  $
and $P$\textit{-semisimple} if $P\left(  A\right)  =0$.

Note \cite{Dix} that the Jacobson radical $\mathrm{rad}:A\longmapsto
\mathrm{rad}(A)$ is not a topological radical on the class of all normed
algebras, but its restriction $\mathrm{Rad}$ to the class of all Banach
algebras is a hereditary topological radical. This radical admits different
extensions to the class of all normed algebras which are topological radicals
(see for instance \cite[Section 2.6]{TR1}). One of them is the regular
extension $\mathrm{Rad}^{r}:A\longmapsto A\cap\mathrm{Rad}\left(
\widehat{A}\right)  $ (see \cite[Section 2.8]{TR1}) where $\widehat{A}$ is the
completion of $A$. We already met this radical in Corollary \ref{inclC-T} and
Theorem \ref{hyphypura}.

Let $\mathcal{R}_{t}$ denote the map $A\longmapsto$ $\mathcal{R}_{t}\left(
A\right)  $ for every normed algebra $A$.

\begin{theorem}
\label{rt}$\mathcal{R}_{t}$ is a hereditary topological radical in the
category of normed algebras morphisms of which are open bounded epimorphisms.
\end{theorem}

\begin{proof}
It follows from the results of Sections \ref{sid} and \ref{stqi}.
\end{proof}

The same was proved for the map $\mathcal{R}_{c}:A\longmapsto$ $\mathcal{R}%
_{c}\left(  A\right)  $ (see \cite[Theorem 4.25]{TR1}).

Moreover, it was proved in \cite[Theorem 3.14]{GBWF} (see also the short
communication \cite{STFaa}) that each normed algebra $A$ has a largest
hypocompact ideal $\mathcal{R}_{hc}(A)$, and the map $\mathcal{R}%
_{hc}:A\longmapsto\mathcal{R}_{hc}(A)$ \textit{is a hereditary topological
radical on the class of normed algebras with open bounded epimorphisms as
morphisms}. It should be noted that for simplicity the results of
\cite[Section 3.2]{GBWF} were formulated for Banach algebras, but the proofs
did not use the completeness.

\begin{theorem}
\label{hf}For every normed algebra $A$, there exists a largest hypofinite
ideal $\mathcal{R}_{hf}(A)$, and the map $\mathcal{R}_{hf}:A\longmapsto
\mathcal{R}_{hf}(A)$ \textit{is} a hereditary topological radical on the class
of normed algebras morphisms of which are bounded homomorphisms with dense image.
\end{theorem}

\begin{proof}
Similar to the proof in \cite[Section 3.2]{GBWF}. One have to replace compact
algebras by approximable ones and take into account that a bounded
homomorphism with dense image maps finite rank elements into finite rank elements.
\end{proof}

We note that $\mathcal{R}_{hf}$-radical algebras are just hypofinite algebras
as well as $\mathcal{R}_{hc}$-radical algebras are hypocompact algebras. Now
we note the following useful consequence of the heredity of radicals
$\mathcal{R}_{hc}$ and $\mathcal{R}_{hf}$.

\begin{corollary}
\label{ihf}Every ideal of a hypocompact (respectively, hypofinite) normed
algebra is a hypocompact (respectively, hypofinite) normed algebra.
\end{corollary}

\section{ Multiplication operators on Banach bimodules\label{s4}}

\subsection{ Banach bimodules}

\subsubsection{Elementary operators\label{seo}}

Let $A,B$ be Banach algebras and $U$ a bimodule over $A,B$ (shortly
$(A,B)$-bimodule). Then in an obvious way $U$ can be considered as an
$(A^{1},B^{1})$-bimodule. We say that $U$ is \textit{ a normed bimodule} if it
is a normed space with a norm $\left\Vert \cdot\right\Vert _{U}$ and
\[
\left\Vert aub\right\Vert _{U}\leq\left\Vert a\right\Vert _{A}\left\Vert
u\right\Vert _{U}\left\Vert b\right\Vert _{B}%
\]
for every $a\in A^{1}$, $b\in B^{1}$ and $u\in U$.

Let $L_{a}$ and $R_{b}$ be operators on $U$ defined by $L_{a}x=ax$ and
$R_{b}x=xb$ for every $x\in U$. By $\mathcal{E\!\ell}_{A,B}(U)$ we denote the
algebra generated by all operators $L_{a},R_{b}$. Its elements are called
\textit{elementary operators on }$U$\textit{ with coefficients in }$A,B$. If
$A,B$ are unital then $\mathcal{E\!\ell}_{A,B}(U)$ coincides with the algebra
$\mathcal{E}_{A,B}(U)$ generated by all $L_{a}R_{b}$. In the general case
$\mathcal{E}_{A,B}(U)$ is an ideal of $\mathcal{E\!\ell}_{A,B}(U)$ which is an
ideal in $\mathcal{E}_{A^{1},B^{1}}(U)$, and the latter can be regarded as a
unitalization of $\mathcal{E\!\ell}_{A,B}(U)$. Note also that
\[
\mathcal{E\!\ell}_{A,B}(U)=\mathcal{E}_{A^{1},B}(U)+\mathcal{E}_{A,B^{1}}(U).
\]

Clearly operators in $\mathcal{E}_{A,B}(U)$ and $\mathcal{E\!\ell}_{A,B}(U)$
can be written in the form $T=\sum_{i=1}^{n}L_{a_{i}}R_{b_{i}}$ and,
respectively, $T=L_{a}+R_{b}+\sum_{i=1}^{n}L_{a_{i}}R_{b_{i}}$, where
$a,a_{i}\in A$, $b,b_{i}\in B_{i}$.

One may consider $U$ as a left $\left(  B^{\mathrm{op}}\right)  $-module,
where $B^{\mathrm{op}}$ is the algebra opposite to $B$. Then there is a
natural homomorphism $\psi=\psi_{U}$ from $A\otimes_{\gamma}B^{\mathrm{op}}$
into the algebra $\mathcal{B}(U)$ of all continuous operators on $U$ given by
\[
\psi:z=\sum_{i=0}^{n}a_{i}\otimes b_{i}\longmapsto\sum_{i=0}^{n}L_{a_{i}%
}R_{b_{i}}%
\]
for every $a_{i}\in A$ and $b_{i}\in B$. Then
\begin{equation}
\left\Vert \psi\left(  z\right)  u\right\Vert _{U}\leq\gamma\left(  z\right)
\left\Vert u\right\Vert _{U} \label{ihq}%
\end{equation}
for every $u\in U$, whence $\ker\psi$ is closed. Since the image of $\psi$
coincides with $\mathcal{E}_{A,B}(U)$, one may consider $\mathcal{E}%
_{A,B}\left(  U\right)  $ as a quotient of $A\otimes_{\gamma}B^{\mathrm{op}}$.
This induces the quotient norm $\left\Vert \cdot\right\Vert _{\mathcal{E}%
_{A,B}}$, or simply $\left\Vert \cdot\right\Vert _{\mathcal{E}}$, on
$\mathcal{E}_{A,B}\left(  U\right)  $ by
\begin{equation}
\left\Vert T\right\Vert _{\mathcal{E}}=\inf\left\{  \sum\left\Vert
a_{i}\right\Vert _{A}\left\Vert b_{i}\right\Vert _{B}:\sum L_{a_{i}}R_{b_{i}%
}=T\right\}  \label{eon}%
\end{equation}
for every $T\in\mathcal{E}_{A,B}\left(  U\right)  $. So $\mathcal{E}%
_{A,B}\left(  U\right)  $ is a normed algebra with respect to $\left\Vert
\cdot\right\Vert _{\mathcal{E}}$.

\begin{proposition}
If $U$ is a normed bimodule then $\left\Vert \cdot\right\Vert _{\mathcal{B}%
\left(  U\right)  }\leq\left\Vert \cdot\right\Vert _{\mathcal{E}}$ on
$\mathcal{E}_{A,B}\left(  U\right)  $, so $\mathcal{E}_{A,B}\left(  U\right)
$ is a normed subalgebra of $\mathcal{B}\left(  U\right)  $ with respect to
$\left\Vert \cdot\right\Vert _{\mathcal{E}_{A,B}}$.
\end{proposition}

\begin{proof}
Indeed, if $T\in\mathcal{E}_{A,B}\left(  U\right)  $ then $\left\Vert
T\right\Vert _{\mathcal{B}\left(  U\right)  }\leq\left\Vert T\right\Vert
_{\mathcal{E}}$ by (\ref{ihq}) and (\ref{eon}).
\end{proof}

In a similar way one can consider $\mathcal{E}_{A^{1},B^{1}}(U)$ (and
therefore $\mathcal{E\!\ell}_{A,B}(U)$) as a normed subalgebra of
$\mathcal{B}\left(  U\right)  $. In the case of unital coefficient algebras
$A,B$ we don't distinguish $\mathcal{E\!\ell}_{A,B}(U)$ from $\mathcal{E}%
_{A,B}(U)$.

\subsubsection{Multiplication operators}

A normed bimodule is called \textit{Banach} if it is a Banach space. It is
clear that the completion $\widehat{U}$ of a normed bimodule $U$ is a Banach
bimodule and that one can identify $\mathcal{E}_{A,B}\left(  U\right)  $ and
$\mathcal{E}_{A,B}\left(  \widehat{U}\right)  $.

Let $U$ be a Banach $\left(  A,B\right)  $-bimodule. Let $\widehat{\mathcal{E}%
}_{A,B}\left(  U\right)  $ be the completion of $\mathcal{E}_{A,B}\left(
U\right)  $ in $\left\Vert \cdot\right\Vert _{\mathcal{E}}$. It is clear that
$\widehat{\mathcal{E}}_{A,B}\left(  U\right)  $ is an algebra of continuous
operators on $U$. The operators in $\widehat{\mathcal{E}}_{A,B}\left(
U\right)  $ are called \textit{multiplication operators} on $U$.

Again, if $U$ is a Banach bimodule then $\widehat{\mathcal{E}}_{A,B}\left(
U\right)  \subset\mathcal{B}\left(  U\right)  $ as usual, and
$\widehat{\mathcal{E}}_{A,B}\left(  U\right)  $ is a Banach subalgebra of
$\mathcal{B}\left(  U\right)  $.

\begin{proposition}
If $I$ and $J$ are flexible ideals of $A$ and $B$ respectively, then
$\mathcal{E}_{I,J}\left(  U\right)  $ is a flexible ideal of $\mathcal{E}%
_{A,B}\left(  U\right)  $, and $\widehat{\mathcal{E}}_{I,J}\left(  U\right)  $
with the norm $\left\Vert \cdot\right\Vert _{\mathcal{E}_{I,J}}$ is a Banach
ideal of $\widehat{\mathcal{E}}_{A,B}\left(  U\right)  $.
\end{proposition}

In what follows we often denote the coefficient algebras by $A_{1},A_{2}$
instead of $A,B$.

\begin{theorem}
\label{mult1-} Let $U$ be a Banach bimodule over normed algebras $A_{1},A_{2}$.

\begin{itemize}
\item[$\mathrm{(i)}$] If $A_{1}$ and $A_{2}$ are hypocompact then the algebra
$\widehat{\mathcal{E}}_{A_{1},A_{2}}(U)$ is hypocompact.

\item[$\mathrm{(ii)}$] If at least one of the algebras $A_{i}$ is tensor
radical then $\widehat{\mathcal{E}}_{A_{1},A_{2}}(U)$ is tensor radical.

\item[$\mathrm{(iii)}$] If both $A_{i}$ are commutative modulo the tensor
radical then $\widehat{\mathcal{E}}_{A_{1},A_{2}}(U)$ is commutative modulo
the tensor radical.
\end{itemize}
\end{theorem}

\begin{proof}
(i) Indeed, it is easy to see that $\widehat{\mathcal{E}}_{A_{1},A_{2}}(U)$ is
isometric to a quotient of $A_{1}\widehat{\otimes}A_{2}^{\mathrm{op}}$. So it
is hypocompact by Theorem \ref{tenz} and Corollary \ref{hypquot}.
(ii) Since $\widehat{\mathcal{E}}_{A_{1},A_{2}}(U)$ is isometric to the
quotient of $A_{1}\widehat{\otimes}A_{2}^{\mathrm{op}}$ by the kernel of the
natural map from $A_{1}\widehat{\otimes}A_{2}^{\mathrm{op}}$ into
$\mathcal{B}(U)$, the statement follows from the fact that a quotient of a
tensor radical normed algebra is tensor radical (which follows easily from
Theorem \ref{hom}).
(iii) Arguing as in (ii), we have only to prove that if a normed algebra $B$
is commutative modulo the tensor radical then so is the quotient of $B$ by a
closed ideal $J$.
Let $q_{J}:B\longrightarrow B/J$ be the standard epimorphism. Then
\[
q_{J}(\mathcal{R}_{t}(B))\subset\mathcal{R}_{t}(B/J)
\]
by Theorem \ref{hom}. Assuming that $B/\mathcal{R}_{t}(B)$ is commutative, we
obtain that
\[
\lbrack a/J,b/J]=q_{J}([a,b])\in\mathcal{R}_{t}(B/J)
\]
for all $a,b\in B$. This means that $B/J$ is commutative modulo the tensor radical.
\end{proof}


\begin{corollary}
\label{sp3}Let $U$ be a Banach bimodule over unital normed algebras
$A_{1},A_{2}$. If $A_{1}$ and $A_{2}$ are hypocompact then $\sigma
_{\mathcal{B}\left(  U\right)  }\left(  T\right)  $ is (finite or) countable
and
\[
\sigma_{\mathcal{B}\left(  U\right)  }\left(  T\right)  =\sigma
_{\widehat{\mathcal{E}}_{A_{1},A_{2}}(U)}\left(  T\right)
\]
for every $T\in\widehat{\mathcal{E}}_{A_{1},A_{2}}(U)$.
\end{corollary}

\begin{proof}
Indeed, $\widehat{\mathcal{E}}_{A_{1},A_{2}}(U)$ is a unital Banach subalgebra
of $\mathcal{B}\left(  U\right)  $, and $\sigma_{\widehat{\mathcal{E}}%
_{A_{1},A_{2}}(U)}\left(  T\right)  $ is countable by Theorem \ref{scat}. So
$\sigma_{\mathcal{B}\left(  U\right)  }\left(  T\right)  =\sigma
_{\widehat{\mathcal{E}}_{A_{1},A_{2}}(U)}\left(  T\right)  $ by Proposition
\ref{sp}.
\end{proof}

Now we consider the sum $\widehat{\mathcal{E}}_{A_{1},I_{2}}%
(U)+\widehat{\mathcal{E}}_{I_{1},A_{2}}(U)$ of Banach subalgebras of
$\mathcal{B}(U)$ for closed ideals $I_{i}$ of $A_{i}$.

\begin{corollary}
\label{ideals1}Let $I_{i}$ be a closed ideal of normed algebra $A_{i}$ for
$i=1,2$. If $I_{1}$ and $I_{2}$ are tensor radical then the algebra
$\widehat{\mathcal{E}}_{A_{1},I_{2}}(U)+\widehat{\mathcal{E}}_{I_{1},A_{2}%
}(U)$ is tensor radical.
\end{corollary}

\begin{proof}
Since $\widehat{\mathcal{E}}_{I_{1},A_{2}}(U)$ and $\widehat{\mathcal{E}%
}_{A_{1},I_{2}}(U)$ are ideals of $\widehat{\mathcal{E}}_{A_{1},A_{2}}(U)$ and
are tensor radical by Theorem \ref{mult1-}, they are contained in
$\mathcal{R}_{t}(\widehat{\mathcal{E}}_{A_{1},A_{2}}(U))$. Since
$\widehat{\mathcal{E}}_{A_{1},I_{2}}(U)+\widehat{\mathcal{E}}_{I_{1},A_{2}%
}(U)$ is a flexible ideal of $\widehat{\mathcal{E}}_{A_{1},A_{2}}(U)$, it is
also tensor radical by Proposition \ref{tr9}.
\end{proof}

Let us consider a more general situation.

\begin{theorem}
\label{quot2} Let $J_{i}\subset I_{i}$ be ideals of $A_{i}$, $i=1,2$. Suppose
that the algebras $\overline{I_{i}}/\overline{J_{i}}$ are tensor radical.
Setting, for brevity, $\widehat{\mathcal{E}}_{I}=\widehat{\mathcal{E}}%
_{A_{1},I_{2}}(U)+\widehat{\mathcal{E}}_{I_{1},A_{2}}(U)$ and $\mathcal{E}%
_{J}={\mathcal{E}}_{A_{1},J_{2}}(U)+{\mathcal{E}}_{J_{1},A_{2}}(U)$ we have
that $\overline{\mathcal{E}_{J}}$ is an ideal of $\widehat{\mathcal{E}}_{I}$
and the algebra $\widehat{\mathcal{E}}_{I}/\overline{\mathcal{E}_{J}}$ is
tensor radical.
\end{theorem}

\begin{proof}
One can clearly assume that the ideals $I_{i}$ and $J_{i}$ are closed in
$A_{i}$ for $i=1,2$.
Consider the Banach ideal $K_{1}=$ $\left(  I_{1}/J_{1}\right)  \widehat
{\otimes}^{{}\left(  \cdot\right)  }\left(  A_{2}/J_{2}\right)  $ in the
Banach algebra $B=\left(  A_{1}/J_{1}\right)  \widehat{\otimes}_{\gamma
}\left(  A_{2}/J_{2}\right)  $. As $K_{1}$ is topologically isomorphic to a
quotient of $\left(  I_{1}/J_{1}\right)  \widehat{\otimes}\left(  A_{2}%
/J_{2}\right)  $ then it is tensor radical. Similarly, the Banach ideal
$K_{2}=\left(  A_{1}/J_{1}\right)  \widehat{\otimes}^{{}\left(  \cdot\right)
}\left(  I_{2}/J_{2}\right)  $ in $B$ is tensor radical. Then their flexible
sum $K_{1}+K_{2}$ in $B$ is tensor radical by Proposition \ref{tr9}.
Hence the quotient $(I_{1}\widehat{\otimes}^{{}\left(  \cdot\right)  }%
A_{2}+A_{1}\widehat{\otimes}^{{}\left(  \cdot\right)  }I_{2})/\overline
{J_{1}\otimes A_{2}+A_{1}\otimes J_{2}}$ is tensor radical because it is
topologically isomorphic to $K_{1}+K_{2}$ by Proposition \ref{quo}.
Consider now the natural epimorphism $\psi:I_{1}\widehat{\otimes}^{{}\left(
\cdot\right)  }A_{2}+A_{1}\widehat{\otimes}^{{}\left(  \cdot\right)  }%
I_{2}\longrightarrow\widehat{\mathcal{E}}_{I}/\overline{\mathcal{E}_{J}}$. It
is clear that $J_{1}\otimes A_{2}+A_{1}\otimes J_{2}\subset\ker\psi$. So there
is a bounded homomorphism from $(I_{1}\widehat{\otimes}^{{}\left(
\cdot\right)  }A_{2}+A_{1}\widehat{\otimes}^{{}\left(  \cdot\right)  }%
I_{2})/\overline{J_{1}\otimes A_{2}+A_{1}\otimes J_{2}}$ onto $\widehat
{\mathcal{E}}_{I}/\overline{\mathcal{E}_{J}}$. This epimorphism is open by the
Banach Theorem. Therefore $\widehat{\mathcal{E}}_{I}/\overline{\mathcal{E}%
_{J}}$ is tensor radical by Theorem \ref{hom}.
\end{proof}

The proved result is important for applications in Section \ref{s6}. Note that
Corollary \ref{ideals1} can be obtained from Theorem \ref{quot2} if one takes
$J_{1}=J_{2}=0$.

In the following result we preserve notation of Theorem \ref{quot2}.

\begin{corollary}
\label{quasrel} Let $I$ and $J$ be as in Theorem $\ref{quot2}$. Then
$\widehat{\mathcal{E}}_{I}\subset Q_{\mathcal{E}_{J}}(\widehat{\mathcal{E}%
}_{A_{1},A_{2}}(U))$. Thus if $T\in\widehat{\mathcal{E}}_{I}$ then, for every
$\varepsilon>0$, there are $m\in\mathbb{N}$ and elementary operators $S_{n}%
\in\mathcal{E}_{J}$ on $U$ such that $\left\Vert T^{n}-S_{n}\right\Vert
_{\mathcal{E}}<\varepsilon^{n}$ for every $n>m$.
\end{corollary}

\begin{proof}
Follows from Proposition \ref{predv1}.
\end{proof}

\subsection{ Operator bimodules}

Let us consider the case that $A_{1}$, $A_{2}$ are the algebras $\mathcal{B}%
(Y)$, $\mathcal{B}(X)$ of all bounded operators on Banach spaces $X,Y$. Let
$U$ be a normed subspace of $\mathcal{B}(X,Y)$ of all bounded operators from
$X$ to $Y$ with the natural bounded $(\mathcal{B}(Y),\mathcal{B}(X))$-bimodule
structure; we refer to it as a \textit{normed subbimodule} of $\mathcal{B}%
(X,Y)$ or, simply, a \textit{normed operator bimodule}. The latter means that
$U$ is supplied with its own norm $\left\Vert \cdot\right\Vert _{U}%
\geq\left\Vert \cdot\right\Vert _{\mathcal{B}}=\left\Vert \cdot\right\Vert $
and
\[
\left\Vert axb\right\Vert _{U}\leq\left\Vert a\right\Vert \left\Vert
x\right\Vert _{U}\left\Vert b\right\Vert
\]
for all $a\in\mathcal{B}(Y)$, $b\in\mathcal{B}(X)$, $x\in U$. It is easy to
see that if $U$ is non-zero then $U$ contains all finite rank operators. We
also may assume that
\[
\left\Vert x\right\Vert _{U}=\Vert x\Vert
\]
for every rank one operator $x\in U$.

When $U$ is complete in $\left\Vert \cdot\right\Vert _{U}$, one says that $U$
is a \textit{Banach operator bimodule}. In this case, for brevity, we also
denote $\widehat{\mathcal{E}}_{\mathcal{B}(Y),\mathcal{B}(X)}(U)$ by
$\widehat{\mathcal{B}}_{\ast}\left(  U\right)  $ and call its elements
$\left(  \mathcal{B}\right)  $-\textit{multiplication operators }on $U$. It is
clear that $\widehat{\mathcal{B}}_{\ast}\left(  U\right)  $ is a Banach
subalgebra of $\mathcal{B}\left(  U\right)  $ with respect to the norm
$\left\Vert \cdot\right\Vert _{\widehat{\mathcal{B}}_{\ast}}=\left\Vert
\cdot\right\Vert _{\widehat{\mathcal{E}}_{\mathcal{B}(Y),\mathcal{B}(X)}}$.

Operator bimodules are closely related to operator ideals. If $\mathcal{U}$ is
a Banach operator ideal in the sense of \cite{P78}, e. g. the ideal
$\mathcal{K}$ of compact operators or the ideal $\mathcal{N}$ of nuclear
operators then each \textit{component} $U=\mathcal{U}(X,Y)$ of $\mathcal{U}$
is a Banach operator bimodule. It can be proved that all Banach operator
bimodules can be obtained in this way.

\subsubsection{Semicompact multiplication and $\left(  \mathcal{K}\right)
$-multiplication operators}

The algebras $A_{i}$ are semisimple so they have no radical ideals. But they
can have pairs of ideals $J_{i}\subset I_{i}$ with radical quotients
$I_{i}/\overline{J_{i}}$. Indeed, Lemma \ref{quot} shows that this is the case
if we take the ideals $\mathcal{K}(X)$, $\mathcal{K}(Y)$ for $I_{i}$, and the
ideals $\mathcal{F}(X)$, $\mathcal{F}(Y)$ for $J_{i}$. The possibility to use
Theorem \ref{quot2} is important for further applications. Before formulate
the corresponding corollaries it will be convenient to introduce special
terminology. Namely we set
\begin{align*}
\widehat{\mathcal{K}}_{\frac{1}{2}}(U)  &  =\widehat{\mathcal{E}}%
_{\mathcal{B}(Y),\mathcal{K}(X)}(U)+\widehat{\mathcal{E}}_{\mathcal{K}%
(Y),\mathcal{B}(X)}(U),\\
{\mathcal{F}}_{\frac{1}{2}}(U)  &  ={\mathcal{E}}_{\mathcal{B}(Y),\mathcal{F}%
(X)}(U)+{\mathcal{E}}_{\mathcal{F}(Y),\mathcal{B}(X)}(U).
\end{align*}
Note that $\widehat{\mathcal{K}}_{\frac{1}{2}}(U)$ is taken here as the sum of
\ Banach ideals $\widehat{\mathcal{E}}_{\mathcal{B}(Y),\mathcal{K}(X)}(U)$ and
$\widehat{\mathcal{E}}_{\mathcal{K}(Y),\mathcal{B}(X)}(U)$ in
$\widehat{\mathcal{B}}_{\ast}\left(  U\right)  $ with respective flexible norm
(see Proposition \ref{fl2}). Thus $\widehat{\mathcal{K}}_{\frac{1}{2}}(U)$
consists of multiplication operators $T$ on $U$ of the form $\sum
_{i=1}^{\infty}L_{a_{i}}R_{b_{i}}$ such that $\sum_{i}\Vert a_{i}\Vert\Vert
b_{i}\Vert<\infty$ and at least one of the operators $a_{i}$ or $b_{i}$ is
compact for every $i$. The norm $\left\Vert T\right\Vert
_{\widehat{\mathcal{K}}_{\frac{1}{2}}}$ is equal to $\inf\sum_{i}\Vert
a_{i}\Vert\Vert b_{i}\Vert$ for all such representations of $T$. We call
operators in $\widehat{\mathcal{K}}_{\frac{1}{2}}(U)$ \textit{semicompact
multiplication operators}.

Similarly, operators in ${\mathcal{F}}_{\frac{1}{2}}(U)$ are called
\textit{semifinite elementary operators}. They are just the elementary
operators $\sum_{i=1}^{n}L_{a_{i}}R_{b_{i}}$ where $a_{i}$ or $b_{i}$ is a
finite rank operator for each $i$.

As a concrete application of Theorem \ref{quot2}, we obtain the following

\begin{corollary}
\label{semi} Let $U\subset\mathcal{B}\left(  X,Y\right)  $ be a Banach
operator bimodule. Then the algebra $\widehat{\mathcal{K}}_{\frac{1}{2}%
}\left(  U\right)  /\overline{\mathcal{F}_{\frac{1}{2}}\left(  U\right)  }$ is
tensor radical.
\end{corollary}

As a consequence we have the following

\begin{corollary}
\label{QQ}Let $U\subset\mathcal{B}\left(  X,Y\right)  $ be a Banach operator
bimodule. Then
\[
\widehat{\mathcal{K}}_{\frac{1}{2}}(U)\subset Q_{\mathcal{F}_{\frac{1}{2}%
}\left(  U\right)  }(\widehat{\mathcal{B}}_{\ast}\left(  U\right)  ).
\]

\end{corollary}

Since the norm in $\widehat{\mathcal{B}}_{\ast}\left(  U\right)  $ majorizes
the operator norm in $\mathcal{B}(U)$, we also obtain the following result.

\begin{corollary}
\label{opnorm} Let $T$ be a semicompact multiplication operator on a Banach
operator bimodule $U$. Then for each $\varepsilon>0$ there are $m\in
\mathbb{N}$ and semifinite elementary operators $S_{n}$ on $U$ such that
$\left\Vert T^{n}-S_{n}\right\Vert _{\mathcal{B}(U)}<\varepsilon^{n}$ for
every $n>m$.
\end{corollary}

The other useful ideal in $\widehat{\mathcal{B}}_{\ast}\left(  U\right)  $,
namely the ideal $\widehat{\mathcal{K}}_{\ast}\left(  U\right)  $ of $\left(
\mathcal{K}\right)  $\textit{-multiplication operators, }is defined by
\[
\widehat{\mathcal{K}}_{\ast}\left(  U\right)  =\widehat{{\mathcal{E}}%
}_{\mathcal{K}(Y),\mathcal{K}(X)}(U).
\]
In many cases, for instance when the norm of $U$ coincides with the operator
one, $\widehat{\mathcal{K}}_{\ast}\left(  U\right)  $ consists of compact
operators on $U$. It is clear that $\widehat{\mathcal{K}}_{\ast}\left(
U\right)  $ is a Banach ideal of $\widehat{\mathcal{K}}_{\frac{1}{2}}(U)$ and
$\widehat{\mathcal{B}}_{\ast}\left(  U\right)  $ with respect to the norm
$\left\Vert \cdot\right\Vert _{\widehat{\mathcal{K}}_{\ast}}=\left\Vert
\cdot\right\Vert _{\widehat{{\mathcal{E}}}_{\mathcal{K}(Y),\mathcal{K}(X)}}$.

\begin{proposition}
$\widehat{\mathcal{K}}_{\ast}\left(  U\right)  $ is a bicompact Banach algebra
for every Banach operator bimodule $U$.
\end{proposition}

\begin{proof}
Indeed, $\widehat{\mathcal{K}}_{\ast}\left(  U\right)  $ is topologically
isomorphic to a quotient of the projective tensor product of bicompact
algebras $\mathcal{K}\left(  Y\right)  $ and $\mathcal{K}\left(  X\right)
^{\mathrm{op}}$ which is bicompact itself by Lemma \ref{tenzcomp}.
\end{proof}

\begin{corollary}
Let $U$ be a Banach operator bimodule. Then $\sigma_{\mathcal{B}\left(
U\right)  }\left(  T\right)  $ is (finite or) countable and $\sigma
_{\mathcal{B}\left(  U\right)  }\left(  T\right)  \cup\left\{  0\right\}
=\sigma_{\widehat{\mathcal{K}}_{\ast}\left(  U\right)  }\left(  T\right)
\cup\left\{  0\right\}  =\sigma_{\widehat{\mathcal{B}}_{\ast}\left(  U\right)
}\left(  T\right)  \cup\left\{  0\right\}  $ for every $T\in
\widehat{\mathcal{K}}_{\ast}\left(  U\right)  $.
\end{corollary}

\begin{proof}
Indeed, $\widehat{\mathcal{K}}_{\ast}\left(  U\right)  $ is a Banach
subalgebra of $\mathcal{B}\left(  U\right)  $ and is a Banach ideal of
$\widehat{\mathcal{B}}_{\ast}\left(  U\right)  $. Since $\sigma_{\widehat
{\mathcal{K}}_{\ast}\left(  U\right)  }\left(  T\right)  $ is countable by
\cite[Theorem 4.4]{A68}, then it is easy to see that $\sigma_{\mathcal{B}%
\left(  U\right)  }\left(  T\right)  \cup\left\{  0\right\}  =\sigma
_{\widehat{\mathcal{K}}_{\ast}\left(  U\right)  }\left(  T\right)
\cup\left\{  0\right\}  $ by Proposition \ref{sp}. Further, we have that
$\sigma_{\widehat{\mathcal{K}}_{\ast}\left(  U\right)  }\left(  T\right)
\cup\left\{  0\right\}  =\sigma_{\widehat{\mathcal{B}}_{\ast}\left(  U\right)
}\left(  T\right)  \cup\left\{  0\right\}  $ by Remark \ref{sp2}.
\end{proof}

\subsubsection{$\left(  \mathcal{N}\right)  $-multiplication operators and
trace}

The norms and spectra of elementary operators were studied in many works. Here
we would like to mention a simple formula for their traces.

Let $\mathcal{N}$ be the operator ideal of nuclear operators. Recall that
every operator $a\in\mathcal{N}\left(  X,Y\right)  $ has a representation
$\sum f_{i}\otimes x_{i}$ with $\sum\left\Vert f_{i}\right\Vert \left\Vert
x_{i}\right\Vert <\infty$ for $f_{i}\in X^{\ast}$ and $x_{i}\in Y$. The
nuclear norm $\left\Vert \cdot\right\Vert _{\mathcal{N}}$ is given by
\[
\left\Vert a\right\Vert _{\mathcal{N}}=\inf\left\{  \sum\left\Vert
f_{i}\right\Vert \left\Vert x_{i}\right\Vert :\sum f_{i}\otimes x_{i}%
=a\right\}  .
\]
So, as is well known, the projective tensor product $X^{\ast}\widehat{\otimes
}_{\gamma}Y$ is identified with $\mathcal{N}\left(  X,Y\right)  $. If $X=Y$,
then the trace of $a$ is defined by
\[
\mathrm{trace}\left(  a\right)  =\sum f_{i}\left(  x_{i}\right)  .
\]

Let ${U}\subset\mathcal{B}(X,Y)$ be a Banach operator bimodule. One can define
the ideal $\mathcal{N}_{\ast}\left(  U\right)  $ of $\left(  \mathcal{N}%
\right)  $-\textit{multiplication operators} by
\[
\widehat{\mathcal{N}}_{\ast}\left(  U\right)  =\widehat{{\mathcal{E}}%
}_{\mathcal{N}(Y),\mathcal{N}(X)}(U)
\]
with respect to the norm $\left\Vert \cdot\right\Vert _{\widehat{\mathcal{N}%
}_{\ast}}=\left\Vert \cdot\right\Vert _{\widehat{{\mathcal{E}}}_{\mathcal{N}%
(Y),\mathcal{N}(X)}}$.

\begin{proposition}
Every $\left(  \mathcal{N}\right)  $-multiplication operator $T$ on ${U}$ is
nuclear. If $T=\sum L_{a_{i}}R_{b_{i}}$ with $\sum\left\Vert a_{i}\right\Vert
_{\mathcal{N}}\left\Vert b_{i}\right\Vert _{\mathcal{N}}<\infty$ for
$a_{i},b_{i}\in\mathcal{N}$, then
\[
\mathrm{trace}\left(  T\right)  =\sum_{i}\mathrm{trace}\left(  a_{i}\right)
\mathrm{trace}\left(  b_{i}\right)  .
\]

\end{proposition}

\begin{proof}
Assume first that $a=f\otimes x$ and $b=g\otimes y$ are rank one operators.
For every $u\in U$, we obtain that
\[
\left(  f\otimes x\right)  u\left(  g\otimes y\right)  =f\left(  uy\right)
g\otimes x.
\]
It is clear that the map $u\longmapsto f\left(  uy\right)  $ is a bounded
linear functional on $U$. Indeed,%
\[
\left\vert f\left(  uy\right)  \right\vert \leq\left\Vert f\right\Vert
\left\Vert u\right\Vert \left\Vert y\right\Vert \leq\left(  \left\Vert
f\right\Vert \left\Vert y\right\Vert \right)  \left\Vert u\right\Vert _{{U}}.
\]
In the same time, $g\otimes x\in U$ for every $g\in X^{\ast}$ and $x\in Y$. So
$L_{a}R_{b}$ is a rank one operator.
This also shows that if $a_{i}$ and $b_{i}$ are nuclear then $L_{a_{i}%
}R_{b_{i}}$ is nuclear, because an absolutely convergent series of nuclear
operators is nuclear. Therefore $T$ is nuclear by the same reason.
Clearly
\[
\mathrm{trace}\left(  L_{f\otimes x}R_{g\otimes y}\right)  =f\left(  \left(
g\otimes x\right)  y\right)  =f\left(  x\right)  g\left(  y\right)
=\mathrm{trace}\left(  f\otimes x\right)  \mathrm{trace}\left(  g\otimes
y\right)  .
\]
Therefore, for nuclear $a_{i}=\sum f_{j}\otimes x_{j}$ and $b_{i}=\sum
g_{k}\otimes y_{k}$, we obtain that
\begin{align*}
\mathrm{trace}(L_{a_{i}}R_{b_{i}})  &  =\sum_{j,k}\mathrm{trace}%
(L_{f_{j}\otimes x_{j}}R_{g_{k}\otimes y_{k}})=\sum_{j,k}\mathrm{trace}%
({f_{j}\otimes x_{j}})\mathrm{trace}({g_{k}\otimes y_{k}})\\
&  =\mathrm{trace}(a_{i})\mathrm{trace}(b_{i}),
\end{align*}
whence
\[
\mathrm{trace}\left(  T\right)  =\sum\mathrm{trace}\left(  L_{a_{i}}R_{b_{i}%
}\right)  =\sum\mathrm{trace}\left(  a_{i}\right)  \mathrm{trace}\left(
b_{i}\right)  .
\]
\end{proof}

\subsection{Some constructions related to multiplication operators}

\subsubsection{ Integral operators}

Let ${U}\subset\mathcal{B}(X,Y)$ be a Banach operator bimodule. Assume first
that $U$ is reflexive as a Banach space. Let $(\Omega,\mu)$ be a measure
space, and let $L_{2}^{\mathcal{B}(Y)}(\Omega,\mu)$ be the space of all
measurable $\mathcal{B}(Y)$-valued functions $a:\omega\longmapsto a(\omega)$
with the norm
\[
\left\Vert a\right\Vert _{L_{2}}=\int_{\Omega}\left\Vert a(\omega)\right\Vert
^{2}d\mu<\infty.
\]
For any two functions $\omega\longmapsto a(\omega)\in L_{2}^{\mathcal{B}%
(Y)}(\Omega,\mu)$ and $\omega\longmapsto b(\omega)\in L_{2}^{\mathcal{B}%
(X)}(\Omega,\mu)$, one can define an operator $T_{a,b}$ on ${U}$ by means of
the Bochner integral \cite[Section 3.3.7]{HP}%
\[
T_{a,b}(x)=\int_{\Omega}a(\omega)xb(\omega)d\mu.
\]
Note that one may define $L_{2}^{\mathcal{K}(Y)}(\Omega,\mu)$ if replace
$\mathcal{B}\left(  Y\right)  $ by $\mathcal{K}\left(  Y\right)  $, and
consider $L_{2}^{\mathcal{K}(Y)}(\Omega,\mu)$ as a subspace of $L_{2}%
^{\mathcal{B}(Y)}(\Omega,\mu)$.

\begin{proposition}
\label{intsemi}Let ${U}\subset\mathcal{B}(X,Y)$ be a Banach operator bimodule,
and let $U$ be reflexive as a Banach space. Then $T_{a,b}\in
\widehat{\mathcal{E}}_{\mathcal{B}(Y),\mathcal{B}(X)}$ and $\Vert T_{a,b}%
\Vert_{\widehat{\mathcal{E}}_{\mathcal{B}(Y),\mathcal{B}(X)}}\leq\left\Vert
a\right\Vert _{L_{2}}\left\Vert b\right\Vert _{L_{2}}$. If in particular
$\omega\longmapsto a(\omega)\in L_{2}^{\mathcal{K}(Y)}(\Omega,\mu)$
(respectively, $\omega\longmapsto b(\omega)\in L_{2}^{\mathcal{K}(X)}%
(\Omega,\mu)$) then $T_{a,b}\in\widehat{\mathcal{E}}_{\mathcal{K}%
(Y),\mathcal{B}(X)}$ and $\Vert T_{a,b}\Vert_{\widehat{\mathcal{E}%
}_{\mathcal{K}(Y),\mathcal{B}(X)}}\leq\left\Vert a\right\Vert _{L_{2}%
}\left\Vert b\right\Vert _{L_{2}}$ (respectively, $T_{a,b}\in
\widehat{\mathcal{E}}_{\mathcal{B}(Y),\mathcal{K}(X)}$ and $\Vert T_{a,b}%
\Vert_{\widehat{\mathcal{E}}_{\mathcal{B}(Y),\mathcal{K}(X)}}\leq\left\Vert
a\right\Vert _{L_{2}}\left\Vert b\right\Vert _{L_{2}}$).
\end{proposition}

\begin{proof}
Let $a_{n}(\omega)\in L_{2}^{\mathcal{B}(Y)}(\Omega,\mu)$ be a sequence of
simple functions that tend to $a(\omega)$ in norm of $L_{2}^{\mathcal{B}%
(Y)}(\Omega,\mu)$. Clearly $T_{a,b}$ is the limit of operators $T_{a_{n},b}$
in the norm topology of $\mathcal{B}(U)$.
Each $a_{n}$ is a finite sum of functions $k_{i}\varphi_{i}(t)$, where
$k_{i}\in\mathcal{B}(Y)$ and $\varphi_{i}$ is the characteristic function of a
measurable set $\Lambda_{i}\subset\Omega$. Hence
\[
T_{a_{n},b}=\sum_{i}L_{k_{i}}R_{t_{i}}%
\]
where $t_{i}=\int_{\Lambda_{i}}b(t)d\mu\in\mathcal{B}\left(  X\right)  $.
Therefore $T_{a_{n},b}\in\mathcal{E}_{\mathcal{B}\left(  Y\right)
,\mathcal{B}\left(  X\right)  }$ and
\[
\Vert T_{a_{n},b}\Vert_{\widehat{\mathcal{E}}_{\mathcal{B}(Y),\mathcal{B}(X)}%
}\leq\Vert a_{n}\Vert_{L_{2}}\Vert b\Vert_{L_{2}}.
\]
It follows from this that the sequence of operators $T_{a_{n},b}$ is
fundamental in $\widehat{\mathcal{E}}_{\mathcal{B}(Y),\mathcal{B}(X)}$, so it
tends to some element $T\in\widehat{\mathcal{E}}_{\mathcal{B}(Y),\mathcal{B}%
(X)}$, and $\Vert T\Vert_{\widehat{\mathcal{E}}_{\mathcal{B}(Y),\mathcal{B}%
(X)}}\leq\Vert a\Vert_{L_{2}}\Vert b\Vert_{L_{2}}$. By the above, $T=T_{a,b}$
and we are done.
The other statements are proved similarly.
\end{proof}

Let $a\in L_{2}^{\mathcal{K}(Y)}(\Omega,\mu)$, $b\in L_{2}^{\mathcal{K}%
(X)}(\Omega,\mu)$, $s\in L_{2}^{{\mathcal{B}}(Y)}(\Omega,\mu)$, $t\in
L_{2}^{\mathcal{B}(X)}(\Omega,\mu)$. Then the operator $T_{a,b,s,t}$ defined
on a Banach operator bimodule $U\subset\mathcal{B}(X,Y)$ by the formula
\[
T_{a,b,s,t}(x)=\int_{I}a(\omega)xt(\omega)d\mu+\int_{I}s(\omega)xb(\omega
)d\mu,
\]
is called an \textit{integral semicompact operator}. Indeed, it follows from
Proposition \ref{intsemi} that this operator is semicompact multiplication operator.

If we wish to remove the restriction of reflexivity of $U$ and still have that
$T_{a,b,s,t}$ belongs to $\widehat{\mathcal{K}}_{\frac{1}{2}}(U)$, we should
impose continuity conditions which allow us to deal with Riemann integral sums
(see for instance \cite[Section 3.3.7]{HP}). For brevity we will formulate the
corresponding result in a form which is far from the most general.

\begin{theorem}
\label{intSemCom} Let ${U}\subset\mathcal{B}(X,Y)$ be a Banach operator
bimodule, and let $(\Omega,\mu)=\left(  I,\mu\right)  $, where $\mu$ is a
regular measure on an interval $I\subset\mathbb{R}$. Then every integral
semicompact operator $T_{a,b,s,t}$ belongs to $\widehat{\mathcal{K}}_{\frac
{1}{2}}(U)$ and
\[
\Vert T_{a,b,s,t}\Vert_{\widehat{\mathcal{K}}_{\frac{1}{2}}(U)}\leq\Vert
a\Vert_{L_{2}}\Vert t\Vert_{L_{2}}+\Vert b\Vert_{L_{2}}\Vert s\Vert_{L_{2}}.
\]

\end{theorem}

\subsubsection{Matrix multiplication operators\label{s432}}

Let $(T_{ji})_{j,i=1}^{n}$ be a matrix of multiplication operators on a Banach
operator bimodule ${U}\subset\mathcal{B}(X,Y)$. It defines an operator
$T=[T_{ij}]$ on $U^{\left(  n\right)  }$ by the formula
\[
T(x_{1},...,x_{n})=(y_{1},...,y_{n})\text{ where }y_{i}=\sum_{j}T_{ij}x_{j}.
\]
Let us denote the algebra of all such operators by $\mathbb{M}_{n}%
(\widehat{\mathcal{E}}_{\mathcal{B}(X),\mathcal{B}(Y)}(U))$. Also, by
$\mathbb{M}_{n}(\widehat{\mathcal{K}}_{\frac{1}{2}}(U))$ we denote the ideal
of $\mathbb{M}_{n}(\widehat{\mathcal{E}}_{\mathcal{B}(X),\mathcal{B}(Y)}(U))$
which consists of all operators $T=[T_{{ij}}]$ with $T_{ij}\in
\widehat{\mathcal{K}}_{\frac{1}{2}}(U)$ for $1\leq i,j\leq n$. In a similar
way we define the subspace $\mathbb{M}_{n}(\mathcal{F}_{\frac{1}{2}}(U))$ of
$\mathbb{M}_{n}(\widehat{\mathcal{K}}_{\frac{1}{2}}(U))$. The closure
$\overline{\mathbb{M}_{n}(\mathcal{F}_{\frac{1}{2}}(U))}$ of $\mathbb{M}%
_{n}(\mathcal{F}_{\frac{1}{2}}(U))$ in $\mathbb{M}_{n}(\widehat{\mathcal{K}%
}_{\frac{1}{2}}(U))$ is an ideal of $\mathbb{M}_{n}(\widehat{\mathcal{K}%
}_{\frac{1}{2}}(U))$.

\begin{theorem}
\label{MatQ} Let $U\subset\mathcal{B}(X,Y)$ be a Banach operator bimodule.
Then the algebra $\mathbb{M}_{n}(\widehat{\mathcal{K}}_{\frac{1}{2}%
}(U))/\overline{\mathbb{M}_{n}(\mathcal{F}_{\frac{1}{2}}(U))}$ is tensor radical.
\end{theorem}

\begin{proof}
The algebra $\mathbb{M}_{n}(\widehat{\mathcal{K}}_{\frac{1}{2}}(U))$ is
topologically isomorphic to $\mathbb{M}_{n}\otimes_{\gamma}(\widehat
{\mathcal{K}}_{\frac{1}{2}}(U))$, where $\mathbb{M}_{n}$ is the algebra of
$n\times n$ matrices. Furthermore, the algebra $\overline{\mathbb{M}%
_{n}(\mathcal{F}_{\frac{1}{2}}(U))}$ is topologically isomorphic to
$\mathbb{M}_{n}\otimes_{\gamma}\overline{\mathcal{F}_{\frac{1}{2}}(U)}$.
Hence, for every Banach algebra $A$, we have that
\begin{align*}
\left(  \mathbb{M}_{n}\left(  \widehat{\mathcal{K}}_{\frac{1}{2}}(U)\right)
/\overline{\mathbb{M}_{n}\left(  \mathcal{F}_{\frac{1}{2}}(U)\right)
}\right)  \widehat{\otimes}A  &  \cong\left(  \left(  \mathbb{M}_{n}%
\otimes_{\gamma}\widehat{\mathcal{K}}_{\frac{1}{2}}(U)\right)  /\left(
\mathbb{M}_{n}\otimes_{\gamma}\overline{\mathcal{F}_{\frac{1}{2}}(U)}\right)
\right)  \widehat{\otimes}A\\
&  \cong\left(  \mathbb{M}_{n}\otimes_{\gamma}\left(  \widehat{\mathcal{K}%
}_{\frac{1}{2}}(U)/\overline{\mathcal{F}_{\frac{1}{2}}(U)}\right)  \right)
\widehat{\otimes}A\\
&  \cong\left(  \widehat{\mathcal{K}}_{\frac{1}{2}}(U)/\overline
{\mathcal{F}_{\frac{1}{2}}(U)}\right)  \widehat{\otimes}\left(  \mathbb{M}%
_{n}\otimes_{\gamma}A\right)
\end{align*}
The latter algebra is radical because $\widehat{\mathcal{K}}_{\frac{1}{2}%
}(U)/\overline{\mathcal{F}_{\frac{1}{2}}(U)}$ is tensor radical by Corollary
\ref{semi}.
\end{proof}

\section{Multiplication operators on algebras satisfying compactness
conditions\label{s5}}


In this section we consider elementary and multiplication operators in the
most popular meaning: as elementary and multiplication operators with
coefficients in a Banach algebra acting on the algebra itself. In terms of the
previous section, we consider the case $A_{1}=A_{2}=U=A$, that is we regard
$A$ as an $A$-bimodule. For brevity we remove the indication of a bimodule in
our standard notation for the multiplication algebra: we write
$\widehat{\mathcal{E}}_{A,A}$ instead of $\widehat{\mathcal{E}}_{A,A}(A)$
(taking the occasion to use $\widehat{\mathcal{E}}_{I,J}$ for ideals
$I,J\subset A$). Furthermore, we denote the algebra of all elementary
operators on $A$ by $\mathcal{E\!\ell}(A)$ instead of $\mathcal{E\!\ell}%
_{A,A}(A)$.

To make our assumptions more concrete we impose various compactness conditions
on $A$. As we know, even the weakest of them, the hypocompactness of $A$,
implies that $\mathrm{Rad}(A)$ coincides with $\mathcal{R}_{t}(A)$.


\subsection{Multiplication operators on algebras commutative modulo the
radical}

Since radical hypocompact Banach algebras are tensor radical, we can apply
results of Section \ref{s4}.

\begin{corollary}
\label{hypomult} If $A$ is a hypocompact Banach algebra then the algebra
$\widehat{\mathcal{E}}_{\mathrm{Rad}(A),A}+\widehat{\mathcal{E}}%
_{A,\mathrm{Rad}(A)}$ is tensor radical and hypocompact.
\end{corollary}

\begin{proof}
The first statement follows from Corollary \ref{ideals1}. Arguing as in the
proof of Theorem \ref{mult1-} with using Theorem \ref{tenz} and Corollary
\ref{hypquot}, we prove the second statement.
\end{proof}

In particular all elementary operators $L_{a}+R_{b}+\sum_{i=1}^{n}L_{a_{i}%
}R_{b_{i}}$ with $a_{i}$ or $b_{i}$ in $\mathrm{Rad}(A)$ for each $i$ are
quasinilpotent elements of $\widehat{\mathcal{E}}_{A,A}$. Since the norm in
$\widehat{\mathcal{E}}_{A,A}$ majorizes the operator norm, they are
quasinilpotent operators on $A$.

Below by spectra of elementary operators we mean their spectra in the algebra
$\mathcal{B}(A)$ of all bounded operators on $A$. Clearly the unitalization
$\mathcal{E\!\ell}(A)^{1}$ of $\mathcal{E\!\ell}(A)$ consists of elements of
the form $\sum_{i=1}^{n}L_{a_{i}}R_{b_{i}}$ where $a_{i},b_{i}\in A^{1}$.

\begin{theorem}
\label{comput} Let $A$ be a hypocompact Banach algebra. If $u=\sum_{i=1}%
^{n}L_{a_{i}}R_{b_{i}}\in\mathcal{E\!\ell}(A)^{1}$, $v=\sum_{j=1}^{m}L_{c_{j}%
}R_{b_{j}}\in\mathcal{E\!\ell}(A)^{1}$ and all commutators $[a_{i},c_{j}]$ and
$[b_{i},d_{j}]$ belong to $\mathrm{Rad}(A)$ then
\begin{equation}
{\sigma}(u+v)\subset{\sigma}(u)+{\sigma}(v) \label{spincl1}%
\end{equation}
and
\begin{equation}
{\sigma}(uv)\subset{\sigma}(u){\sigma}(v). \label{spincl2}%
\end{equation}

\end{theorem}

\begin{proof}
Let us denote $A^{1}$ by $B$ and $\mathrm{Rad}(A)$ by $J$ for brevity. Let
$C=B{\widehat{\otimes}}B^{\mathrm{op}}$ and $E=J{\otimes}B^{\mathrm{op}%
}+B{\otimes}J^{\mathrm{op}}$. As $J=\mathcal{R}_{t}(A)$ by Theorem
\ref{hyphypura}, we have that $E\subset\mathrm{Rad}(C)$. Setting $u^{\prime
}=\sum_{i}a_{i}{\otimes}b_{i}$ and $v^{\prime}=\sum_{j}c_{j}{\otimes}b_{j}$,
we have that $[u^{\prime},v^{\prime}]\in E$ (because $[a{\otimes}b,c{\otimes
}d]=[a,c]{\otimes}bd+ca{\otimes}[b,d]$). So these elements commute modulo the
radical of $C$.
Let $\phi:C\longrightarrow\mathcal{B}(A)$ be the homomorphism sending
$a\otimes b$ to $L_{a}R_{b}$ for every $a,b\in B$. Then the algebra
$D=\phi(C)$ supplied with the norm of the quotient $C/\ker\phi$ is a Banach
subalgebra of $\mathcal{B}(A)$. The elements $u=\phi(u^{\prime})$ and
$v=\phi(v^{\prime})$ commute modulo $\phi(\mathrm{Rad}(C))\subset
\mathrm{Rad}(D)$. Hence
\begin{equation}
{\sigma}_{D}(u+v)\subset{\sigma}_{D}(u)+{\sigma}_{D}(v) \label{spincl1P}%
\end{equation}
and
\begin{equation}
{\sigma}_{D}(uv)\subset{\sigma}_{D}(u){\sigma}_{D}(v), \label{spincl2P}%
\end{equation}
where $\sigma_{D}(x)$ denotes the spectrum of $x\in D$ with respect to $D$.
As $C$ is hypocompact by Theorem \ref{tenz}, and $D$ is isomorphic to a
quotient of $C$, then $D$ is hypocompact by Corollary \ref{hypquot}. By
Theorem \ref{scat}, ${\sigma}_{D}(u)$ and ${\sigma}_{D}(v)$ are finite or
countable. Using Proposition \ref{sp}, we get that ${\sigma}_{D}(u)=\sigma(u)$
and ${\sigma}_{D}(v)=\sigma(v)$. So the inclusions (\ref{spincl1P}) and
(\ref{spincl2P}) imply (\ref{spincl1}) and (\ref{spincl2}).
\end{proof}

Let $A$ be a Banach algebra. Recall that the center modulo the radical\ or
\textquotedblleft$\mathrm{Rad}$-center\textquotedblright\ $Z_{\mathrm{Rad}%
}(A)$ is the set $\{a\in A:[a,x]\in\mathrm{Rad}(A)\text{ for all }x\in A\}$.

\begin{corollary}
If $A$ is a hypocompact Banach algebra and $u\in L_{Z_{\mathrm{Rad}}(A)}%
R_{A}+L_{A}R_{Z_{\mathrm{Rad}}(A)}$ then inclusions $(\ref{spincl1})$ and
$(\ref{spincl2})$ hold for all $v\in\mathcal{E\!\ell}(A)$.
\end{corollary}

\medskip Let us call a subalgebra $B$ of a Banach algebra \textit{spectrally
computable} if inclusions (\ref{spincl1}) and (\ref{spincl2}) hold for all
elements $u,v\in B$.

\begin{corollary}
\label{hypcomput} If $A$ is a hypocompact Banach algebra commutative modulo
the radical then $\mathcal{E\!\ell}(A)$ is a spectrally computable subalgebra
of $\mathcal{B}\left(  A\right)  $.
\end{corollary}

\begin{remark}
\emph{In virtue of Corollary \ref{sp3} and by continuity of the spectrum on
operators with countable spectra,} \emph{the previous results as well as the
results of Section \ref{ss2} can be extended to multiplication operators. But
we prefer to present them in less general and more traditional setting of
elementary operators.}
\end{remark}

\subsection{Engel algebras\label{ss2}}

A Banach Lie algebra $\mathcal{L}$ is called \textit{Engel} if all operators
$\mathrm{ad}_{\mathcal{L}}(a):x\rightarrow\lbrack a,x]$ on $\mathcal{L}$ are
quasinilpotent. This is a natural functional-analytic extension of the class
of nilpotent Lie algebras because the latter can be defined as Lie algebras
for which all operators $\mathrm{ad}_{\mathcal{L}}\left(  a\right)  $ are
nilpotent of some restricted order \cite{Zelmanov}. A Banach algebra $A$ is
said to be \textit{Engel} if it is Engel as a Banach Lie algebra that is if
all operators $L_{a}-R_{a}$ are quasinilpotent. It is proved in
\cite[Proposition 5.21]{ST2005} (and can be easily deduced from a more general
result of Aupetit and Mathieu \cite{Aup-Ma}) that all Engel Banach algebras
are commutative modulo the radical. We call $A$ \textit{strongly Engel} if
\[
{\sigma}\left(  \sum_{i=1}^{n}L_{a_{i}}R_{b_{i}}\right)  \subset{\sigma
}\left(  \sum_{i=1}^{n}a_{i}b_{i}\right)
\]
for all $a_{i},b_{i}\in A^{1}$ and $n\in\mathbb{N}$. It will be shown below
that for hypocompact Banach algebras these notions coincide.

\begin{theorem}
\label{eng} Let $A$ be a hypocompact Banach algebra commutative modulo the
radical. Suppose that $A$ is generated (as a Banach algebra) by a subset $M$
such that the operators $L_{a}-R_{a}$ are quasinilpotent for all $a\in M$.
Then $A$ is strongly Engel.
\end{theorem}

\begin{proof}
By Corollary \ref{hypcomput}, the algebra $\mathcal{E\!\ell}(A)$ is spectrally
computable. Using this fact, it can be easily shown that the set $E$ of all
$a\in A$, for which ${\sigma}(L_{a}-R_{a})=\{0\}$ is a subalgebra of $A$.
Indeed, if $a,b\in E$ then
\begin{align*}
\sigma(L_{a+b}-R_{a+b})  &  =\sigma(L_{a}-R_{a}+L_{b}-R_{b})\\
&  \subset\sigma(L_{a}-R_{a})+\sigma(L_{b}-R_{b})=\{0\},
\end{align*}
so $a+b\in E$. Furthermore,
\begin{align*}
\sigma(L_{ab}-R_{ab})  &  =\sigma(L_{a}(L_{b}-R_{b})+(L_{a}-R_{a}%
)R_{b}+R_{ab-ba})\\
&  \subset\sigma(L_{a})\sigma(L_{b}-R_{b}))+\sigma(L_{a}-R_{a})\sigma
(R_{b})+\sigma(R_{ab-ba})=\{0\}
\end{align*}
whence $ab\in E$.
Since $A$ is generated by $M$ and $M\subset E$, the subalgebra $E$ is dense in
$A$.
Since $A$ is hypocompact, its elements have countable spectra by Theorem
\ref{scat} and therefore the spectra of all operators $L_{a}-R_{a}$ are
countable. Hence they are the points of continuity of the spectral radius. It
follows that $E$ is closed whence $E=A$. We proved that $A$ is an Engel algebra.
To see that $A$ is strongly Engel, note that an operator $\sum_{i}^{n}%
L_{a_{i}}R_{b_{i}}$ can be written as $\sum_{i}^{n}L_{a_{i}}(R_{b_{i}%
}-L_{b_{i}})+L_{c}$, where $c=\sum_{i}^{n}a_{i}b_{i}$. Since $\mathcal{E\!\ell
}(A)$ is spectrally computable and $\sigma(R_{b_{i}}-L_{b_{i}})=\{0\}$, we
obtain that
\[
\sigma\left(  \sum_{i}^{n}L_{a_{i}}R_{b_{i}}\right)  \subset\sigma
(L_{c})\subset\sigma(c).
\]
\end{proof}

\begin{corollary}
\label{genEng} Let $A$ be a hypocompact Banach algebra generated by an Engel
closed Lie subalgebra $\mathcal{L}$. Then $A$ is a strongly Engel Banach
algebra commutative modulo the radical.
\end{corollary}

\begin{proof}
Without loss of generality, we consider the case when $A$ is unital and
generated as a Banach algebra by $\mathcal{L}$ and the identity element $1$.
Let us show first that all operators $L_{a}-R_{a}$ with $a\in\mathcal{L}$ are
quasinilpotent on $A$ (by our assumptions, they are quasinilpotent on
$\mathcal{L}$). Indeed, they are bounded derivations of $A$ and their spectra
are countable (because spectra of elements of $A$ are countable by Theorem
\ref{scat}). By \cite[Corollary 3.7]{ST2005}, they are quasinilpotent on the
closed subalgebra of $A$ generated by $\mathcal{L}$, that is on $A$. Taking
into account Theorem \ref{eng}, we have only to prove that $A$ is commutative
modulo the radical.
Let $\pi$ be a strictly irreducible representation of $A$ on $X$. We will
obtain a contradiction assuming that $\dim(X)>1$. Changing $A$ by $A/\ker{\pi
}$, one may suppose that $\pi$ is faithful.
We already know that spectra of elements of $A$ are countable. Now we claim
that the spectra of elements of $\mathcal{L}$ are one-point. Indeed, if
$\sigma\left(  a\right)  $ is not a singleton for some element $a\in
\mathcal{L}$ then it is not connected and, by \cite[Proposition 3.16]{ST2005},
$a$ has a non-trivial Riesz projection $p$ commuting with $\mathcal{L}$. Hence
$p$ is in the center of $A$, which is impossible, because $A$ is a primitive
Banach algebra.
Define the function $h:\mathcal{L}\longrightarrow\mathbb{C}$ by $h\left(
a\right)  =\lambda$ if $\sigma\left(  a\right)  =\left\{  \lambda\right\}  $,
for every $a\in\mathcal{L}$. We claim that $h$ is a character of $\mathcal{L}%
$, i.e. $h$ is a bounded linear functional on $\mathcal{L}$ that vanishes on
$\left[  \mathcal{L},\mathcal{L}\right]  $. Indeed, by using Proposition
\ref{trans}, one can find a proper closed ideal $J$ of $A$ such that $A/J$ is
a compact Banach algebra. As $1/J$ is a compact element, then clearly it is a
finite rank element. Hence $A/J$ is finite-dimensional. As $\sigma\left(
a/J\right)  \subset\sigma\left(  a\right)  $ for every $a\in\mathcal{L}$, then
one can define the function $g:\mathcal{L}/J\longrightarrow\mathbb{C}$ by
$g\left(  a/J\right)  =\lambda$ if $\sigma\left(  a/J\right)  =\left\{
\lambda\right\}  $ for every $a\in\mathcal{L}$. It is clear that $g\left(
a/J\right)  =h\left(  a\right)  $ for every $a\in\mathcal{L}$ and
$\mathcal{L}/J$ is a nilpotent Lie subalgebra of $A/J$, whence $A/J$ is
commutative modulo the radical by the Engel theorem and $g$ is a character of
$\mathcal{L}/J$. So $h$ is a character of $\mathcal{L}$.
Now, replacing every element $a$ of $\mathcal{L}$ with $\sigma\left(
a\right)  =\left\{  \lambda\right\}  $ by $a-\lambda$, one may assume that
$\mathcal{L}$ consists of quasinilpotent elements. Then
\[
G=\exp(\mathcal{L})=\{\exp(a):a\in\mathcal{L}\}
\]
is a group by \cite{WWQ}, and $\sigma(b)=\{1\}$ for each $b\in G$. It follows
that $r(M)=1$ for every precompact subset $M$ of $G$. By Corollary
\ref{hypoBW}, $\rho(M)=1$. In particular,
\begin{equation}
\rho(M\cup\left\{  1,\exp\left(  \lambda a\right)  \right\}  )=1 \label{for}%
\end{equation}
for every $a\in\mathcal{L}$ and $\lambda\in\mathbb{C}$.
Choose an arbitrary element $a\in\mathcal{L}$ and define the function $f$ on
$\mathbb{C}$ by $f(\lambda)=\rho(M(\exp(\lambda a)-1)/\lambda)$ for a fixed
precompact subset $M$ of $G$. This function is subharmonic by \cite[Theorem
3.5]{ST2000} and tends to zero when $\lambda\rightarrow\infty$, because
\begin{align*}
\rho\left(  M\left(  \exp(\lambda a)-1\right)  \right)   &  \leq\rho(\left(
M\cup\left\{  1\right\}  \right)  \exp(\lambda a)-M\cup\left\{  1\right\}  )\\
&  \leq\rho(\left(  M\cup\{1,\exp(\lambda a)\}\right)  ^{2}-\left(
M\cup\{1,\exp(\lambda a)\}\right)  ^{2})\\
&  \leq\rho\left(  2\,\mathrm{abs}\left(  \left(  M\cup\{1,\exp(\lambda
a)\}\right)  ^{2}\right)  \right) \\
&  =2\rho\left(  \left(  M\cup\{1,\exp(\lambda a)\}\right)  ^{2}\right) \\
&  =2\rho\left(  M\cup\{1,\exp(\lambda a)\}\right)  ^{2}=2,
\end{align*}
by (\ref{for}), where $\mathrm{abs}\left(  S\right)  $ denotes the absolutely
convex hull of a bounded set $S\subset A$ and the equality $\rho\left(
\mathrm{abs}\left(  S\right)  \right)  =\rho\left(  S\right)  $
\cite[Proposition 2.6]{ST2000} easily follows from the characterization of
$\rho\left(  S\right)  $ as infimum of $\left\Vert S\right\Vert ^{\prime}$
when $\left\Vert \cdot\right\Vert ^{\prime}$ runs over all algebra norms
equivalent to given one. Hence $f$ is constant and, moreover, $f(\lambda)=0$
for all $\lambda$. In particular, $f(0)=0$, whence it is easy to see that
\[
\rho(aM)=0.
\]
Now if an element $x$ belongs to the linear span of $M$ then
\[
\rho(ax)=0
\]
by Lemma \ref{comb}. Since linear span $\mathrm{lin}(G)$ of $G$ is a
subalgebra of $A$ and the closure of $\mathrm{lin}(G)$ contains $\mathcal{L}$,
we conclude that $\mathrm{lin}(G)$ is dense in $A$. The previous argument
shows that
\[
\rho(ax)=0
\]
for all $x\in\mathrm{lin}(G)$. Since spectra of elements of $A$ are countable,
the spectral radius is continuous on $A$, whence
\[
\rho(ax)=0
\]
for all $x\in A$. This means that $a\in\mathrm{Rad}(A)$. So $\mathcal{L}%
\subset\mathrm{Rad}(A)$ and the closed algebra generated by $\mathcal{L}$ is
radical. Then $\dim(X)=1$, a contradiction.
\end{proof}

\subsection{Generalized multiplication operators}

It follows from the above that if $A$ is a radical\textit{ }hypocompact
algebra then all operators in $\mathcal{E\!\ell}\left(  A\right)  $ and, more
generally, in $\widehat{\mathcal{E}}_{A,A}$ are quasinilpotent. Here we
discuss the possibility to extend this result to operators in the norm-closure
$\mathrm{Mul}(A)$ of $\mathcal{E\!\ell}\left(  A\right)  $ in $\mathcal{B}%
(A)$. Note first of all that it is an open problem if $\mathrm{Mul}(A)$ is
radical even for bicompact $A$ (and even for the case that $A$ is a radical
closed algebra of compact operators). We call elements of $\mathrm{Mul}(A)$
\textit{generalized multiplication operators}.

\begin{theorem}
\label{mul0.2} If $A$ is a compact algebra and $J=\mathrm{Rad}(A)$ then the
closed ideal $I$ generated by $L_{J}R_{A}\cup L_{A}R_{J}$ is contained in
$\mathrm{Rad}(\mathrm{Mul}(A))$. As a consequence, $L_{J}+R_{J}+I$ consists of quasinilpotents.
\end{theorem}

\begin{proof}
Since $A$ is compact, $L_{a^{2}}R_{b}$ is compact for any $a,b\in A$. Indeed,
since
\[
L_{a}R_{b}+L_{b}R_{a}=1/2(L_{a+b}R_{a+b}-L_{a-b}R_{a-b})
\]
is a compact operator, the same is true for
\[
L_{a^{2}}R_{b}=L_{a}(L_{a}R_{b}+L_{b}R_{a})-(L_{a}R_{a})L_{b}.
\]
Now it follows from Corollary \ref{hypomult} that if $b\in J$ then $L_{a^{2}%
}R_{b}T$ is a compact quasinilpotent operator for every elementary operator
$T$ on $A$. By continuity of the spectral radius, the same is true for all
$T\in\mathrm{Mul}(A)$. Hence $L_{a^{2}}R_{b}\in\mathrm{Rad}(\mathrm{Mul}(A))$.
By the Nagata-Higman theorem (for $n=2$), every product $a_{1}a_{2}a_{3}$ of
elements of $A$ can be represented as a finite combination of elements of form
$x^{2}y$ and $uv^{2}$. Since clearly $L_{x^{2}y}R_{b},L_{uv^{2}}R_{b}%
\in\mathrm{Rad}(\mathrm{Mul}(A))$, then
\[
L_{A^{3}}R_{J}\subset\mathrm{Rad}(\mathrm{Mul}(A)).
\]
This implies that
\[
\left(  L_{A}R_{J}\right)  ^{3}\subset\mathrm{Rad}(\mathrm{Mul}(A)).
\]
As the Jacobson radical of a Banach algebra is closed, we obtain that
$\overline{L_{A}R_{J}}$ consists of quasinilpotent operators. Since
$\overline{L_{A}R_{J}}$ is an ideal of $\mathrm{Mul}(A)$, then
\[
\overline{L_{A}R_{J}}\subset\mathrm{Rad}(\mathrm{Mul}(A)).
\]
We proved that $L_{A}R_{J}\subset\mathrm{Rad}(\mathrm{Mul}(A))$. Similarly, we
have that
\[
L_{J}R_{A}\subset\mathrm{Rad}(\mathrm{Mul}(A)).
\]
Hence $I\subset\mathrm{Rad}(\mathrm{Mul}(A))$. Then $L_{J}+R_{J}+I$ consists
of quasinilpotents if and only if $L_{J}/I+R_{J}/I$ consists of
quasinilpotents in $\mathrm{Mul}(A)/I$. The last is obvious because $L_{a}/I$
and $R_{b}/I$ commute for every $a,b\in A$ and are quasinilpotents for every
$a,b\in J$.
\end{proof}

Let us denote by $\mathrm{Mul}_{2}(A)$ the closed subalgebra of $\mathrm{Mul}%
(A)$ generated by all operators $L_{a}R_{b}$, where $a,b\in A$.

\begin{corollary}
\label{mul2} If a radical Banach algebra $A$ is compact then
\[
\mathrm{Mul}_{2}(A)\subset\mathrm{Rad}(\mathrm{Mul}(A)).
\]
As a consequence, the algebra $\mathcal{E\!\ell}(A)+\mathrm{Mul}_{2}(A)$
consists of quasinilpotent operators.
\end{corollary}

\subsection{Permanently radical algebras}

Let us call a class $\mathcal{P}$ of Banach algebras \textit{permanent} if for
each $A\in\mathcal{P}$ and each bounded homomorphism $f:A\longrightarrow B$
with dense image, the algebra $B$ is also in $\mathcal{P}$. Examples of
permanent classes are commutative algebras, separable algebras,
finite-dimensional algebras, amenable algebras, algebras with bounded
approximate identities.

An example of Dixon \cite[Example 9.3]{Dix} shows that the class of all
radical Banach algebras is not permanent. We say that a Banach algebra $A$ is
\textit{permanently radical} if for every bounded homomorphism
$f:A\longrightarrow B$ the closure of the image $f(A)$ in $B$ is a radical
Banach algebra.

It follows from (i) of the following theorem that the class of all permanently
radical Banach algebras is permanent. We also show that it is extension stable.

\begin{theorem}
\label{prad}Let $A$ be a Banach algebra.

\begin{itemize}
\item[$\mathrm{(i)}$] If $A$ is permanently radical then so is $\overline
{g\left(  A\right)  }$ for every bounded homomorphism $g:A\longrightarrow B$
of Banach algebras.

\item[$\mathrm{(ii)}$] If a closed ideal $J$ and the quotient $A/J$ of $A$ are
permanently radical then $A$ is permanently radical.

\item[$\mathrm{(iii)}$] If $\{I_{\alpha}\}_{\alpha\in\Lambda}$ is an
increasing net of permanently radical closed ideals in $A$ and $\cup
_{\alpha\in\Lambda}I_{\alpha}$ is dense in $A$ then $A$ is permanently radical.
\end{itemize}
\end{theorem}

\begin{proof}
(i) Let $C=\overline{g\left(  A\right)  }$ and $f:C\longrightarrow D$ be a
bounded homomorphism of Banach algebras with $\overline{f(C)}=D$. Then
$f{\circ g}$ is a bounded homomorphism $A\longrightarrow D$ with dense image.
If $A$ is permanently radical then $D$ is radical.
(ii) Let $f:A\longrightarrow B$ be a bounded homomorphism with $\overline
{f(A)}=B$. Then $I:=\overline{f(J)}$ is a radical ideal of $B$. Hence
$I\subset\mathrm{Rad}(B)$, whence $g=q_{\mathrm{Rad}(B)}{\circ}f$ is a bounded
homomorphism of $A$ into $C=B/\mathrm{Rad}(B)$ and $I\subset\ker g$. Thus
there is a bounded homomorphism $h:A/J\rightarrow C$ such that $g=h{\circ
}q_{J}$. As $C=\overline{h(A/J)}$, $C$ is radical. Then $C=0$, whence $B$ is radical.
(iii) If $f:A\rightarrow B$ is a bounded homomorphism with dense image then
all $\overline{f(I_{\alpha})}$ are radical ideals of $B$. Hence $f(I_{\alpha
})\subset\mathrm{Rad}(B)$, for each $\alpha$, and $f(A)\subset\mathrm{Rad}(B)$
by density. Thus $B=\mathrm{Rad}(B)$.
\end{proof}

\begin{remark}
It is not clear if $\mathrm{(ii)}$ may be reversed. It follows from
$\mathrm{(i)}$ that a quotient of a permanently radical Banach algebra is
permanently radical, but what one can say about ideals?
\end{remark}

Clearly the class of all permanently radical Banach algebras contains all
radical commutative Banach algebras and all finite-dimensional radical algebras.

\begin{theorem}
\label{harp} Every radical hypofinite Banach algebra $A$ is permanently radical.
\end{theorem}

\begin{proof}
Let us show first that each topologically irreducible representation $\pi$ of
$A$ on a Banach space $X$ is zero. Indeed, assume that $\pi\neq0$ and let
$J=\ker\pi$, then $A/J$ contains a non-zero finite rank element $a/J$. Since
$\pi(a)\neq0$ there is $0\neq x\in X$ with $\pi(a)x\neq0$, whence $\pi
(A)\pi(a)x$ is a dense subspace of $X$. Since $\pi(a)\pi(A)\pi(a)x=\pi(aAa)x$
is a finite-dimensional subspace, we conclude that $\dim(\pi(a)X)<\infty$. Let
$I=\pi(A)\cap\mathcal{F}(X)$, this is a non-zero ideal of $\pi(A)$. Hence $I$
has no closed invariant subspaces. On the other hand, $I$ consists of
nilpotent operators (indeed, if a finite rank operator is the image of a
quasinilpotent element under a representation then it is nilpotent). By the
Lomonosov Theorem \cite{Lom} (or by an earlier result of Barnes \cite{Barnes}%
), $I$ has an invariant subspace. This contradiction shows that $\pi=0$.
Let now $f:A\longrightarrow B$ be a continuous homomorphism with dense image.
If $\pi$ is a strictly irreducible representation of $B$ then $\pi\circ f$ is
a topologically irreducible representation of $A$ whence $\pi\circ f=0$ and
$\pi=0$. This shows that $B$ is radical.
\end{proof}

\begin{corollary}
\label{harp1} If $A$ is a hypofinite Banach algebra then $\mathrm{Rad}(A)$ is
permanently radical.
\end{corollary}

\begin{proof}
$\mathrm{Rad}(A)$ is a hypofinite Banach algebra, because it is an ideal of
$A$ (see Corollary \ref{ihf}). So apply Corollary \ref{harp}.
\end{proof}

It would be convenient to formulate a result established in the proof of
Theorem \ref{harp} as follows.

\begin{proposition}
\label{TI} Each topologically irreducible representation of a radical
hypofinite Banach algebra is trivial.
\end{proposition}

Is any radical bicompact Banach algebra permanently radical? Note that the
positive answer would imply that all hypocompact radical Banach algebras are
permanently radical. But even if the answer is affirmative it needs another
approach because the following result shows that Proposition \ref{TI} doesn't
extend to radical hypocompact algebras.

\begin{theorem}
\label{read-bons} There is a radical bicompact, singly generated Banach
algebra $A$ with a non-trivial topologically irreducible contractive
representation by bounded operators.
\end{theorem}

\begin{proof}
Let $T$ be a quasinilpotent operator on a Banach space $X$ without non-trivial
closed invariant subspaces (the existence of such operators is a famous
example by Read \cite{Read}). Let $B$ be the subalgebra of ${\mathcal{B}}(X)$
generated by $T$. It follows from Bonsall's theorem \cite[Theorem 3]{Bons}
that there is an algebra norm $\Vert\cdot\Vert^{\prime}$ on $B$ such that
\begin{itemize}
\item[$\mathrm{1)}$] $\Vert a\Vert\leq\Vert a\Vert^{\prime}$ for each $a\in B$,
\item[$\mathrm{2)}$] the completion $A$ of $B$ in $\left\Vert \cdot\right\Vert
^{\prime}$ is a Banach subalgebra of $\mathcal{B}\left(  X\right)  $,
\item[$\mathrm{3)}$] the element $b$ of $A$ corresponding to $T$ is compact.
\end{itemize}
\noindent Since $A$ is generated by $b$, it is a bicompact, singly generated
Banach algebra. As every compact element of a Banach algebra has countable
spectrum by \cite[Theorem 4 .4]{A68}, $\sigma_{A}\left(  b\right)
=\sigma\left(  T\right)  $ by Proposition \ref{sp}(ii). Hence $b$ is a
quasinilpotent element of $A$, and $A$ is radical. As $A$ is embedded into
$\mathcal{B}\left(  X\right)  $, let $\pi$ be the natural representation of
$A$ by bounded operators on $X$. Then ${\pi}(b)=T$, and $\pi$ is topologically
irreducible and contractive.
\end{proof}

\begin{theorem}
\label{mulperm0} If $A$ is a compact Banach algebra and $\mathrm{Rad}(A)$ is
permanently radical, then $L_{\mathrm{Rad}(A)}\cup R_{\mathrm{Rad}(A)}%
\subset\mathrm{Rad}(\mathrm{Mul}(A)).$
\end{theorem}

\begin{proof}
Let $I=\mathrm{Rad}(\mathrm{Mul}(A))$, $C=\mathrm{Mul}(A)/I$ and $q=q_{I}$.
Define $\phi:\mathrm{Rad}(A)\rightarrow C$ by $\phi(a)=q(L_{a})$ for any $a$.
Then the algebra $\overline{\phi(\mathrm{Rad}(A))}$ is radical.
For any $a\in\mathrm{Rad}(A)$ and $T\in\mathcal{E\!\ell}(A)$, we have that
\[
L_{a}T\in L_{\mathrm{Rad}(A)}+L_{\mathrm{Rad}(A)}R_{A}\subset L_{\mathrm{Rad}%
(A)}+\mathrm{Rad}(\mathrm{Mul}(A))
\]
by Theorem \ref{mul0.2}. It follows that $q(L_{a}T)\in\overline{\phi
(\mathrm{Rad}(A))}$. By continuity, the same is true for all $T\in
\mathrm{Mul}(A)$. Thus all $q(L_{a}T)$ are quasinilpotent. This shows that
$L_{\mathrm{Rad}(A)}\mathrm{Mul}(A)$ consists of quasinilpotents, whence
$L_{\mathrm{Rad}(A)}\subset\mathrm{Rad}(\mathrm{Mul}(A))$. Similarly, we have
that $R_{\mathrm{Rad}(A)}\subset\mathrm{Rad}(\mathrm{Mul}(A))$.
\end{proof}

\begin{corollary}
\label{mulperm1} If $A$ is an approximable Banach algebra then
\[
L_{\mathrm{Rad}(A)}\cup R_{\mathrm{Rad}(A)}\subset\mathrm{Rad}(\mathrm{Mul}%
(A)).
\]

\end{corollary}

\begin{proof}
Clearly $A$ is compact. Furthermore, ${\mathrm{Rad}(A)}$ is permanently
radical by Corollary \ref{harp1}.
\end{proof}

\begin{corollary}
\label{mulperm2} If $A$ is an approximable Banach algebra and $A$ is
commutative modulo $\mathrm{Rad}(A)$, then $\mathrm{Mul}(A)$ is commutative
modulo $\mathrm{Rad}(\mathrm{Mul}(A))$.
\end{corollary}

\begin{proof}
For all $a,b\in A$, $[L_{a},L_{b}]=L_{[a,b]}\in L_{\mathrm{Rad}(A)}%
\subset{\mathrm{Rad}(}${$\mathrm{Mul}$}${(A))}$ and, similarly, $[R_{a}%
,R_{b}]\in{\mathrm{Rad}(}${$\mathrm{Mul}$}${(A))}$. Since also $[L_{a}%
,R_{b}]=0\in{\mathrm{Rad}(}${$\mathrm{Mul}$}${(A))}$, we get that
$\mathrm{Mul}(A)/\mathrm{Rad}(\mathrm{Mul}(A))$ is commutative.
\end{proof}

\subsection{Chains of closed ideals}

Now we consider invariant subspaces of the algebras of elementary operators.
It was proved by Wojty\'{n}ski \cite{W78} that the well known problem of the
existence of a non-trivial closed ideal in a radical Banach algebra has the
positive answer if the algebra has a non-zero compact element. The proof of
this fact, based on the invariant subspace theorem for Volterra semigroups is
given in \cite{T1998}. The following theorem presents another proof and a
slightly more general formulation of this result.

Recall that a \textit{central multiplier} on a Banach algebra $A$ is a bounded
linear operator on $A$ commuting with left and right multiplications.

\begin{theorem}
\label{ideal} If a radical Banach algebra $A$ has a non-zero compact element
then either the multiplication in $A$ is trivial or $A$ has a closed ideal
invariant under all central multipliers.
\end{theorem}

\begin{proof}
Let an element $a\in A$ be compact. Then $I=\left\{  b:L_{a}R_{b}%
\in\mathcal{K}(A)\right\}  $ is a non-zero closed ideal in $A$ invariant under
central multipliers. So we have to assume that $I=A$. Setting $J=\left\{
c:L_{c}R_{b}\in\mathcal{K}(A)\right\}  \text{ for all }b\in A$, we similarly
reduce to the case that $J=A$. In other words, we may suppose that $A$ is bicompact.
Recall that a subspace invariant under an algebra of operators and its
commutant is called \textit{hyperinvariant} for this algebra. Note that the
set of all central multipliers is the commutant of $\mathrm{Mul}(A)$. So our
aim is to show that $\mathrm{Mul}(A)$ has a non-trivial hyperinvariant
subspace. Since $\mathrm{Mul}_{2}(A)$ is an ideal in $\mathrm{Mul}(A)$ it
suffices to show the same for $\mathrm{Mul}_{2}(A)$ (see for example
\cite{T1998}). By Corollary \ref{mul2}, $\mathrm{Mul}_{2}(A)$ is a radical
algebra of compact operators. Hence it has a hyperinvariant subspace by
\cite{Sh84}.
\end{proof}

Is this possible to strengthen the result and to prove the existence of a
total chain of closed ideals? We will show that the answer to this question is negative.

Recall that a chain (i.e. a set linearly ordered by the inclusion)
$\mathfrak{N}$ of closed subspaces of a Banach space $X$ is \textit{total} if
it is not contained in a larger chain of subspaces. This is equivalent to the
conditions that $\mathfrak{N}$ is complete and for any elements $Y_{1}\subset
Y_{2}$ of $\mathfrak{N}$, either $\dim(L_{2}/L_{1})=1$ or there exists an
intermediate subspace in $\mathfrak{N}$.

Let us call by \textit{a gap} in the lattice of closed ideals of a Banach
algebra $A$ a pair $I_{1}\subset I_{2}$ of closed ideals without intermediate
ideals, and in this case the quotient $I_{2}/I_{1}$ is called a
\textit{gap-quotient }of the lattice. It is easy to show by transfinite
induction that if $\dim(I_{2}/I_{1})=1$ for any gap, then $A$ has a total
chain of closed ideals (moreover, each chain of ideals extends to a total one).

An example of a gap is a pair $(0,I)$ where $I$ is a minimal closed ideal. So
if each chain of closed ideals in $A$ extends to a total one then each minimal
closed ideal is one-dimensional. We show now that these properties can fail in
the class of radical bicompact algebras. Then it will be shown that for
radical hypofinite algebras the situation is different.

\begin{theorem}
\label{contr}

\begin{itemize}
\item[$\mathrm{(i)}$] There is a radical bicompact Banach algebra without a
total chain of closed ideals.

\item[$\mathrm{(ii)}$] A radical bicompact Banach algebra can have an
infinite-dimensional minimal closed ideal.
\end{itemize}
\end{theorem}

\begin{proof}
Let $A$ be a commutative bicompact radical Banach algebra with a topologically
irreducible representation ${\pi}:A\longrightarrow\mathcal{B}(X)$ (see Theorem
\ref{read-bons}). On the Banach space $B=A{\oplus}X$ with the norm $\Vert
a{\oplus}x\Vert=\max\{\Vert a\Vert,\Vert x\Vert\}$ introduce a multiplication
by $(a{\oplus}x)(b{\oplus}y)=ab{\oplus}{\pi}(a)y$. Then $B$ is a Banach
algebra. Since $\left(  a{\oplus}x\right)  ^{n}=a^{n}{\oplus}{\pi}(a^{n-1})x$,
then $\Vert(a{\oplus}x)^{n}\Vert\leq(\Vert a\Vert+\Vert x\Vert)\Vert
a^{n-1}\Vert$ for every $n>0$, whence $B$ is radical.
We show that $B$ is a bicompact algebra. For any $a{\oplus}x,b{\oplus}y\in B$,
the operator $T=L_{a{\oplus}x}R_{b{\oplus}y}$ maps any $c{\oplus}z$ into
$acb{\oplus}{\pi}(ac)y$. As $\mathrm{ball}(B)=\mathrm{ball}(A){\oplus
}\mathrm{ball}(X)$, we obtain that
\[
T(\mathrm{ball}(B))\subset L_{a}R_{b}(\mathrm{ball}(A)){\oplus}{\pi}%
(L_{a}(\mathrm{ball}(A))y).
\]
As all operators $L_{a}R_{b}$ in $A$ are compact, it suffices to prove the
precompactness of the set ${\pi}(L_{a}(\mathrm{ball}(A))y)$. In other words,
we have to show that any operator $S_{y}:c\longmapsto{\pi}(ac)y$ is compact.
If take $y\in{\pi}(A)X$ with $y={\pi}(d)z$ for some $d\in A$ and $z\in X$,
then $S_{y}$ is compact because it decomposes through $L_{a}R_{d}$. It follows
that $S_{y}$ is compact for any $y$ in the linear span $Y$ of ${\pi}(A)X$. But
$Y$ is dense in $X$ because it is invariant for ${\pi}(A)$. Hence for any
$y\in X$ there is a sequence $y_{n}\rightarrow y$ in $Y$. It follows that
$\Vert S_{y}-S_{y_{n}}\Vert\rightarrow0$, so $S_{y}$ is compact for every
$y\in X$.
The subspace $I=0{\oplus}X$ is a closed ideal of $B$ and it follows easily
from topological transitivity of $\pi$ that $I$ is a minimal closed ideal.
Moreover, each non-zero closed ideal $J$ of $B$ contains $I$. Indeed, $J$
cannot be a subspace of $I$. Hence there is $a{\oplus}x\in J$ with $a\neq0$.
But then
\[
0{\oplus}{\pi}(b){\pi}(a)y=(b{\oplus}0)(a{\oplus}x)(0{\oplus}y)\in J
\]
for any $b\in A$, $y\in X$, whence $I\subset J$.
We see that $B$ has no total chains of closed ideals and has an
infinite-dimensional minimal closed ideal.
\end{proof}

In the remaining part of the section we obtain some \textquotedblleft
affirmative\textquotedblright$\;$ results.

\begin{lemma}
\label{gap} If $A$ is a radical compact Banach algebra and $J\subset I$ is a
gap of closed ideals of $A$ then $AI\subset J$ or $IA\subset J$.
\end{lemma}

\begin{proof}
Assume the contrary. Then $\overline{AI}=\overline{IA}=I$ (otherwise we obtain
an intermediate ideal) whence $\overline{AIA}=I$. One may assume that
$\dim(I/J)>1$ because otherwise the statement is trivial. Let $\pi$ be the
natural representation of $\mathrm{Mul}(A)$ on the space $X=I/J$. It is
topologically irreducible because if $Y$ is an invariant closed subspace of
$\pi$ then $\{x\in I:x/J\in Y\}$ is a closed ideal between $J$ and $I$. By
Corollary \ref{mul2}, the algebra $\mathrm{Mul}_{2}(A)$ is contained in the
radical of $\mathrm{Mul}(A)$. If $\pi(\mathrm{Mul}_{2}(A))$ contains a
non-zero compact operator then it has a non-trivial invariant closed subspace
by the Lomonosov Theorem \cite{Lom}. As $\pi(\mathrm{Mul}_{2}(A))$ is an ideal
of $\pi(\mathrm{Mul}(A))$, this implies that $\pi(\mathrm{Mul}(A))$ has a
non-trivial invariant closed subspace, a contradiction.
As we saw in the proof of Theorem \ref{mul0.2}, the operator $L_{a^{2}}R_{b}$
is a compact operator in $\mathrm{Mul}_{2}(A)$ for every $a,b\in A$.
Therefore, $\pi(L_{a^{2}}R_{b})$ is also a compact operator. By the above,
\[
\pi(L_{a^{2}}R_{b})=0.
\]
In other words, $a^{2}Ib\subset J$. Since $\overline{IA}=I$, we get that
$a^{2}I\subset J$. Thus
\[
\pi(L_{a})^{2}=0
\]
for all $a\in A$, whence $A^{3}I\subset J$ by the Nagata-Higman theorem. Since
$\overline{AI}=I$, we obtain a contradiction.
\end{proof}

\begin{theorem}
\label{infch} If $A$ is an infinite-dimensional compact radical Banach algebra
then any chain of closed ideals of $A$ extends to an infinite chain of closed
ideals of $A$.
\end{theorem}

\begin{proof}
Suppose, to the contrary, that there is a maximal chain
\[
0=J_{0}\subset J_{1}\subset...\subset J_{n}=A
\]
of closed ideals of $A$. Then each pair $(J_{k-1},J_{k})$ is a gap. It follows
from Lemma \ref{gap} that $AJ_{k}A\subset J_{k-1}$ for every $k>0$. Hence%
\[
A^{2n+1}=0.
\]
It follows that $\mathrm{Mul}(A)$ is also a nilpotent algebra. Hence it has no
non-trivial topologically irreducible representations. But for any gap
$(J_{k-1},J_{k})$ its representation on $X_{k}=J_{k}/J_{k-1}$ is topologically
irreducible and at least one of $X_{k}$ must be infinite dimensional. We
obtained a contradiction.
\end{proof}

Recall that by $\mathcal{A}(X)$ we denote the operator norm closure of the
ideal $\mathcal{F}(X)$ of all finite rank operators on $X$. Our aim is to show
that if ${\mathcal{A}(X)\neq\mathcal{K}}(X)$ then between ${\mathcal{A}(X)}$
and ${\mathcal{K}}(X)$ there are intermediate closed ideals.

\begin{corollary}
\label{k(x)}

\begin{itemize}
\item[$\mathrm{(i)}$] If $\dim({\mathcal{K}}(X)/\mathcal{A}(X))=n$ (where $n$
is a finite number or $\infty$) then ${\mathcal{K}}(X)$ has a chain of $n$
different closed ideals, containing ${\mathcal{F}}(X)$.

\item[$\mathrm{(ii)}$] Let $M$ and $N$ be closed ideals of ${\mathcal{B}}(X)$
with ${\mathcal{A}}(X)\subset N\subsetneqq M\subset{\mathcal{K}}(X)$. If
$M^{2}$ is not contained in $N$ then ${\mathcal{B}}(X)$ has a closed ideal
between $N$ and $M$.

In particular, if $({\mathcal{K}}(X)/\mathcal{A}(X))^{2}\neq0$ then there is a
closed ideal of ${\mathcal{B}}(X)$ between $\mathcal{A}(X)$ and ${\mathcal{K}%
}(X)$.

\item[$\mathrm{(iii)}$] If the algebra $\mathcal{K}(X)/\mathcal{A}(X)$ is not
nilpotent then every maximal chain of closed ideals of ${\mathcal{B}}(X)$
between $\mathcal{A}(X)$ and ${\mathcal{K}}(X)$ is infinite.
\end{itemize}
\end{corollary}

\begin{proof}
(i) The algebra ${\mathcal{Q}}(X)={\mathcal{K}}(X)/\mathcal{A}(X)$ is radical
by Corollary \ref{quot}. If its dimension $n$ is finite then clearly it has a
chain of $n$ ideals (since the nilpotent algebra $\mathrm{Mul}({\mathcal{Q}%
}(X))$ is triangularizable). In any case it is bicompact, so if $n=\infty$
then it has an infinite chain of ideals by Theorem \ref{infch}. The preimages
of these ideals in ${\mathcal{K}}(X)$ form a chain of ideals of ${\mathcal{K}%
}(X)$ containing $\overline{{\mathcal{F}}(X)}$. This proves (i).
(ii) Assume, to the contrary, that there are no closed ideals between $N$ and
$M$. As $\overline{M^{2}+N}$ is a closed ideal of ${\mathcal{B}}(X)$ strictly
containing $N$, then
\[
M=\overline{M^{2}+N}.
\]
As $\{T\in M:TM\subset N\}$ is a closed ideal of ${\mathcal{B}}(X)$ strictly
contained in $M$, then
\[
N=\{T\in M:TM\subset N\}
\]
and, similarly,
\[
N=\{T\in M:MT\subset N\}.
\]
By (i), there is a closed ideal $I$ of $\mathcal{K}(X)$ intermediate between
$M$ and $N$. Set $J=\overline{MIM+N}$. Then
\[
N\subset J\subset I\subsetneqq M.
\]
If $J=N$ then $MIM\subset N$ whence, by above, $IM\subset N$ and therefore
$I\subset N$, a contradiction. Thus
\[
N\subsetneqq J\subsetneqq M.
\]
As
\[
\mathcal{B}(X)J\mathcal{B}(X)\subset\overline{\mathcal{B}(X)MIM\mathcal{B}%
(X)+N}\subset\overline{MIM+N}=J,
\]
we obtained that $J$ is an intermediate closed ideal of $\mathcal{B}(X)$
between $N$ and $M$. Part (ii) is proved.
(iii) Assuming that $\mathcal{Q}(X)$ is not nilpotent, choose a maximal chain
$\left(  I_{\alpha}\right)  $ of closed ideals of $\mathcal{B}\left(
X\right)  $ between $\mathcal{A}(X)$ and $\mathcal{K}(X)$. If it is finite,
namely%
\[
\mathcal{A}(X)=I_{0}\subsetneqq I_{1}\subsetneqq I_{2}\subsetneqq
\cdots\subsetneqq I_{n}=\mathcal{K}(X),
\]
then
\[
I_{k}^{2}\subset I_{k-1}%
\]
for every $k>0$ by (ii). Then $\mathcal{Q}(X)^{2^{n}}=0$, a contradiction.
\end{proof}

\begin{example}
\label{willis} \emph{ To construct an example of a Banach space $X$ for which
the algebra $\mathcal{K}(X)/\mathcal{A}(X)$ is not nilpotent, one can use a
remarkable result of Willis \cite{Wil}}.

\emph{ Recall that $X$ is said to have }approximation property\emph{ (AP)
(respectively, }compact approximation property\emph{ (CAP)) if for each
compact set }$M\subset X$\emph{ and each }$\varepsilon>0$\emph{, there is a
finite rank (respectively, compact) operator }$S=S(M,\varepsilon)$\emph{ with
}$\Vert Sx-x\Vert<\varepsilon$\emph{ for all }$x\in M$\emph{. If }$S$\emph{
always can be chosen in such a way that }$\Vert S\Vert\leq C$\emph{ for some
fixed }$C>0$\emph{ then one says that }$X$\emph{ has }bounded approximation
property\emph{ (BAP) (respectively, }bounded compact approximation
property\emph{ (BCAP)).}

\emph{ It was proved in \cite{Wil} that there exists a space $X$ which has not
AP but has BCAP. Let us show that this is a space we need. Indeed, it follows
from BCAP that the algebra $\mathcal{K}(X)$ has a bounded approximate
identity: to construct it one have to take for the index set the set of all
pairs }$\lambda=(M,\varepsilon)$\emph{ where }$M$\emph{ is a compact subset of
}$X$\emph{ and }$\varepsilon>0$\emph{, and denote by }$S_{\lambda}$\emph{ an
operator }$S=S(M,\varepsilon)$\emph{ from the definition of BCAP. In
particular, $\mathcal{K}(X)^{n}$ is dense in $\mathcal{K}$}$(X)$\emph{ for
each $n$. Hence if $\mathcal{K}(X)/\mathcal{A}$}$(X)$\emph{ is nilpotent then
$\mathcal{A}(X)=\mathcal{K}$}$(X)$\emph{. Therefore $\mathcal{A}(X)$ has a
bounded approximate identity }$e_{\lambda}$\emph{ and one can assume that
}$e_{\lambda}\in$\emph{$\mathcal{F}$}$(X)$\emph{ for each }$\lambda$\emph{.
Let us show that this implies AP (in contradiction with the choice of $X$)}.

\emph{ It is easy to see (considering rank one operators) that }$e_{\lambda
}x\rightarrow x$\emph{ for each }$x\in X$\emph{. Now if a compact subset }%
$M$\emph{ of }$X$\emph{ and }$\varepsilon>0$\emph{ are given, let us choose a
finite $\varepsilon$-net }$M_{0}$\emph{ in }$M$\emph{ and an index }$\mu
$\emph{ with }$\Vert e_{\mu}x-x\Vert<\varepsilon$\emph{ for all }$x\in M_{0}%
$\emph{. Then }$\Vert e_{\mu}x-x\Vert<t\varepsilon$\emph{ for all }$x\in
M$\emph{, where }$t=2+\sup_{\lambda}\Vert e_{\lambda}\Vert$\emph{.}
\end{example}

Let us denote by $\mathcal{A}$ and $\mathcal{K}$ the closed operator ideals of
approximable and compact operators, respectively.

\begin{corollary}
There is an infinite chain of closed operator ideals intermediate between
$\mathcal{A}$ and $\mathcal{K}$.
\end{corollary}

\begin{proof}
Let $Z$ be a Banach space with non-nilpotent $\mathcal{K}(Z)/\mathcal{A}(Z)$
(see Example \ref{willis}). By Corollary \ref{k(x)}, between $\mathcal{K}(Z)$
and $\mathcal{A}(Z)$ there is an infinite chain $\left\{  I_{\alpha}\right\}
$ of closed ideals of $\mathcal{B}(Z)$. For each pair $(X,Y)$ of Banach
spaces, we denote by $\mathcal{U}_{\alpha}(X,Y)$ the set of all operators
$T\in\mathcal{K}(X,Y)$ such that $ATB\in I_{\alpha}$ for all $A\in
\mathcal{B}(Y,Z)$ and $B\in\mathcal{B}(Z,X)$. It is easy to check that each
$\mathcal{U}_{\alpha}$ is a closed operator ideal between $\mathcal{A}$ and
$\mathcal{K}$, that all $\mathcal{U}_{\alpha}$ are different and that they
form a chain.
\end{proof}

\begin{theorem}
\label{apprgap} If $A$ is a radical approximable Banach algebra then each
gap-quotient in the lattice of the closed ideals of $A$ is one-dimensional.
\end{theorem}

\begin{proof}
Let $J\subset I$ be a gap of closed ideals of $A$. Then either $AI\subset J$
or $IA\subset J$ by Lemma \ref{gap}. Suppose that $IA\subset J$. Denote by
$\pi$ the natural representation of $\mathrm{Mul}(A)$ on $I/J$. Then we have
that ${\pi}(R_{A})=0$, whence ${\pi}(\mathcal{E\!\ell}(A))={\pi}(L_{A})$ and
${\pi}(\mathrm{Mul}(A))\subset\overline{{\pi}(L_{A})}$.
The map $a\longmapsto{\pi}(L_{a})$ is a topologically irreducible
representation of $A$ on $I/J$. Since such a representation of $A$ must be
trivial by Proposition \ref{TI}, it acts on a one-dimensional space.
\end{proof}

\begin{corollary}
Every radical hypofinite Banach algebra has a total chain of closed ideals,
and each minimal closed ideal in such an algebra is one-dimensional.
\end{corollary}

Let us call a subspace $I$ of a Banach algebra $A$ \textit{a quasiideal} if
$AIA\subset I$. Clearly each ideal is a quasiideal. The converse is true if
$A$ has a (non-necessarily bounded) approximate identity.

\begin{theorem}
\label{quas} Any bicompact radical Banach algebra has a total chain of closed quasiideals.
\end{theorem}

\begin{proof}
Closed quasiideals are invariant subspaces of the radical algebra
$\mathrm{Mul}_{2}(A)$ of compact operators. As such algebras are
triangularizable, our statement follows.
\end{proof}

\section{Spectral subspaces of elementary and multiplication
operators\label{s6}}

In this section we consider invariant subspaces of semicompact multiplication
operators, on which the operators are surjective (in particular, eigenspaces
with non-zero eigenvalues or spectral subspaces corresponding to clopen
subsets of spectra non-containing $0$). Our approach will be based (apart of
the tensor radical technique) on a study of operators acting in ordered pairs
of Banach spaces. In Section \ref{s6dob} we improve the results for
semicompact elementary operators by another technique to show that such
invariant subspaces are contained in the component of every quasi-Banach
operator ideal.

\subsection{Operators on an ordered pair of Banach spaces}

Let $\mathcal{X},\mathcal{Y}$ be Banach spaces, and $\mathcal{Y}%
\subset\mathcal{X}$. Suppose that
\begin{equation}
\left\Vert y\right\Vert _{\mathcal{X}}\leq\left\Vert y\right\Vert
_{\mathcal{Y}} \label{n1}%
\end{equation}
for all $y\in\mathcal{Y}$. We refer to such a subspace $\mathcal{Y}$ as a
\textit{Banach subspace} of $\mathcal{X}$ and call $\left(  \mathcal{Y}%
,\mathcal{X}\right)  $ an \textit{ordered pair of Banach spaces}.

Denote by $\mathcal{B}\left(  \mathcal{X}||\mathcal{Y}\right)  $ the space of
all operators $T\in\mathcal{B}(\mathcal{X})$ such that $T\mathcal{Y}%
\subset\mathcal{Y}$. It is non-zero, for instance the identity operator
$1_{\mathcal{X}}$ in $\mathcal{B}(\mathcal{X})$ lies in $\mathcal{B}\left(
\mathcal{X}||\mathcal{Y}\right)  $.

\begin{theorem}
\label{pair1}Let $\mathcal{Y}$ be a Banach subspace of a Banach space
$\mathcal{X}$. Then $T|_{\mathcal{Y}}\in\mathcal{B}(\mathcal{Y})$ for any
$T\in\mathcal{B}\left(  \mathcal{X}||\mathcal{Y}\right)  $, and $\mathcal{B}%
\left(  \mathcal{X}||\mathcal{Y}\right)  $ is a unital Banach subalgebra of
$\mathcal{B}(\mathcal{X})$ with respect to the norm
\[
\Vert T\Vert_{\mathcal{B}\left(  \mathcal{X}||\mathcal{Y}\right)  }%
=\max\left\{  \Vert T\Vert_{\mathcal{B}(\mathcal{X})},\Vert T|_{\mathcal{Y}%
}\Vert_{\mathcal{B}(\mathcal{Y})}\right\}
\]
for any $T\in\mathcal{B}\left(  \mathcal{X}||\mathcal{Y}\right)  $.
\end{theorem}

\begin{proof}
We first show that $T|_{\mathcal{Y}}$ is a bounded operator on $\mathcal{Y}$.
To apply the Closed Graph Theorem, it is sufficient to show that the
conditions $\left\Vert y_{n}\right\Vert _{\mathcal{Y}}\rightarrow0$ and
$\left\Vert Ty_{n}-u\right\Vert _{\mathcal{Y}}\rightarrow0$ as $n\rightarrow
\infty$ for $\left\{  y_{n}\right\}  \subset\mathcal{Y}$ imply $u=0$. If these
conditions hold, then also $\left\Vert y\right\Vert _{\mathcal{X}}%
\rightarrow0$ and $\left\Vert Ty_{n}-u\right\Vert _{\mathcal{X}}\rightarrow0$
as $n\rightarrow\infty$. As $T$ is bounded on $\mathcal{X}$, then $u=0$.
As a consequence, $\Vert\cdot\Vert_{\mathcal{B}\left(  \mathcal{X}%
||\mathcal{Y}\right)  }$ is a norm on $\mathcal{B}\left(  \mathcal{X}%
||\mathcal{Y}\right)  $. This norm is clearly a unital algebra norm that
majorizes $\Vert\cdot\Vert_{\mathcal{B}\left(  \mathcal{X}\right)  }$ on
$\mathcal{B}\left(  \mathcal{X}||\mathcal{Y}\right)  $. To finish the proof,
it remains to show that $\Vert\cdot\Vert_{\mathcal{B}\left(  \mathcal{X}%
||\mathcal{Y}\right)  }$ is complete.
Let $\left\{  T_{n}\right\}  $ be a fundamental sequence in $\mathcal{B}%
\left(  \mathcal{X}||\mathcal{Y}\right)  $. Then there are $T\in
\mathcal{B}(\mathcal{X})$ and $S\in\mathcal{B}(\mathcal{Y})$ such that
$\left\Vert T_{n}-T\right\Vert _{\mathcal{B}(\mathcal{X})}\rightarrow0$ and
$\left\Vert T_{n}|_{\mathcal{Y}}-S\right\Vert _{\mathcal{B}(\mathcal{Y}%
)}\rightarrow0$ as $n\rightarrow\infty$. Then $\left\Vert T_{n}y-Ty\right\Vert
_{\mathcal{X}}\rightarrow0$ and $\left\Vert T_{n}y-Sy\right\Vert
_{\mathcal{X}}\leq\left\Vert T_{n}y-Sy\right\Vert _{\mathcal{Y}}\rightarrow0$
as $n\rightarrow\infty$, for every $y\in\mathcal{Y}$. This shows that
$\mathcal{Y}$ is invariant for $T$ and $T|_{\mathcal{Y}}=S$.
\end{proof}

As usual, $\Vert S\Vert_{\mathcal{B}(\mathcal{X},\mathcal{Y})}$ denotes the
operator norm of an operator $S$ in $\mathcal{B}(\mathcal{X},\mathcal{Y})$.

\begin{proposition}
Let $\mathcal{Y}$ be a Banach subspace of a Banach space $\mathcal{X}$. Then
$\mathcal{B}(\mathcal{X},\mathcal{Y})$ is a Banach algebra with respect to the
usual norm $\Vert\cdot\Vert_{\mathcal{B}(\mathcal{X},\mathcal{Y})}$.
\end{proposition}

\begin{proof}
It is clear that $\mathcal{B}(\mathcal{X},\mathcal{Y})$ is a Banach space. If
$S,T\in\mathcal{B}(\mathcal{X},\mathcal{Y})$ then
\[
\Vert STx\Vert_{\mathcal{Y}}\leq\Vert S\Vert_{\mathcal{B}(\mathcal{X}%
,\mathcal{Y})}\Vert Tx\Vert_{\mathcal{X}}\leq\Vert S\Vert_{\mathcal{B}%
(\mathcal{X},\mathcal{Y})}\Vert Tx\Vert_{\mathcal{Y}}\leq\Vert S\Vert
_{\mathcal{B}(\mathcal{X},\mathcal{Y})}\Vert T\Vert_{\mathcal{B}%
(\mathcal{X},\mathcal{Y})}\Vert x\Vert_{\mathcal{X}}%
\]
for every $x\in\mathcal{X}$, whence $\mathcal{B}(\mathcal{X},\mathcal{Y})$ is
a Banach algebra.
\end{proof}

\begin{proposition}
\label{idBinom} Let $\mathcal{Y}$ be a Banach subspace of a Banach space
$\mathcal{X}$. Then $\mathcal{B}(\mathcal{X},\mathcal{Y})$ is a Banach ideal
of $\mathcal{B}\left(  \mathcal{X}||\mathcal{Y}\right)  $ with respect to
$\left\Vert \cdot\right\Vert _{\mathcal{B}(\mathcal{X},\mathcal{Y})}$ which is
a flexible norm.
\end{proposition}

\begin{proof}
The inclusion $\mathcal{B}(\mathcal{X},\mathcal{Y})\subset\mathcal{B}\left(
\mathcal{X}||\mathcal{Y}\right)  $ follows by Theorem \ref{pair1}. Let
$S\in\mathcal{B}(\mathcal{X},\mathcal{Y})$ and $P,T\in\mathcal{B}\left(
\mathcal{X}||\mathcal{Y}\right)  $. It is clear that $PS,ST\in\mathcal{B}%
(\mathcal{X},\mathcal{Y})\subset\mathcal{B}\left(  \mathcal{X}||\mathcal{Y}%
\right)  $. So $\mathcal{B}(\mathcal{X},\mathcal{Y})$ is an ideal of
$\mathcal{B}\left(  \mathcal{X}||\mathcal{Y}\right)  $. As $\Vert
Sx\Vert_{\mathcal{Y}}\leq\Vert S\Vert_{\mathcal{B}(\mathcal{X},\mathcal{Y}%
)}\Vert x\Vert_{\mathcal{X}}$ for every $x\in\mathcal{X}$, one obtains from
(\ref{n1}) that
\begin{align*}
\Vert Sx\Vert_{\mathcal{X}}  &  \leq\Vert Sx\Vert_{\mathcal{Y}}\leq\Vert
S\Vert_{\mathcal{B}(\mathcal{X},\mathcal{Y})}\Vert x\Vert_{\mathcal{X}}\text{
and }\\
\Vert Sy\Vert_{\mathcal{Y}}  &  \leq\Vert S\Vert_{\mathcal{B}(\mathcal{X}%
,\mathcal{Y})}\Vert y\Vert_{\mathcal{X}}\leq\Vert S\Vert_{\mathcal{B}%
(\mathcal{X},\mathcal{Y})}\Vert y\Vert_{\mathcal{Y}}%
\end{align*}
for every $x\in\mathcal{X}$ and $y\in\mathcal{Y}$. Therefore
\[
\Vert S\Vert_{\mathcal{B}(\mathcal{X})}\leq\max\left\{  \Vert S\Vert
_{\mathcal{B}(\mathcal{X})},\left\Vert S|_{\mathcal{Y}}\right\Vert
_{\mathcal{B}\left(  \mathcal{Y}\right)  }\right\}  =\Vert S\Vert
_{\mathcal{B}\left(  \mathcal{X}||\mathcal{Y}\right)  }\leq\Vert
S\Vert_{\mathcal{B}(\mathcal{X},\mathcal{Y})}.
\]
It follows that
\begin{align*}
\left\Vert PSTx\right\Vert _{\mathcal{Y}}  &  \leq\left\Vert P|_{\mathcal{Y}%
}\right\Vert _{\mathcal{B}(\mathcal{Y})}\left\Vert STx\right\Vert
_{\mathcal{Y}}\leq\left\Vert P|_{\mathcal{Y}}\right\Vert _{\mathcal{B}%
(\mathcal{Y})}\left\Vert S\right\Vert _{\mathcal{B}(\mathcal{X},\mathcal{Y}%
)}\left\Vert Tx\right\Vert _{\mathcal{X}}\\
&  \leq\left\Vert P|_{\mathcal{Y}}\right\Vert _{\mathcal{B}(\mathcal{Y}%
)}\left\Vert S\right\Vert _{\mathcal{B}(\mathcal{X},\mathcal{Y})}\left\Vert
T\right\Vert _{\mathcal{B}(\mathcal{X})}\left\Vert x\right\Vert _{\mathcal{X}%
}\\
&  \leq\left\Vert P\right\Vert _{\mathcal{B}\left(  \mathcal{X}||\mathcal{Y}%
\right)  }\left\Vert S\right\Vert _{\mathcal{B}(\mathcal{X},\mathcal{Y}%
)}\left\Vert T\right\Vert _{\mathcal{B}\left(  \mathcal{X}||\mathcal{Y}%
\right)  }\left\Vert x\right\Vert _{\mathcal{X}}.
\end{align*}
for every $x\in\mathcal{X}$, whence
\[
\left\Vert PST\right\Vert _{\mathcal{B}(\mathcal{X},\mathcal{Y})}%
\leq\left\Vert P\right\Vert _{\mathcal{B}\left(  \mathcal{X}||\mathcal{Y}%
\right)  }\left\Vert S\right\Vert _{\mathcal{B}(\mathcal{X},\mathcal{Y}%
)}\left\Vert T\right\Vert _{\mathcal{B}\left(  \mathcal{X}||\mathcal{Y}%
\right)  }.
\]
\end{proof}

\begin{theorem}
\label{pair2} Let $\mathcal{Y}$ be a Banach subspace of a Banach space
$\mathcal{X}$. Then every operator $T\in\mathcal{B}\left(  \mathcal{X}\right)
$ such that $T\mathcal{X}\subset\mathcal{Y}$ belongs to $\mathcal{B}\left(
\mathcal{X},\mathcal{Y}\right)  $.
\end{theorem}

\begin{proof}
By the Closed Graph Theorem, it suffices to show that $u=0$ if $\left\Vert
x_{n}\right\Vert _{\mathcal{X}}\rightarrow0$ and $\left\Vert Tx_{n}%
-u\right\Vert _{\mathcal{Y}}\rightarrow0$ as $n\rightarrow\infty$. Indeed, the
last implies that $\left\Vert Tx_{n}-u\right\Vert _{\mathcal{X}}\rightarrow0$
as $n\rightarrow\infty$. As $T$ is bounded on $\mathcal{X}$, then $u=0$.
\end{proof}

\subsection{Invariant subspaces for operators on an ordered pair of Banach
spaces}

We consider those invariant subspaces of an operator on an ordered pair
$(\mathcal{Y},\mathcal{X})$ of Banach spaces on which the operator is
surjective. Clearly such a subspace is contained in $\mathcal{Y}$ if the
operator belongs to $\mathcal{B}\left(  \mathcal{X},\mathcal{Y}\right)  $. We
show that the same is true if the operator belongs to $\mathcal{B}\left(
\mathcal{X}||\mathcal{Y}\right)  $ and is quasinilpotent modulo $\mathcal{B}%
\left(  \mathcal{X},\mathcal{Y}\right)  $.

\begin{theorem}
\label{spectrclos}Let $\mathcal{X}$ be a Banach space, $\mathcal{Y}$ a Banach
subspace of $\mathcal{X}$, and let $T\in\mathcal{B}\left(  \mathcal{X}%
||\mathcal{Y}\right)  $. Assume that
\begin{equation}
T\in Q_{\mathcal{B}(\mathcal{X},\mathcal{Y})}\left(  \mathcal{B}\left(
\mathcal{X}||\mathcal{Y}\right)  \right)  . \label{sc}%
\end{equation}
If $\mathcal{Z}$ is a closed subspace of $\mathcal{X}$ such that
$\mathcal{Z}=T\mathcal{Z}$, then $\mathcal{Z}\subset\mathcal{Y}$. Moreover,
the norms $\Vert\cdot\Vert_{\mathcal{X}}$ and $\Vert\cdot\Vert_{\mathcal{Y}}$
are equivalent on $\mathcal{Z}$, so $\mathcal{Z}$ is also closed in
$\mathcal{Y}$.
\end{theorem}

\begin{proof}
By the Open Mapping Theorem, there is $t>0$ such that for each $z\in
\mathcal{Z}$ there is $w\in\mathcal{Z}$ with $Tw=z$ and
\[
\Vert w\Vert_{\mathcal{X}}\leq t\Vert z\Vert_{\mathcal{X}}.
\]
Let $\varepsilon>0$ be such that $\varepsilon<t^{-1}$. It follows from
Proposition \ref{idBinom} that $\mathcal{B}(\mathcal{X},\mathcal{Y})$ is an
ideal of $\mathcal{B}\left(  \mathcal{X}||\mathcal{Y}\right)  $. By our
assumption and Proposition \ref{predv1}(iii), there is $m\in\mathbb{N}$ such
that
\[
\mathrm{dist}_{\mathcal{B}\left(  \mathcal{X}||\mathcal{Y}\right)  }\left(
T^{m},\mathcal{B}(\mathcal{X},\mathcal{Y})\right)  <\varepsilon^{m}.
\]
Therefore there is an operator $S\in\mathcal{B}(\mathcal{X},\mathcal{Y})$ and
an operator $P\in\mathcal{B}\left(  \mathcal{X}||\mathcal{Y}\right)  $ such
that
\[
T^{m}=S+P\text{ with }\max\left\{  \left\Vert P\right\Vert _{\mathcal{B}%
\left(  \mathcal{X}\right)  },\left\Vert P|_{\mathcal{Y}}\right\Vert
_{\mathcal{B}\left(  \mathcal{Y}\right)  }\right\}  <\varepsilon^{m}.
\]
Then
\begin{equation}
\left\Vert Py\right\Vert _{\mathcal{Y}}\leq\varepsilon^{m}\left\Vert
y\right\Vert _{\mathcal{Y}} \label{femy}%
\end{equation}
for every $y\in\mathcal{Y}$.
It follows from the definition of $t$ that for every $z\in\mathcal{Z}$, there
is $z^{\circ}\in\mathcal{Z}$ with $T^{m}z^{\circ}=z$ and
\[
\Vert z^{\circ}\Vert_{\mathcal{X}}<t^{m}\Vert z\Vert_{\mathcal{X}}.
\]
Let $z_{0}:=z\in\mathcal{Z}$ be arbitrary. Set $z_{1}=z_{0}^{\circ}$,
$z_{2}=z_{1}^{\circ}$, and so on:
\[
z_{k+1}=z_{k}^{\circ}%
\]
for every integer $k>0$. Thus $z_{k}=T^{m}z_{k+1}=Sz_{k+1}+Pz_{k+1}$ with
\begin{equation}
\Vert z_{k}\Vert_{\mathcal{X}}\leq t^{mk}\Vert z\Vert_{\mathcal{X}}.
\label{ryj}%
\end{equation}
Rewriting this in the form
\[
z_{k}-Pz_{k+1}=Sz_{k+1},
\]
multiplying both sides of the equation by $P^{k}$ and summing obtained
equalities for $k=0,1,2,\ldots$, one formally obtains that
\begin{equation}
z=Sz_{1}+PSz_{2}+P^{2}Sz_{3}+... \label{ser}%
\end{equation}
Since $S\mathcal{X}\subset\mathcal{Y}$, all elements $Sz_{k}$ and
$P^{k-1}Sz_{k}$ belong to $\mathcal{Y}$. As $\Vert z_{k}\Vert_{\mathcal{X}%
}<t^{mk}\Vert z\Vert_{\mathcal{X}}$, one obtains that
\begin{equation}
\left\Vert Sz_{k}\right\Vert _{\mathcal{Y}}<\Vert S\Vert_{\mathcal{B}%
(\mathcal{X},\mathcal{Y})}\left\Vert z_{k}\right\Vert _{\mathcal{X}}\leq
t^{mk}\Vert S\Vert_{\mathcal{B}(\mathcal{X},\mathcal{Y})}\left\Vert
z\right\Vert _{\mathcal{X}}. \label{ser1}%
\end{equation}
It follows from (\ref{femy}) and (\ref{ser1}) that
\[
\left\Vert P^{k-1}Sz_{k}\right\Vert _{\mathcal{Y}}\leq\varepsilon
^{m(k-1)}\left\Vert Sz_{k}\right\Vert _{\mathcal{Y}}\leq\left(  \varepsilon
t\right)  ^{mk}\left(  \varepsilon^{-m}\Vert S\Vert_{\mathcal{B}%
(\mathcal{X},\mathcal{Y})}\left\Vert z\right\Vert _{\mathcal{X}}\right)  .
\]
As $\varepsilon t<1$, we have that
\begin{equation}
\sum_{k=1}^{\infty}\left\Vert P^{k-1}Sz_{k}\right\Vert _{\mathcal{Y}}%
\leq\left(  \frac{t^{m}}{1-\varepsilon^{m}t^{m}}\Vert S\Vert_{\mathcal{B}%
(\mathcal{X},\mathcal{Y})}\right)  \left\Vert z\right\Vert _{\mathcal{X}%
}<\infty. \label{ser2}%
\end{equation}
Since $\Vert\cdot\Vert_{\mathcal{Y}}$ is a complete norm on $\mathcal{Y}$, it
follows from (\ref{ser2}) and (\ref{ser}) that $z\in\mathcal{Y}$ with the
estimation%
\[
\left\Vert z\right\Vert _{\mathcal{Y}}\leq\left(  \frac{t^{m}}{1-\varepsilon
^{m}t^{m}}\Vert S\Vert_{\mathcal{B}(\mathcal{X},\mathcal{Y})}\right)
\left\Vert z\right\Vert _{\mathcal{X}}.
\]
Therefore $\left\Vert \cdot\right\Vert _{\mathcal{X}}$ and $\left\Vert
\cdot\right\Vert _{\mathcal{Y}}$ are equivalent on $\mathcal{Z}$. As
$\mathcal{Z}$ is closed with respect to $\left\Vert \cdot\right\Vert
_{\mathcal{X}}$, it is closed with respect to $\left\Vert \cdot\right\Vert
_{\mathcal{Y}}$.
\end{proof}

\begin{corollary}
\label{eigen} Let $\mathcal{X},\mathcal{Y}$ and $T$ be as in Theorem
\emph{\ref{spectrclos}}. Then

\begin{itemize}
\item[$\mathrm{(i)}$] If $\lambda\neq0$ is an eigenvalue of $T$ then the
eigenspace $\{x\in\mathcal{X}:Tx=\lambda x\}$ is contained in $\mathcal{Y}$.

\item[$\mathrm{(ii)}$] If $\sigma_{0}$ is a clopen subset of the spectrum
$\sigma_{\mathcal{B}\left(  \mathcal{X}\right)  }(T)$ of $T$ and
$0\notin\sigma_{0}$ then the spectral subspace $E_{\sigma_{0}}(T)$ is
contained in $\mathcal{Y}$.
\end{itemize}
\end{corollary}

\begin{proof}
Indeed, these subspaces are closed in $\mathcal{X}$, invariant under $T$, and
the restriction of $T$ to everyone of them is invertible.
\end{proof}

\subsection{Semicompact multiplication operators}

In this section we apply Theorem \ref{spectrclos} to semicompact
multiplication operators considering their action on an ordered pair of spaces
of nuclear and, respectively, bounded operators.

\subsubsection{Multiplication operators on an ordered pair of operator ideals}

First we estimate the norms of multiplication operators on an ordered pair of
components of Banach operator ideals.

Let $V=\mathcal{V}\left(  X,Y\right)  $ and $U=\mathcal{U}\left(  X,Y\right)
$, where $\mathcal{V}$ and $\mathcal{U}$ are Banach operator ideals. We assume
that $V\subset U$ and that $V$ is a Banach subspace of $U$. As $V$ is an
invariant subspace for the algebra $\widehat{\mathcal{B}}_{\ast}\left(
U\right)  $ of all multiplication operators on $U$, then the algebra
$\mathcal{B}\left(  \mathcal{X}||\mathcal{Y}\right)  $ contains
$\widehat{\mathcal{B}}_{\ast}\left(  U\right)  $ by Theorem \ref{pair1}.

\begin{lemma}
\label{maj}Let $V=\mathcal{V}\left(  X,Y\right)  $ and $U=\mathcal{U}\left(
X,Y\right)  $ for Banach operator ideals $\mathcal{V}$ and $\mathcal{U}$, and
$V\subset U$. Then $\left\Vert T|_{V}\right\Vert _{\widehat{\mathcal{B}}%
_{\ast}\left(  V\right)  }\leq\left\Vert T\right\Vert _{\widehat{\mathcal{B}%
}_{\ast}\left(  U\right)  }$ for every $T\in\widehat{\mathcal{B}}_{\ast
}\left(  U\right)  $.
\end{lemma}

\begin{proof}
Let $W=\mathcal{B}\left(  Y\right)  \widehat{\otimes}\mathcal{B}\left(
X\right)  ^{\mathrm{op}}$ for brevity. Recall that the norms $\left\Vert
\cdot\right\Vert _{\widehat{\mathcal{B}}_{\ast}\left(  V\right)  }$ in
$\widehat{\mathcal{B}}_{\ast}\left(  V\right)  $ and $\left\Vert
\cdot\right\Vert _{\widehat{\mathcal{B}}_{\ast}\left(  U\right)  }$ in
$\widehat{\mathcal{B}}_{\ast}\left(  U\right)  $ are the quotient norms
inherited from $W/\ker\psi$ and $W/\ker\varphi$ respectively, where
$\psi:W\longrightarrow\mathcal{B}\left(  V\right)  $ and $\varphi
:W\longrightarrow\mathcal{B}\left(  U\right)  $ are bounded homomorphisms that
associate with every $a\otimes b\in W$ the operator $L_{a}R_{b}$.
Let $P=\varphi\left(  w\right)  $ and $S=\psi\left(  w\right)  $ for some
$w\in W$. If $\varphi\left(  w\right)  =0$ then $Px=0$ for every $x\in U$. In
particular, $Px=0$ for every $x\in V$ and%
\[
\psi\left(  w\right)  =S=P|_{V}=0.
\]
This shows that $\ker\varphi\subset\ker\psi$, and we are done.
\end{proof}

\begin{proposition}
\label{maj2}Let $V=\mathcal{V}\left(  X,Y\right)  $ and $U=\mathcal{U}\left(
X,Y\right)  $ for Banach operator ideals $\mathcal{V}$ and $\mathcal{U}$, and
$V\subset U$. Then $\left\Vert T\right\Vert _{\mathcal{B}\left(  U||V\right)
}\leq\left\Vert T\right\Vert _{\widehat{\mathcal{B}}_{\ast}\left(  U\right)
}$ for every $T\in\widehat{\mathcal{B}}_{\ast}\left(  U\right)  $.
\end{proposition}

\begin{proof}
Indeed, as $\left\Vert \cdot\right\Vert _{\widehat{\mathcal{B}}_{\ast}\left(
U\right)  }$ majorizes the norm $\left\Vert \cdot\right\Vert _{\mathcal{B}%
\left(  U\right)  }$ on $\widehat{\mathcal{B}}_{\ast}\left(  U\right)  $, we
obtain from Lemma \ref{maj} that
\[
\left\Vert T\right\Vert _{\mathcal{B}\left(  U||V\right)  }=\max\left\{
\left\Vert T\right\Vert _{\mathcal{B}\left(  U\right)  },\left\Vert
T|_{V}\right\Vert _{\mathcal{B}\left(  V\right)  }\right\}  \leq\left\Vert
T\right\Vert _{\widehat{\mathcal{B}}_{\ast}\left(  U\right)  }%
\]
for every $T\in\widehat{\mathcal{B}}_{\ast}\left(  U\right)  $.
\end{proof}

\subsubsection{Applications to semicompact multiplication operators}


Let $X,Y$ be arbitrary Banach spaces. Let $\mathcal{X}=\mathcal{B}\left(
X,Y\right)  $ and $\mathcal{Y}=\mathcal{N}\left(  X,Y\right)  $, the space of
all nuclear operators $X\longrightarrow Y$. It is clear that $\mathcal{Y}$ is
a Banach subspace of $\mathcal{X}$. Also, $\mathcal{X}$ and $\mathcal{Y}$ are
Banach operator bimodules over the algebras $\mathcal{B}(X)$ and
$\mathcal{B}(Y)$, so the algebra $\mathcal{B}\left(  \mathcal{X}%
||\mathcal{Y}\right)  $ contains the algebra $\widehat{\mathcal{B}}_{\ast
}\left(  \mathcal{X}\right)  $ of all multiplication operators
\[
T=\sum_{i}L_{a_{i}}R_{b_{i}}\text{ with }\sum_{i}\Vert a_{i}\Vert\Vert
b_{i}\Vert<\infty
\]
where $a_{i}\in\mathcal{B}\left(  Y\right)  $, $b_{i}\in\mathcal{B}\left(
X\right)  .$ Recall that a multiplication operator $T$ is called
\textit{semicompact} if it can be written in the form
\begin{equation}
T=\sum_{i}L_{a_{i}}R_{t_{i}}+\sum_{j}L_{s_{j}}R_{b_{j}}, \label{semicompop}%
\end{equation}
where all $a_{i}$ and $b_{j}$ are compact operators, and
\[
\sum_{i}\Vert a_{i}\Vert\Vert t_{i}\Vert+\sum_{j}\Vert s_{j}\Vert\Vert
b_{j}\Vert<\infty.
\]
Also, an elementary operator $T$ is called\textit{ semifinite } if it can be
written in the form (\ref{semicompop}) with $a_{i}$ and $b_{j}$ of finite
rank. The algebras of all semicompact multiplication operators on
$\mathcal{X}$ and all semifinite elementary operators on $\mathcal{X}$ are
denoted by $\widehat{\mathcal{K}}_{\frac{1}{2}}(\mathcal{X})$ and
$\mathcal{F}_{\frac{1}{2}}(\mathcal{X})$, respectively.

In particular, from above we have that $\widehat{\mathcal{K}}_{\frac{1}{2}%
}(\mathcal{X})\subset\mathcal{B}\left(  \mathcal{X}||\mathcal{Y}\right)  .$

\begin{theorem}
\label{cond}Let $\mathcal{X}=\mathcal{B}\left(  X,Y\right)  $ and
$\mathcal{Y}=\mathcal{N}\left(  X,Y\right)  $. Then
\begin{equation}
\widehat{\mathcal{K}}_{\frac{1}{2}}(\mathcal{X})\subset Q_{\mathcal{B}\left(
\mathcal{X},\mathcal{Y}\right)  }\left(  \mathcal{B}\left(  \mathcal{X}%
||\mathcal{Y}\right)  \right)  . \label{fq2}%
\end{equation}

\end{theorem}

\begin{proof}
By Corollary \ref{QQ}, we have that
\begin{equation}
\widehat{\mathcal{K}}_{\frac{1}{2}}(\mathcal{X})\subset Q_{\mathcal{F}%
_{\frac{1}{2}}(\mathcal{X})}\left(  \widehat{\mathcal{B}}_{\ast}\left(
\mathcal{X}\right)  \right)  . \label{fq3}%
\end{equation}
As $\left\Vert \cdot\right\Vert _{\mathcal{B}\left(  \mathcal{X}%
||\mathcal{Y}\right)  }\leq\left\Vert \cdot\right\Vert _{\widehat{\mathcal{B}%
}_{\ast}\left(  \mathcal{X}\right)  }$ on $\widehat{\mathcal{B}}_{\ast}\left(
\mathcal{X}\right)  $ by Proposition \ref{maj2}, it follows that
\begin{equation}
Q_{\mathcal{F}_{\frac{1}{2}}(\mathcal{X})}\left(  \widehat{\mathcal{B}}_{\ast
}\left(  \mathcal{X}\right)  \right)  \subset Q_{\mathcal{F}_{\frac{1}{2}%
}(\mathcal{X})}\left(  \mathcal{B}\left(  \mathcal{X}||\mathcal{Y}\right)
\right)  . \label{fq5}%
\end{equation}
Since $S\mathcal{X}\subset\mathcal{Y}$ for every $S\in\mathcal{F}_{\frac{1}%
{2}}(\mathcal{X})$, we have that%
\[
\mathcal{F}_{\frac{1}{2}}(\mathcal{X})\subset\mathcal{B}(\mathcal{X}%
,\mathcal{Y})
\]
by Theorem \ref{pair2}. Therefore, we obtain that
\begin{equation}
Q_{\mathcal{F}_{\frac{1}{2}}(\mathcal{X})}\left(  \mathcal{B}\left(
\mathcal{X}||\mathcal{Y}\right)  \right)  \subset Q_{\mathcal{B}\left(
\mathcal{X},\mathcal{Y}\right)  }\left(  \mathcal{B}\left(  \mathcal{X}%
||\mathcal{Y}\right)  \right)  , \label{fq6}%
\end{equation}
and (\ref{fq2}) follows from (\ref{fq3}), (\ref{fq5}) and (\ref{fq6}).
\end{proof}

Now we are able to apply Theorem \ref{spectrclos} to obtain the following

\begin{theorem}
\label{space} Let $T$ be a semicompact multiplication operator on
$\mathcal{B}(X,Y)$. Suppose that a closed subspace $\mathfrak{\mathcal{Z}}$ of
$\mathcal{B}(X,Y)$ is invariant for $T$ and that $T$ is surjective on
$\mathfrak{\mathcal{Z}}$. Then $\mathfrak{\mathcal{Z}}$ consists of nuclear
operators, and the usual operator norm is equivalent to the nuclear norm on
$\mathfrak{\mathcal{Z}}$.

In particular, all eigenspaces of $T$ corresponding to non-zero eigenvalues
and all spectral subspaces of $T$ corresponding to clopen subsets of
$\sigma\left(  T\right)  $ non-containing $0$ consist of nuclear operators.
\end{theorem}

The following result holds for integral semicompact operators by Proposition
\ref{intsemi}, Theorems \ref{intSemCom} and \ref{space}.

\begin{theorem}
\label{spaceIntegral} Let $T_{a,b,s,t}$ be an integral semicompact operator on
$\mathcal{X}=\mathcal{B}(X,Y)$ in the conditions of Proposition $\ref{intsemi}%
$ or Theorem $\ref{intSemCom}$. Then all invariant subspaces of $T_{a,b,s,t}$
on which it is surjective consist of nuclear operators. In particular, each
solution $x$ of the equation
\[
T_{a,b,s,t}x=\lambda x
\]
where $\lambda\neq0$, is a nuclear operator.
\end{theorem}

We may apply previous results to matrix multiplication operators (see Section
\ref{s432}).

\begin{corollary}
\label{matr} Let a matrix $(T_{pq})_{p,q=1}^{n}$ consist of semicompact
multiplication operators and let $T$ be the matrix multiplication operator
defined by this matrix. Then the spectral subspaces of $T$ that correspond to
clopen subsets of $\sigma(T)$ non-containing $0$, consist of $n$-tuples of
nuclear operators.
\end{corollary}

\begin{proof}
Let $\mathcal{X}=\mathcal{B}(X,Y)^{\left(  n\right)  }$ (the direct sum of $n$
copies of $\mathcal{B}(X,Y)$), $\mathcal{Y}=\mathcal{N}(X,Y)^{\left(
n\right)  }$ and $U=\mathcal{B}(X,Y)$. Then it is easy to see that
\[
\mathbb{M}_{n}\left(  \widehat{\mathcal{K}}_{\frac{1}{2}}(U)\right)
\subset\mathbb{M}_{n}\left(  \widehat{\mathcal{B}}_{\ast}(U)\right)
\subset\mathcal{B}\left(  \mathcal{X}||\mathcal{Y}\right)  \text{ and
}\mathbb{M}_{n}\left(  \mathcal{F}_{\frac{1}{2}}(U)\right)  \subset
\mathcal{B}(\mathcal{X},\mathcal{Y}).
\]
Now a similar argument as in Theorem \ref{cond} shows that
\[
\mathbb{M}_{n}(\widehat{\mathcal{K}}_{\frac{1}{2}}(U))\subset Q_{\mathcal{B}%
(\mathcal{X},\mathcal{Y})}\left(  \mathcal{B}\left(  \mathcal{X}%
||\mathcal{Y}\right)  \right)  ,
\]
and it remains to apply Theorem \ref{spectrclos}.
\end{proof}

\subsection{Semicompact elementary operators\label{s6dob}}

Let $X,Y$ be Banach spaces. Assume now that $T$ is an elementary operator on
$\mathcal{B}(X,Y)$:
\[
T=\sum_{i=1}^{n}L_{a_{i}}R_{x_{i}}+\sum_{j=1}^{k}L_{y_{j}}R_{b_{j}},
\]
where all $x_{i}\in\mathcal{B}\left(  X\right)  $, $y_{j}\in\mathcal{B}(Y)$,
$a_{i}\in\mathcal{K}(Y)$, $b_{j}\in\mathcal{K}(X)$. According to our
terminology, $T$ is a \textit{semicompact elementary operator}.

Our aim is to show that the statement of Theorem \ref{space} in this case can
be considerably strengthened: invariant subspaces on which $T$ is surjective
are contained in the component $\mathfrak{J}\left(  X,Y\right)  $ of each
quasi-Banach operator ideal $\mathfrak{J}$. In this situation the approach
based on the tensor products of Banach algebras and the tensor spectral radius
theory is not directly applicable and for the proof that some power of $T$ is
close (in a proper sense) to a semifinite elementary operator, we use the
arguments based on the analysis of triangularizable sets of compact operators.

\subsubsection{Quasi-Banach operator ideals}

Recall that \textit{a quasinorm} on a linear space $\mathcal{L}$ is a map
$\left\Vert \cdot\right\Vert _{\mathcal{L}}:\mathcal{L}\rightarrow
\mathbb{R}^{+}$ satisfying the conditions
\begin{align}
\left\Vert x+y\right\Vert _{\mathcal{L}}  &  \leq t_{\mathcal{L}}(\left\Vert
x\right\Vert _{\mathcal{L}}+\left\Vert y\right\Vert _{\mathcal{L}})\text{ for
all }x,y\in\mathcal{L}\text{ and some }t_{\mathcal{L}}\geq1,\label{qn}\\
\left\Vert \lambda x\right\Vert _{\mathcal{L}}  &  =|\lambda|\left\Vert
x\right\Vert _{\mathcal{L}}\text{ for all }\lambda\in\mathbb{C},x\in
\mathcal{L},\text{ and}\nonumber\\
\left\Vert x\right\Vert _{\mathcal{L}}  &  =0\text{ iff }x=0.\nonumber
\end{align}
By \cite[Page 162]{K1}, each quasinorm generates a linear (metrizable)
Hausdorff topology on $\mathcal{L}$. We say that $\mathcal{L}$ is
\textit{complete under the quasinorm} if it is complete in this topology.

Furthermore, a \textit{quasi-Banach operator ideal} $\mathfrak{J}$ (see
\cite{P78}) consists of components $\mathfrak{J}\left(  X,Y\right)
\subset\mathcal{B}\left(  X,Y\right)  $ complete under a quasinorm $\left\Vert
\cdot\right\Vert _{\mathfrak{J}\left(  X,Y\right)  }=\left\Vert \cdot
\right\Vert _{\mathfrak{J}}$, where $X$ and $Y$ run over Banach spaces, and
satisfying the following conditions

1) $t_{\mathfrak{J}\left(  X,Y\right)  }=t_{\mathfrak{J}}$ for some
$t_{\mathfrak{J}}\geq1$ and all Banach spaces $X$ and $Y$, where
$t_{\mathfrak{J}\left(  X,Y\right)  }$ is the constant $t_{\mathcal{L}}$ in
(\ref{qn}) for $\mathcal{L}=\mathfrak{J}\left(  X,Y\right)  $.

2) $\left\Vert axb\right\Vert _{\mathfrak{J}}\leq\left\Vert a\right\Vert
\left\Vert x\right\Vert _{\mathfrak{J}}\left\Vert b\right\Vert $ for all
$x\in\mathfrak{J}\left(  X,Y\right)  $, $a\in\mathcal{B}\left(  Y,Z\right)
,b\in\mathcal{B}(W,X)$, where $Z$ and $W$ run over Banach spaces,

3) $\left\Vert x\right\Vert _{\mathfrak{J}}=||x||$ for each operator $x$ of
rank one.


By \cite[Theorem 6.2.5]{P78}, each quasi-Banach ideal $\mathfrak{J}$ has an
equivalent quasinorm $\left\vert \cdot\right\vert _{\mathfrak{J}}$ with the
property that there is a number $p$ such that $0<p\leq1$ and%
\begin{equation}
\left\vert x+y\right\vert _{\mathfrak{J}}^{p}\leq\left\vert x\right\vert
_{\mathfrak{J}}^{p}+\left\vert y\right\vert _{\mathfrak{J}}^{p} \label{add}%
\end{equation}
for every $x,y\in\mathfrak{J}\left(  X,Y\right)  $ and for all Banach spaces
$X,Y$ (one can take $p$ as a number satisfying $\left(  2t\right)  ^{p}=2$ for
$t\geq t_{\mathfrak{J}}$). We assume that a quasinorm in consideration
satisfies this condition, and write $\Vert\cdot\Vert_{p}$ or $\Vert\cdot
\Vert_{p,\mathfrak{J}}$ instead of $\left\vert \cdot\right\vert _{\mathfrak{J}%
}$. In this case we say that $\mathfrak{J}$ is a $p$-\textit{Banach operator
ideal}. It should be noted that the topology of $\mathfrak{J}\left(
X,Y\right)  $ is given by the metric $\mathrm{d}\left(  x,y\right)  =\Vert
x-y\Vert_{p}^{p}$. In the same way we denote the corresponding quasinorm on
bounded operators $T$ on $\mathfrak{J}\left(  X,Y\right)  $:
\[
\Vert T\Vert_{p}=\Vert T\Vert_{p,\mathfrak{J}}=\inf\left\{  t>0:\Vert
Tx\Vert_{p}\leq t\Vert x\Vert_{p}\text{ for all }x\in\mathfrak{J}\left(
X,Y\right)  \right\}  .
\]

\begin{lemma}
\label{p-norm} Let $\mathfrak{J}$ be a $p$-Banach operator ideal, and let $T$
be an elementary operator on $\mathcal{B}\left(  X,Y\right)  $, $Tx=\sum
_{i=1}^{n}a_{i}xb_{i}$ for every $x\in\mathcal{B}\left(  X,Y\right)  $, where
$a_{i}\in\mathcal{B}\left(  Y\right)  $, $b_{i}\in\mathcal{B}\left(  X\right)
$. Then $T$ is bounded on $\mathfrak{J}\left(  X,Y\right)  $ and
\[
\Vert Tx\Vert_{p}\leq n^{\frac{1-p}{p}}\left(  \sum_{i=1}^{n}\Vert a_{i}%
\Vert\Vert b_{i}\Vert\right)  \Vert x\Vert_{p}%
\]
for all $x\in\mathfrak{J}\left(  X,Y\right)  $.
\end{lemma}

\begin{proof}
It follows from (\ref{add}) under $\left\vert \cdot\right\vert _{\mathfrak{J}%
}=\Vert\cdot\Vert_{p}$ that
\[
\Vert Tx\Vert_{p}^{p}\leq\sum_{i=1}^{n}\Vert a_{i}\Vert^{p}\Vert b_{i}%
\Vert^{p}\Vert x\Vert_{p}^{p}%
\]
for all $x\in\mathfrak{J}\left(  X,Y\right)  $.
Since the function $f(t)=t^{p}$ is concave for $t\geq0$ and $0<p\leq1$, we
obtain that
\begin{equation}
\sum_{i=1}^{n}t_{i}^{p}\leq n^{1-p}\left(  \sum_{i=1}^{n}t_{i}\right)  ^{p}.
\label{tp}%
\end{equation}
Applying this to $t_{i}=\Vert a_{i}\Vert\Vert b_{i}\Vert$, we get that
\[
\Vert Tx\Vert_{p}^{p}\leq n^{1-p}\left(  \sum_{i=1}^{n}\Vert a_{i}\Vert\Vert
b_{i}\Vert\right)  ^{p}\Vert x\Vert_{p}^{p}%
\]
which gives what we need.
\end{proof}

In a short form the statement of the previous lemma can be written as
follows:
\[
\text{If }T=\sum_{i=1}^{n}L_{a_{i}}R_{b_{i}}\text{ then }\Vert T\Vert_{p}\leq
n^{\frac{1-p}{p}}\sum_{i=1}^{n}\Vert a_{i}\Vert\Vert b_{i}\Vert.
\]

Similarly we obtain the following

\begin{lemma}
\label{est}Let $a$ be a finite rank operator in $\mathcal{B}\left(
X,Y\right)  $. Then
\[
\left\Vert a\right\Vert _{p}\leq n^{\frac{1-p}{p}}\left\Vert a\right\Vert
_{\mathcal{N}\left(  X,Y\right)  },
\]
where $n$ is the rank of $a$ and $\left\Vert \cdot\right\Vert _{p}$ is the
$p$-norm of a $p$-Banach operator ideal $\mathfrak{J}$.
\end{lemma}

\begin{proof}
It is easy to check that for any $\varepsilon>0$ there are rank one operators
$a_{i}$ such that
\[
a=\sum_{i=1}^{n}a_{i},
\]
where $n$ is the rank of $a$, and
\[
\sum_{i=1}^{n}\left\Vert a_{i}\right\Vert \leq\left\Vert a\right\Vert
_{\mathcal{N}\left(  X,Y\right)  }+\varepsilon.
\]
On the other hand,
\[
\left\Vert a\right\Vert _{p}^{p}\leq\sum_{i=1}^{n}\left\Vert a_{i}\right\Vert
_{p}^{p}=\sum_{i=1}^{n}\left\Vert a_{i}\right\Vert ^{p}\leq n^{1-p}\left(
\sum_{i=1}^{n}\left\Vert a_{i}\right\Vert \right)  ^{p}=n^{1-p}\left(
\left\Vert a\right\Vert _{\mathcal{N}\left(  X,Y\right)  }+\varepsilon\right)
^{p}%
\]
by (\ref{tp}). As $\varepsilon$ is arbitrary, we obtain that $\left\Vert
a\right\Vert _{p}\leq n^{\frac{1-p}{p}}\left\Vert a\right\Vert _{\mathcal{N}%
\left(  X,Y\right)  }$.
\end{proof}

\subsubsection{Quasinilpotence of semicompact elementary operators modulo
semifinite ones with respect to a quasinorm}

If $W$ is a closed subspace of a Banach space $X$ then for each $x\in X$, we
will write $x/W$ instead of $x+W$ for the corresponding element of $X/W$. If
moreover $a$ is an operator on $X$ leaving $W$ invariant then we denote by
$a|_{W}$ and $a|_{X/W}$ its restriction to $W$ and, respectively, the operator
induced by $a$ in $X/W$.

\begin{lemma}
\label{subspace} Let $W$ be a closed subspace of a Banach space $X$, and let
$a,b$ be operators on $X$ which preserve $W$ invariant. Then
\begin{equation}
\Vert ab\Vert\leq2\Vert a|_{W}\Vert\Vert b\Vert+\Vert a\Vert\Vert
b|_{X/W}\Vert. \label{f80}%
\end{equation}

\end{lemma}

\begin{proof}
For any $x\in X$ and $\varepsilon>0$, choose $y\in W$ with $\Vert
bx-y\Vert\leq\Vert\left(  bx\right)  /W\Vert_{X/W}+\varepsilon$. Then
\[
\Vert y\Vert\leq\left\Vert bx\right\Vert +\Vert\left(  bx\right)
/W\Vert_{X/W}+\varepsilon\leq2\Vert bx\Vert+\varepsilon,
\]
whence we obtain that
\begin{align*}
\Vert abx\Vert &  \leq\Vert ay\Vert+\Vert a\Vert\Vert bx-y\Vert\leq\Vert
a|_{W}\Vert\Vert y\Vert+\Vert a\Vert(\Vert\left(  bx\right)  /W\Vert
_{X/W}+\varepsilon)\\
&  \leq\Vert a|_{W}\Vert(2\Vert bx\Vert+\varepsilon)+\Vert a\Vert(\Vert
b|_{X/W}\Vert\Vert x/W\Vert_{X/W}+\varepsilon).
\end{align*}
Since $\Vert x/W\Vert_{X/W}\leq\Vert x\Vert$ and $\varepsilon$ is arbitrary,
we obtain that
\[
\Vert abx\Vert\leq(2\Vert a|_{W}\Vert\Vert b\Vert+\Vert a\Vert\Vert
b|_{X/W}\Vert)\Vert x\Vert
\]
which is what we need.
\end{proof}

\begin{lemma}
\label{cepochka} Let $0=X_{0}\subset X_{1}\subset X_{2}\subset...\subset
X_{k}\subset X$ be a chain of closed subspaces in a Banach space $X$. Let
$m\geq k$ and let $a_{1},\ldots,a_{m}\in\mathcal{B}(X)$ preserve all $X_{j}$
invariant: $a_{i}X_{j}\subset X_{j}$. If $\Vert a_{i}\Vert\leq\alpha$ for all
$i$, and $\Vert a_{i}|_{X_{j}/X_{j-1}}\Vert\leq\beta$ for all $i,j$, then
\[
\Vert a_{1}a_{2}\cdots a_{m}\Vert\leq2^{m}C_{m}^{k}\alpha^{k}\beta^{m-k}.
\]

\end{lemma}

\begin{proof}
We use induction in $m,k$. For the base of the induction, note that the
statement is evidently true for $k=0$ and for $m=k$, and Lemma \ref{subspace}
establishes it for $k=1,m=2$. Now assuming that the statement holds for
$(m-1,k-1)$ and $(m-1,k)$ we prove that it holds for $(m,k)$. Indeed, setting
$W=X_{1}$, $a=a_{1}$, $b=a_{2}\cdots a_{m}$ in the notation of Lemma
\ref{subspace}, we obtain from (\ref{f80}) that
\[
\Vert a_{1}a_{2}\cdots a_{m}\Vert\leq2\Vert a_{1}|_{X_{1}}\Vert\Vert
a_{2}\cdots a_{m}\Vert+\Vert a_{1}\Vert\Vert a_{2}\cdots a_{m}|_{X/X_{1}}%
\Vert.
\]
By the induction assumption, we have that
\[
\Vert a_{2}\cdots a_{m}\Vert\leq2^{{m-1}}C_{m-1}^{k}\alpha^{k}\beta^{m-1-k}.
\]
Furthermore, the operators $a_{i}|_{X/X_{1}}$ preserve the chain
$\{X_{i}/X_{1}:i\leq k\}$ which consists of $k-1$ non-trivial elements. Hence
again by the induction assumption, we obtain that
\[
\Vert a_{2}\cdots a_{m}|_{X/X_{1}}\Vert\leq2^{m-1}C_{m-1}^{k-1}\alpha
^{k-1}\beta^{m-k}.
\]
Therefore
\begin{align*}
\Vert a_{1}a_{2}\cdots a_{m}\Vert &  \leq2\beta2^{{m-1}}C_{m-1}^{k}\alpha
^{k}\beta^{m-1-k}+\alpha2^{m-1}C_{m-1}^{k-1}\alpha^{k-1}\beta^{m-k}\\
&  \leq2^{m}\alpha^{k}\beta^{m-k}(C_{m-1}^{k}+C_{m-1}^{k-1})=2^{m}C_{m}%
^{k}\alpha^{k}\beta^{m-k}.
\end{align*}
\end{proof}

\begin{lemma}
\label{products-oper} Let $K$ be a finite set of compact operators in the
radical of an operator algebra $A\subset\mathcal{B}(X)$, and let $F$ be a
bounded subset of $A$. Let $\lambda\in(0,1)$. For each $m$, let $E_{K,F}%
^{\lambda}(m)$ denote the set of all products $b_{1}...b_{m}$ of elements in
$K\cup F$ in which the number of those $b_{i}$ that belong to $K$ is greater
than or equal to ${\lambda}m$. Then $\Vert E_{K,F}^{\lambda}(m)\Vert
^{1/m}\rightarrow0$ for $m\rightarrow\infty$.
\end{lemma}

\begin{proof}
Without loss of generality, one may assume that $\left\Vert K\cup F\right\Vert
=1$. By \cite{Sh84} (see also \cite{ST2000}), for each $\varepsilon>0$ there
is a finite chain $0\subset X_{1}\subset...\subset X_{k}\subset X$ of
invariant subspaces for $A$ such that $\Vert b|_{X_{j}/X_{j-1}}\Vert
\leq\varepsilon$ for all $b\in K$ and all $j\leq k$. It follows that
\[
\Vert c|_{X_{j}/X_{j-1}}\Vert\leq\varepsilon
\]
if $c=b_{1}...b_{p}a$, where $b_{i}\in F$ and $a\in K$.
Each product $b_{1}...b_{m}\in E_{K,F}^{\lambda}(m)$ can be written in the
form $c_{1}c_{2}...c_{l}$, where all of $c_{i}$ are as above and
$l\geq{\lambda}m$. Applying the result of Lemma \ref{cepochka}, we obtain
that
\[
\Vert b_{1}...b_{m}\Vert\leq2^{l}C_{l}^{k}\varepsilon^{l}\leq2^{m}%
m^{k}\varepsilon^{\lambda m}.
\]
Thus
\[
\Vert E_{K,F}^{\lambda}(m)\Vert^{1/m}\leq2m^{k/m}\varepsilon^{{\lambda}}%
\leq3\varepsilon^{{\lambda}}%
\]
for sufficiently big $m$. Since $\varepsilon$ is arbitrary, we are done.
\end{proof}

A subset $M$ of a Banach algebra $A$ is called \textit{bicompact} if
$L_{a}R_{b}$ is a compact operator for every $a,b\in M$.

\begin{proposition}
\label{products} Let $M$ be a finite bicompact subset of $A$ in the radical of
a Banach algebra $A$, and let $N$ be a bounded subset of $A$. For each $m$,
let $H(m)$ denote the set of all products $x_{1}...x_{m}$ of elements in
$M\cup N$ in which the number of those $x_{i}$ that belong to $M$ is greater
than or equal to $m/2$. Then $\Vert H(m)\Vert^{1/m}\rightarrow0$ under
$m\rightarrow\infty$.
\end{proposition}

\begin{proof}
We may assume that $A$ is unital and $1$ is the unit of $A$. Let $K$ be the
set of all operators $L_{a}R_{b}$ on $A$, where $a,b\in M$. Let $F=\{L_{a}%
:a\in N\}\cup\{R_{a}:a\in N\}$.
We claim that every product $w$ in $H(m)$ can be written as $T(1)$, where $T$
is a product of operators in which the number of operators in $K$ is greater
than or equal to $[m/4]$ (and the number of operators in $F$ is less than or
equal to $m/2+1$). Indeed, we do as follows. Represent $w$ as the product of
$w_{1}$ and $w_{2}$ in which of each the number of those $x_{i}$ that belong
to $M$ is greater than or equal to $[m/4]$. Let $a_{0}=1$, $w_{1}=w_{3}%
x_{i}v_{1}$ and $w_{2}=v_{2}x_{j}w_{4}$ for some $x_{i},x_{j}\in M$, where
$v_{1}$ and $v_{2}$ do not contain any elements from $M$ as a factor. Then
$w=w_{3}a_{1}w_{4}$, where $a_{1}=S_{1}L_{v_{1}}R_{v_{2}}\left(  a_{0}\right)
$ and $S_{1}=L_{x_{i}}R_{x_{j}}\in K$. Arguing by induction, we obtain that
$w=w_{2k+1}a_{k}w_{2k+2}$, where $a_{k}=S_{k}P_{k}\left(  a_{k-1}\right)  $,
$S_{k}\in K$ and $P_{k}$ is a product of operators in $F$, for $k\leq\lbrack
m/4]$. So we obtain the required representation $w$ as $T\left(  1\right)  $
for some $k\geq\lbrack m/4]$.
Now it follows in the notation of Lemma \ref{products-oper} that
\[
\Vert H(m)\Vert\leq\left\Vert E_{K,F}^{1/3}\left(  \frac{2m}{3}\right)
\right\Vert
\]
for sufficiently big $m$, and it remains to apply Lemma \ref{products-oper}.
\end{proof}

\begin{theorem}
\label{power} Let $T=\sum_{i=1}^{n}L_{a_{i}}R_{x_{i}}+\sum_{j=1}^{k}L_{y_{j}%
}R_{b_{j}}$ be a semicompact elementary operator on $\mathcal{B}(X,Y)$, where
all $x_{i}\in\mathcal{B}\left(  X\right)  $, $y_{j}\in\mathcal{B}(Y)$,
$a_{i}\in\mathcal{K}(Y)$, $b_{j}\in\mathcal{K}(X)$. Then for any
$\varepsilon>0$, there is $m\in\mathbb{N}$ and an operator $S=\sum
_{i=1}^{(n+k)^{m}}L_{c_{i}}R_{d_{i}}$ such that
\[
\Vert T^{m}-S\Vert_{p}<\varepsilon^{m}%
\]
and $c_{i}$ or $d_{i}$ is of finite rank for each $i$, where $\left\Vert
\cdot\right\Vert _{p}$ is the $p$-norm of a $p$-Banach operator ideal
$\mathfrak{J}$.
\end{theorem}

\begin{proof}
A required decomposition of the operator $T^{m}$ into the sum of $(n+k)^{m}$
summands can be written as
\[
T^{m}=T_{1}+T_{2},
\]
where in $T_{1}$ we gather those summands where the number of factors
$L_{a_{i}}$ is more than the number of factors $L_{y_{j}}$ (hence their number
$\geq m/2$), while summands in $T_{2}$ have more factors $R_{b_{j}}$ than
factors $R_{x_{i}}$.
Let $A=\mathcal{B}(Y)/\mathcal{A}(Y)$ and $q:\mathcal{B}(Y)\longrightarrow A$
be the standard epimorphism. Let $M=\{q(a_{1}),...,q(a_{n})\}$ and
$N=\{q(y_{1}),...,q(y_{k})\}$. Then $M$ is a bicompact subset of $A$ in the
radical of $A$. Writing $T_{1}$ as $\sum L_{w_{i}}R_{z_{i}}$, where $L_{w_{i}%
}R_{z_{i}}$ are the above summands of $T_{1}$ in the decomposition of $T^{m}$,
we note that the corresponding family $H(m)$ (see the above lemma) consists of
all products of elements which are $q$-images of left coefficients $w_{i}$ of
summands in $T_{1}$. By Lemma \ref{products}, there is $m$ such that
\[
\Vert H(m)\Vert<\varepsilon^{m}.
\]
This means that for every $w_{i}$ there is a finite rank operator $r_{i}$
with
\[
\Vert r_{i}-w_{i}\Vert<\varepsilon^{m}.
\]
Then, setting $S_{1}=\sum L_{r_{s}}R_{z_{s}}$, we obtain a semifinite
multiplication operator such that $S_{1}-T_{1}$ can be represented in the form
$\sum L_{v_{i}}R_{z_{i}}$ with
\[
\sum_{i}\Vert v_{i}\Vert\Vert z_{i}\Vert<(n+k)^{m}\varepsilon^{m}.
\]
For brevity, we rewrite this in the form
\[
\left\Vert S_{1}-T_{1}\right\Vert _{\ast}<(n+k)^{m}\varepsilon^{m}.
\]
Similarly, we find a semifinite operator $S_{2}=\sum L_{e_{s}}R_{f_{s}}$,
where all $f_{s}$ are finite rank operators, with
\[
\left\Vert S_{2}-T_{2}\right\Vert _{\ast}<(n+k)^{m}\varepsilon^{m}.
\]
Hence setting $S=S_{1}+S_{2}$ we obtain that $S$ is semifinite and
\[
\left\Vert T-S\right\Vert _{\ast}<2(n+k)^{m}\varepsilon^{m}.
\]
As the number of elementary summands (of length one) in $S$ is $(n+k)^{m}$ by
our choice then we obtain that
\[
\left\Vert T^{m}-S\right\Vert _{p}<((n+k)^{m})^{\frac{1-p}{p}}2(n+k)^{m}%
\varepsilon^{m}.
\]
by Lemma \ref{p-norm}. Changing $\varepsilon$ by $\gamma\varepsilon$ for
sufficiently small $\gamma$, we obtain the required inequality.
\end{proof}

\subsubsection{Spectral subspaces of semicompact elementary operators}

Now we are able to prove the following

\begin{theorem}
\label{quasinorm} Let $T=\sum_{i=1}^{n}L_{a_{i}}R_{x_{i}}+\sum_{j=1}%
^{k}L_{y_{j}}R_{b_{j}}$be a semicompact elementary operator on $\mathcal{B}%
(X,Y)$, where all $x_{i}\in\mathcal{B}\left(  X\right)  $, $y_{j}%
\in\mathcal{B}(Y)$, $a_{i}\in\mathcal{K}(Y)$, $b_{j}\in\mathcal{K}(X)$.
Suppose that $T\mathfrak{\mathcal{Z}}=\mathfrak{\mathcal{Z}}$ for a closed
subspace $\mathfrak{\mathcal{Z}}$ of $\mathcal{B}(X,Y)$. Then
$\mathfrak{\mathcal{Z}}$ is contained in $\mathfrak{J}\left(  X,Y\right)  $
for any quasi-Banach operator ideal $\mathfrak{J}$.
\end{theorem}

\begin{proof}
One may suppose that $\mathfrak{J}$ is a $p$-Banach operator ideal with
$p$-norm $\left\Vert \cdot\right\Vert _{p}$ for $0<p\leq1$.
By the Open Mapping Theorem, there is $t>0$ such that for each $z\in
\mathcal{Z}$ there is $w\in\mathcal{Z}$ with $Tw=z$ and $\Vert w\Vert\leq
t\Vert z\Vert$. Take $\varepsilon>0$ such that $\varepsilon<t^{-1}$, and
choose $m$ and $S$ as in Theorem \ref{power}. Setting $P=T^{m}-S$, we have
that $\Vert P\Vert_{p}\leq\varepsilon^{m}$.
On the other hand, $S$ maps each operator from $\mathcal{B}\left(  X,Y\right)
$ into an operator of rank $\leq d$, where $d$ is the sum of the ranks of
finite rank coefficients of $S$. As $\mathfrak{J}$ contains all finite rank
operators, we have that $Sx\in\mathfrak{J}\left(  X,Y\right)  $ and
\[
\Vert Sx\Vert_{p}\leq d^{\frac{1-p}{p}}\Vert Sx\Vert_{\mathcal{N}\left(
X,Y\right)  }%
\]
by Lemma \ref{est}, for every $x\in\mathcal{B}(X,Y)$. It follows from Theorem
\ref{pair2} that $S$ is bounded as an operator $\mathcal{B}\left(  X,Y\right)
\longrightarrow\mathcal{N}\left(  X,Y\right)  $. Let $s$ be its norm as such
an operator. Then
\[
\left\Vert Sx\right\Vert _{\mathcal{N}\left(  X,Y\right)  }\leq s\left\Vert
x\right\Vert
\]
for every $x\in\mathcal{B}(X,Y)$. As a result, we obtain that
\begin{equation}
\Vert Sx\Vert_{p}\leq d^{\frac{1-p}{p}}s\left\Vert x\right\Vert \label{est1}%
\end{equation}
for all $x\in\mathcal{B}(X,Y)$.
Now we may argue as in the proof of Theorem \ref{spectrclos}. Then, as we saw,
each $z\in\mathfrak{\mathcal{Z}}$ can be expressed as in (\ref{ser}):
\[
z=\sum_{j=0}^{\infty}P^{j}Sz_{j+1}%
\]
with the estimation
\begin{equation}
\left\Vert z_{j}\right\Vert \leq t^{mj}\left\Vert z\right\Vert \label{est2}%
\end{equation}
for every $j$ (see (\ref{ryj}) in Theorem \ref{spectrclos}). As all
$P^{j}Sz_{j+1}\in\mathfrak{J}\left(  X,Y\right)  $, we may estimate their
$p$-norms $\Vert P^{j}Sz_{j+1}\Vert_{p}$ as follows. We have that
\[
\Vert P^{j}Sz_{j+1}\Vert_{p}\leq\Vert P\Vert_{p}^{j}\Vert Sz_{j+1}\Vert_{p}%
\]
and that
\[
\Vert Sz_{j+1}\Vert_{p}\leq d^{\frac{1-p}{p}}s\Vert z_{j+1}\Vert\leq
d^{\frac{1-p}{p}}st^{m\left(  j+1\right)  }\Vert z\Vert
\]
by (\ref{est1}) and (\ref{est2}). By our choice, we have that $\Vert
P\Vert_{p}\leq\varepsilon^{m}$. So we obtain that
\[
\Vert P^{j}Sz_{j+1}\Vert_{p}\leq\varepsilon^{mj}d^{\frac{1-p}{p}}st^{m\left(
j+1\right)  }\Vert z\Vert=\left(  d^{\frac{1-p}{p}}st^{m}\Vert z\Vert\right)
(\varepsilon t)^{mj},
\]
whence
\[
\sum_{j=0}^{\infty}\Vert P^{j}Sz_{j+1}\Vert_{p}^{p}\leq\left(  d^{\frac
{1-p}{p}}st^{m}\Vert z\Vert\right)  ^{p}\sum_{j=0}^{\infty}\left(
(\varepsilon t)^{mp}\right)  ^{j}<\infty
\]
because of $(\varepsilon t)^{mp}<1$. As $\mathfrak{J}\left(  X,Y\right)  $ is
complete under $\left\Vert \cdot\right\Vert _{p}$, the convergence of this
series implies that $z\in\mathfrak{J}\left(  X,Y\right)  $.
\end{proof}

As a consequence, we obtain the following

\begin{corollary}
Let $T$ be a semicompact elementary operator on $\mathcal{B}\left(
X,Y\right)  $. Then all eigenspaces of $T$ corresponding to non-zero
eigenvalues and all spectral subspaces of $T$ corresponding to clopen subsets
of $\sigma\left(  T\right)  $ non-containing $0$ are contained in
$\mathfrak{J}\left(  X,Y\right)  $ for any quasi-Banach operator ideal
$\mathfrak{J}$.
\end{corollary}


\begin{thebibliography}{99}                                                                                               %


\bibitem {A68}J. C. Alexander, Compact Banach algebras, \textit{Proc. London
Math. Soc. }(3) \textbf{18} (1968) 1--18.

\bibitem {Aup}B. Aupetit, \textit{Propri\'{e}t\'{e}s spectrales des
alg\`{e}bres de Banach}, Lect. Notes Math., 735, Berlin, Springer-Verlag, 1979.

\bibitem {Aup-Ma}B. Aupetit, M. Mathieu, The continuity of Lie homomorphisms,
\textit{Stud. Math. } \textbf{138} (2) (2000) 193-199.

\bibitem {Bar1}B. A. Barnes, Banach algebras which are ideals of a Banach
algebra, \textit{Pacific J. Math.} \textbf{38}, 1 (1971) 1--7.

\bibitem {Barnes}B. A. Barnes, Density theorems for algebras of operators and
annihilator Banach algebras, \textit{Mich. Math. J.} \textbf{19} (1972), 149
-- 155.

\bibitem {BL76}J. Bergh and J. L\"{o}fstr\"{o}m, \textit{Interpolation
spaces}, \textit{An introduction}, Springer-Verlag, Berlin, 1976.

\bibitem {BW}M. A. Berger and Y. Wang, Bounded semigroups of matrices,
\textit{Linear algebra Appl.} 166 (1992) 21--27.

\bibitem {Bons}F. F. Bonsall, Operators that act compactly on an algebra of
operators, \textit{ Bull. London Math. Soc.} \textbf{1} (2), (1969) 163 -- 170.

\bibitem {BD}F. F. Bonsall, J. Duncan, \textit{Complete normed algebras},
Springer-Verlag, Berlin, Heidelberg, New York, 1973.

\bibitem {B98}N. Bourbaki, \textit{Elements of mathematics, Algebra 1,
Chapters 1-3}, Springer-Verlag, Berlin, 1998.

\bibitem {Curto}R. E. Curto, Spectral theory of elementary operators, In
\textquotedblleft\textit{Elementary operators and applications}%
\textquotedblright, Editor M. Mathieu, World Sci. Publ., Singapour, New
Jersey, London, 1992, 3-54.

\bibitem {DF93}A. Defant, K. Floret, \textit{Tensor norms and operator
ideals}, Elsevier, Amsterdam, 1993.

\bibitem {Dix}P. G. Dixon, Topologically irreducible representations and
radicals in Banach algebras, \textit{Proc. London Math. Soc.} (3) \textbf{74}
(1997) 174--200.

\bibitem {FR}C-K. Fong, H. Radjavi, On ideals and Lie ideals of compact
operators, \textit{Math. Ann.} \textbf{262} (1983) 23--28.

\bibitem {HP}E. Hille, R. Fillips, \textit{Functional analysis and
semigroups}, AMS, Providence, 1957.

\bibitem {KST}M. Kennedy, V. S. Shulman, Yu. V. Turovskii, Invariant subspaces
of subgraded Lie algebras of compact operators, \textit{Integr. Eq. Oper.
Theory} \textbf{63} (2009) 47--93.

\bibitem {K1}G. K\"{o}the, \textit{Topological vector spaces I, }%
Spinger-Verlag, New York, 1969.\textit{ }

\bibitem {Lom}V. I. Lomonosov, Invariant subspaces for operators commuting
with compact operators, \textit{Funct. Anal. Appl.} \textbf{7} (1973), 213--214.

\bibitem {Read}C. J. Read, Quasinilpotent operators and the invariant subspace
problem, \textit{J. London Math. Soc.} (2) \textbf{56} (1997) 595-606

\bibitem {P78}A. Pietsch, \textit{Operator ideals}, Veb Deutscher Verlag der
Wissenschaften, Berlin, 1978.

\bibitem {RS}G.-C. Rota and W. G. Strang, A note on the joint spectral radius,
\textit{Indag. Math.} \textbf{22} (1960) 379--381.

\bibitem {Rud}W. Rudin, \textit{Functional analysis}, McGraw-Hill, New York, 1991.

\bibitem {Ruston}A. F. Ruston, \textit{Fredholm theory in Banach spaces},
Cambridge Univ. Press, Cambridge, 1986.

\bibitem {Sh84}V. S. Shulman, On invariant subspaces of Volterra operators,
\textit{Functional Anal. i Prilozen.} \textbf{18} (2) (1984) 84--85 (In Russian).

\bibitem {ST2000}V. S. Shulman, Yu. V. Turovskii, Joint spectral radius,
operator semigroups and a problem of W. Wojty\'{n}ski, \textit{J. Funct.
Anal.} \textbf{177} (2000) 383--441.

\bibitem {STFaa}V. S. Shulman, Yu. V. Turovskii, Radicals in Banach algebras
and some problems in the theory of radical Banach algebras, \textit{ Funct.
Anal. Appl.} \textbf{35} (2001) 312--314.

\bibitem {ST2001-S}V. S. Shulman, Yu. V. Turovskii, Formulae of joint spectral
radii for sets of operators, \textit{Stud. Math.} \textbf{149} (2002) 23-37.

\bibitem {ST2005}V. S. Shulman, Yu. V. Turovskii, Invariant subspaces of
operator Lie algebras and Lie algebras with compact adjoint action, \textit{J.
Funct. Anal.} \textbf{223} (2005) 425--508.

\bibitem {TR1}V. S. Shulman, Yu. V. Turovskii, Topological radicals, I. Basic
properties, tensor products and joint quasinilpotence\textit{,} \textit{Banach
Center Publ.}, 67 (2005) 293--333.

\bibitem {GBWF}V. S. Shulman, Yu. V. Turovskii, Application of topological
radicals to calculation of joint spectral radii, \textit{preprint}:
arXiv:0805.0209 (2 May 2008).

\bibitem {T1998}Yu. V. Turovskii, Volterra semigroups have invariant
subspaces, \textit{J. Funct. Anal.} \textbf{162} (2) (1999), 313--323.

\bibitem {Vala}K. Vala, On compact sets of compact operators, \textit{Ann.
Acad. Sci. Fenn.} Ser. A I \textbf{351} (1964) 1--8.

\bibitem {Ves}E. Vesentini, On the subharmonicity of the spectral radius,
\textit{Boll. Un. Mat. Ital. } \textbf{4} (1968) 427-429.

\bibitem {Wil}G. Willis, Compact approximation property does not imply
approximation property, \textit{Stud. Math. } \textbf{103} (1992) 99-108.

\bibitem {W78-2}W. Wojtynski, A note on compact Banach-Lie algebras of
Volterra type, Bull. Acad. Polon. Sci., Ser.sci. math.,astr.et phys 26 (1978),
2, 105--107.

\bibitem {W78}W. Wojtynski, On the existence of closed two-sided ideals in
radical Banach algebras with compact elements, Bull. Acad. Polon. Sci.,
Ser.sci. math.,astr.et phys 26 (1978), 2, 109--113.

\bibitem {WWQ}W. Wojtynski, Quasinilpotent Banach-Lie algebras are
Baker-Campbell-Hausdorff, \textit{J. Funct. Anal.} \textbf{153} (1998), 405-413.

\bibitem {Zelmanov}E. I. Zelmanov, On Engel Lie algebras, \textit{DAN USSR},
\textbf{292} (1987) no. 2, 265-268 (In Russian).
\end{thebibliography}
\end{document}